\documentclass[amssymb,11pt]{article}
\usepackage{amsfonts}
\usepackage[colorlinks,linkcolor=blue,citecolor=cyan,urlcolor=black,hyperindex,CJKbookmarks]{hyperref}
\usepackage[dvipdfm]{graphicx}
\usepackage{amsthm}
\usepackage{amsmath}
\usepackage{multirow}
\usepackage{indentfirst}
\usepackage{booktabs}
\usepackage{bm}
\usepackage{subfigure}
\usepackage{algorithm,algorithmic}
\usepackage{cases}
 \newcommand{\h}[1]{\mathbf{#1}}

\date{}
\textwidth=16cm \textheight=25.8cm \vspace{1.4cm}

\newtheorem{definition}{Definition}[section]

\newtheorem{theorem}{Theorem}[section]

\renewcommand{\thefootnote}{\fnsymbol{footnote}}

\begin{document}

\title{\Large Low-rankness and Smoothness Meet Subspace: A Unified Tensor Regularization for Hyperspectral Image Super-resolution}
\begin{upshape}
\author{Jun Zhang$^{1,2}$, Chao Yi$^{1}$, Mingxi Ma$^{1}$, Mengling He$^{1}$, Chao Wang$^{3,\ast}$}
\end{upshape}\maketitle

\begin{center}
\indent\indent \textsl{\ \scriptsize \ \ \ \ \ \
$^1$College of Science, Jiangxi University of Water Resources and Electric Power, Nanchang 330099, Jiangxi, China\\
$^2$Key Laboratory of Engineering Mathematics and Advanced Computing, Jiangxi University of Water Resources and Electric Power, Nanchang 330099, Jiangxi, China\\
$^3$Department of Statistics and Data Science, Southern University of Science and Technology, Shenzhen 518055, Guangdong, China\\
\indent\indent}
\end{center}
\renewcommand{\thefootnote}{\fnsymbol{footnote}}
\footnotetext{
\ $^{\ast}$\scriptsize Corresponding author.
\\
\textsl{Email addresses: }junzhang0805@126.com (J. Zhang), chaoyi1002@126.com (C. Yi), mingxima@126.com (M. Ma), heml0816@163.com (M. He), wangc6@sustech.edu.cn (C. Wang).
 \indent\indent\small }

\textbf{Abstract.}\
Hyperspectral image super-resolution (HSI-SR) has emerged as a challenging yet critical problem in remote sensing. Existing approaches primarily focus on regularization techniques that leverage low-rankness and local smoothness priors. Recently, correlated total variation has been introduced for tensor recovery, integrating these priors into a single regularization framework. Direct application to HSI-SR, however, is hindered by the high spectral dimensionality of hyperspectral data.
In this paper, we propose a unified tensor regularizer, called JLRST, which jointly encodes low-rankness and local smoothness priors under a subspace framework. Specifically, we compute the gradients of the clustered coefficient tensors along all three tensor modes to fully exploit spectral correlations and nonlocal similarities in HSI. By enforcing priors on subspace coefficients rather than the entire high resolution HSI data, the proposed method achieves improved computational efficiency and accuracy. Furthermore, to mitigate the bias introduced by the tensor nuclear norm (TNN), we introduce the mode-3 logarithmic TNN to process gradient tensors. An alternating direction method of multipliers with proven convergence is developed to solve the proposed model. Experimental results demonstrate that our approach significantly outperforms state-of-the-art model-based methods in HSI-SR.

\begin{flushleft}
\end{flushleft}


\section{Introduction}

Hyperspectral imaging technology has become a research hotspot in remote sensing due to its ability to capture continuous spectral information of ground targets in the visible to infrared wavelength range.
Hyperspectral images (HSIs) typically contain dozens to hundreds of spectral bands, and their rich spectral information provides significant advantages in applications such as face recognition \cite{HFRW}, target detection \cite{MLBH}, and land cover classification \cite{HICV}.
Despite the rapid development of hyperspectral imaging technology, the acquired HSIs typically have lower spatial resolution due to the energy limitations of hyperspectral sensors.
In contrast, multispectral sensors capture images with higher spatial but lower spectral resolution \cite{AACS}. Thus, fusing multispectral image (MSI) and HSI from the same scene has become a widely used approach to obtain a high-resolution hyperspectral image (HR-HSI) \cite{HISRM,HAMIFW}, a process known as hyperspectral image super-resolution (HSI-SR).

\subsection{Related Works}
In the past decade, various HSI-SR approaches have emerged, generally classified into three categories: matrix factorization-based methods \cite{ACFF,TMR}, tensor factorization-based methods \cite{HSAC,FSTRD}, and deep learning-based methods \cite{NARH,PPSA,deep4,deep5,deep1,deep2,deep3,CLoRF}.
Deep learning-based methods can adaptively extract features in a data-driven manner and effectively model the complex nonlinear relationship between spectral and spatial information, thereby achieving superior fusion performance. However, these methods typically rely on large-scale training data and are prone to decreased generalization capability when data is insufficient. Additionally, their decision-making process lacks interpretability.
This research focuses on model-driven HSI-SR methods, which explicitly rely on physical observation models and interpretable priors.

In matrix factorization-based approaches for HSI-SR, every spectral vector is representable as a linear combination of multiple spectral characteristics.
Consequently, the estimation of target HR-HSI is transformed into determining both spectral bases and their corresponding coefficients. However, accurate estimation of these components necessitates the incorporation of additional prior knowledge.

Initially, sparsity was extensively applied in the field of HSI-SR.
Kawakami et al. \cite{HHIV} pioneered the introduction of matrix factorization into the fusion problem, employing the sparse prior to learn spectral bases from LR-HSI and subsequently estimating the corresponding coefficients from HR-RGB.
Building on this, Huang et al. \cite{SASI} utilized the K-SVD algorithm to learn spectral bases from LR-HSI. However, since this method independently sparsely encodes the spatial pixels of the target HSI without accounting for spatial correlation, the fusion performance is not outstanding.
Therefore, Dong et al. \cite{HISR} proposed a non-negative structural sparse representation method that fully utilizes both spatial and spectral correlations.
In addition, Han et al. \cite{SSCS} presented an HSI-SR method based on self-similarity constrained sparse representation, which mitigates the impact of outliers in sparse coding learning.
In \cite{SSSL}, Xue et al. proposed an HSI-SR approach based on structured sparse low-rank representation that effectively exploits spatial and spectral low-rank priors. Moreover, Chen et al. \cite{FOHM} introduced a novel HSI-MSI fusion approach that integrates a spatial-spectral dual dictionary with a structured sparse low-rank representation, effectively utilizing the spatial and spectral information of HSI.

In recent years, total variation (TV) regularization has been widely studied in the field of image processing owing to its remarkable ability to preserve edge information \cite{ACFF,HISRVN,HIFW,TVRM,PTCW}. For the HSI-SR problem,
Simões et al. \cite{ACFF} proposed a subspace representation-based approach, where TV regularization was used to enhance spatial smoothness.
In \cite{HISRVN}, an HSI-SR method based on nonlocal low-rank tensor approximation and TV regularization was presented, which considered the nonlocal self-similarity and the local smooth structure.
Subsequently, in our previous work \cite{HIFW}, a hybrid regularization method for HSI-MSI fusion was proposed, where TV regularization was combined with sparse prior and superpixel-based nuclear norm.
However, the conventional TV model employs uniform weighting for the two sub-variables in the gradient operator, resulting in a poor coupling effect with local features of the image \cite{IDBOT,PIRU}.
To address this issue, the adaptive TV (ATV) regularization was extensively investigated in HSI-SR \cite{FSTRD,LTNN,HSFU}.
Particularly, a new HSI-MSI fusion method was developed that incorporates the ATV regularization and the superpixel-based weighted nuclear norm (WNN) \cite{HSFU}. Through the incorporation of ATV, this approach achieves superior handling of image local features. Besides,
ATV regularization has also been successfully applied to tensor completion \cite{ATVAS}, and HSI denoising \cite{HIDWW}.
These approaches not only exploit the piecewise smoothness of images but also incorporate their intrinsic low-rank property.

On the other hand, given the inherent low-rank property of HSIs, numerous HSI-SR approaches that apply low-rank constraints have been proposed. Among these, nuclear norm minimization (NNM) has emerged as a predominant approach, since it provides the tightest convex relaxation for low-rank matrix approximation \cite{RSSB,LADM}.
However, since NNM treats all singular values equally during the soft-thresholding process, it overlooks the fact that larger singular values contain more important information than smaller ones. To this end, Zhang et al. \cite{WNNM} presented the WNN for HSI-SR, which assigns distinct weights to singular values, ensuring that larger singular values undergo less shrinkage and thus enhancing flexibility.
In contrast to the matrix nuclear norm, the tensor nuclear norm (TNN) demonstrates superior capability in preserving the intrinsic low-rank structure of tensor data \cite{HMIF,HISRVA,CTBT,GTNN}. In \cite{HISRVA}, an HSI-SR method based on factor group sparsity regularized subspace representation was developed. Specifically, by introducing TNN regularization to constrain the low-dimensional coefficient subspace, this approach effectively captures spatial self-similarity in images.
Similarly,
TNN uniformly weights all singular values, disregarding their different physical interpretations and consequently producing a biased low-rank tensor approximation.
In nuclear norm approximation, the sum of logarithmic singular values exhibits superior performance over the sum of singular values. Therefore, Dian et al. \cite{TMR} presented an HSI-SR method using subspace-based logarithmic tensor nuclear norm (LTNN) regularization. Additionally, under the tensor ring (TR) decomposition framework, an HSI-MSI fusion method using the LTNN regularization was proposed \cite{LTNN}. This approach leverages the mode-2 LTNN on three TR factors to more effectively maintain latent low-rank structures.

To enhance the spatial smoothness of images, TNN is often integrated with TV regularization in many applications
\cite{LTNN,TNNB,RLRT}.
In \cite{TNNB}, a model for color and multispectral image denoising was proposed by combining TNN with TV regularization. This method employs the TNN to characterize the low-rank structure of underlying data, while utilizing two distinct TV regularizations to maintain image local details. In \cite{RLRT},  a robust tensor completion approach via transformed TNN and TV regularization was presented.
However, these approaches encode low-rank and local smooth priors as the sum of two independent regularization terms, and their performance is critically dependent on the balancing parameter.
To address this issue, for tensor recovery, Wang et al. \cite{CTV} presented an innovative regularizer termed as tensor correlated total variation (t-CTV).
Specifically, this regularizer requires gradient computation along each tensor mode, followed by TNN evaluation of the resulting gradient tensors along corresponding dimensions. Additionally,
Wang et al. \cite{NMSCTV} proposed a non-convex mode-shuffled tensor correlated total variation (NMS-t-CTV) for HSI-SR. This method modifies t-CTV via a mode-shuffle strategy, converting the convex approximation into a non-convex one, thereby achieving compact and tight representations of the multilevel and multi-dimensional structural correlations inherent in HSI.

\subsection{Research Motivations}
To fully exploit spectral correlations in HSI and enhance computational
efficiency \cite{sub1,sub2}, we focus on subspace-based HSI-SR methods.
In addition,
although TR decomposition has a strong ability to explore intrinsic data structures, it often leads to substantial computational costs. This challenge can be mitigated by projecting the original HR-HSI into a low-dimensional subspace.

Inspired by the proven advantages of prior characterization through a unified regularizer,
we propose a subspace-based HSI-SR approach utilizing joint low-rank and smooth tensor (JLRST) regularization. Although t-CTV imposes the TNN constraint on the background tensor in the gradient domain,
the convex TNN treats each singular value equally and ignores the importance of larger singular values, which may result in a biased low-rank approximation.
To alleviate this issue, our proposed regularizer uses mode-3 LTNN in the gradient domain.
Specifically, we decompose the target HR-HSI into the spectral basis and its corresponding coefficients. The spectral basis is first obtained by performing singular value decomposition on LR-HSI.
Since HR-HSI often exhibits self-similarities, it naturally contains many similar patches. To this end, we group the patches in coefficients in light of the clustering structure learned from the HR-MSI.
Subsequently, patches from the same group are collected into a 3D tensor. Furthermore, we impose the JLRST regularization on these tensors, where the LTNN along the third mode is applied to the related gradient tensors. Finally, the constructed model is solved by the alternating direction method of multipliers (ADMM).

\subsection{Principal Contributions}
The principal contributions of this work are outlined as follows:
\begin{itemize}
    \item A novel tensor regularizer named JLRST is proposed to estimate the coefficient tensor instead of directly processing HR-HSI. To our knowledge, this is the first work in subspace-based HSI-SR to encode both global low-rank and local smooth priors of a tensor simultaneously.
    \item To produce a more precise low-rank approximation, we employ the mode-3 LTNN to constrain the gradient tensors rather than the TNN.
    \item
    We establish the corresponding convergence theory for the proposed algorithm.
    Experimental evaluations across four datasets illustrate the efficacy of the proposed approach, with additional convergence analysis validating the algorithm's stability.
\end{itemize}

The remainder of this paper is structured as follows. In Section \ref{section2}, we introduce symbols and provide some definitions related to tensor operations. In Section \ref{section3}, we propose a new subspace-based HSI-SR model via joint low-rank and smooth tensor regularization and present its optimization algorithm with fine convergence. The experimental results, parameter analysis, and convergence analysis are given in Section \ref{section4}. Finally, we conclude this article in Section \ref{section5}.

\section{Preliminaries}
\label{section2}
This section is devoted to introducing some notations and definitions related to tensors.

In this paper, scalars are signified by lowercase and uppercase letters, i.e., $m,\ M\in\mathbb{R}$. Vectors are represented by bold lowercase letters, i.e., $\mathbf{b}\in\mathbb{R}^{N}$. We denote matrices by bold uppercase letters, i.e., $\mathbf{X}\in\mathbb{R}^{M\times{N}}$. A three-order tensor is expressed with calligraphic letter, i.e., $\mathcal{X}\in\mathbb{R}^{I_1\times{I_2}\times{I_3}}$, whose mode-$n$ unfolding matrices $\mathbf{X}_{(1)}\in\mathbb{R}^{I_1\times{I_2I_3}}$, $\mathbf{X}_{(2)}\in\mathbb{R}^{I_2\times{I_1I_3}}$ and $\mathbf{X}_{(3)}\in\mathbb{R}^{I_3\times{I_1I_2}}$ are constructed by respectively arranging the first, second and third mode vectors of tensor $\mathcal{X}$ in columns. $\mathcal{X}(i,:,:)$, $\mathcal{X}(:,i,:)$ and $\mathcal{X}(:,:,i)$ represent the $i$-th horizontal, lateral, and frontal slices of $\mathcal{X}$, respectively. By performing the FFT along the third mode of $\mathcal{X}$ yields $\overline{\mathcal{X}}\in\mathbb{C}^{I_1\times I_2\times I_3}$. In MATLAB, the FFT is executed using the command $\mathbf{fft}$, so we can obtain
$\overline{\mathcal{X}} = \mathbf{fft}(\mathcal{X},[\ ], 3 )$. Additionally, $\mathcal{X}$ can be recovered from $\overline{\mathcal{X}}$ by applying the inverse FFT along the third mode, i.e., $\mathcal{X} = \mathbf{ifft}(\overline{\mathcal{X}}, [\ ], 3)$.

\begin{definition}(Tensor mode-n product with a matrix \cite{TMR}): The mode‑$n$ product between a tensor $\mathcal{A}\in\mathbb{R}^{I_1\times{I_2}\times{\cdots}\times{I_N}}$ and a matrix $\mathbf{B}\in\mathbb{R}^{K\times I_n}$  is defined as follows:
\begin{equation}
\mathcal{X}=\mathcal{A}\times_n\mathbf{B}\in\mathbb{R}^{I_1\times {\cdots}\times{I_{n-1}}\times K\times I_{n+1} \times{\cdots}\times{I_N}}.
\label{formula T-Mproduct}
\end{equation}
\end{definition}

By applying tensor matricization, equation $(\ref{formula T-Mproduct})$
can be reformulated in an unfolded form as $\mathbf{X_{(n)}}= \mathbf{B}\mathbf{A}_{(n)}$.

\begin{definition}(t-SVD \cite{TMR}): For a given tensor  $\mathcal{X}\in\mathbb{R}^{I_1\times{I_2}\times{I_3}}$, the t-SVD of $\mathcal{X}$ is
\begin{equation}
\mathcal{X}=\mathcal{U}\ast\mathcal{S}\ast\mathcal{V}^{T}
\label{formula t-svd}
\end{equation}
where $\mathcal{U}\in\mathbb{R}^{I_1\times{I_1}\times{I_3}}$ and $\mathcal{V}\in\mathbb{R}^{I_2\times{I_2}\times{I_3}}$ are orthogonal tensors, and $\mathcal{S}\in\mathbb{R}^{I_1\times{I_2}\times{I_3}}$ is an f-diagonal tensor.
\end{definition}

\begin{definition}(Gradient tensor \cite{CTV}): For $\mathcal{X}\in\mathbb{R}^{I_1\times{I_2}\times{\cdots}\times{I_N}}$, its gradient tensor along the $n$-th mode is expressed as
\begin{equation}
\nabla_{n}\mathcal{(X)}= \mathcal{X}\times_n\mathbf{D}_{I_n},\ n=1,2,\cdots,N
\label{formula gradient tensor}
\end{equation}
where $\mathbf{D}_{I_n}$ is a row circulant matrix of $(-1,1,0,\cdots,0)$.
\end{definition}

\section{Proposed Method}
\label{section3}

\subsection{Observation Model}
The target HR-HSI is represented by $\mathcal{Z}\in\mathbb{R}^{W\times{H}\times{S}}$, where $W$ and $H$ correspond to the spatial dimensions and $S$ stands for the spectral dimension.

The observed LR-HSI is indicated by $\mathcal{X}\in\mathbb{R}^{w\times{h}\times{S}}$,
which has $w \times h$ pixels and $S$ bands. $\mathcal{X}$ can be considered as the spatially downsampled version of $\mathcal{Z}$. Therefore, $\mathcal{X}$ can be represented as follows:
\begin{equation}
\begin{aligned}
\mathbf{X}_{(3)}=\mathbf{Z}_{(3)}\mathbf{B}\mathbf{S}
\label{formula7}
\end{aligned}
\end{equation}
where $\mathbf{B}\in\mathbb{R}^{WH\times WH}$ denotes the point spread functional, which is assumed to be consistent across all spectral bands. The matrix $\mathbf{S}\in\mathbb{R}^{WH\times {wh}}$
is used for spatial downsampling.

$\mathcal{Y}\in\mathbb{R}^{W\times{H}\times{s}}$ stands for the observed HR-MSI from the same scene.
$\mathcal{Y}$ can be regarded as the spectrally downsampled version of $\mathcal{Z}$, i.e.,
\begin{equation}
\begin{aligned}
\mathbf{Y}_{(3)}=\mathbf{F}\mathbf{Z}_{(3)}
\label{formula8}
\end{aligned}
\end{equation}
where $\mathbf{F}\in\mathbb{R}^{s\times S}$ is a matrix representing the spectral response function of the multispectral sensor.

\subsection{Subspace Learning}
HSI typically has a strong correlation in bands, and spectral vectors reside in a low-dimensional subspace  \cite{FHID}. Thus, the HR-HSI could be factorized as
\begin{equation}
\begin{aligned}
\mathcal{Z}=\mathcal{C}\times_3\mathbf{R}
\label{formula9}
\end{aligned}
\end{equation}
where $\mathcal{C}\in\mathbb{R}^{W\times H\times L }$ is a coefficient tensor. The subspace $\mathbf{R}\in\mathbb{R}^{S\times L}$ ($L< S$) represents a spectral basis. The condition $L<S$ signifies that the spectral vector exists in a low-dimensional subspace,
thereby enhancing the computational efficiency through dimensionality reduction. Due to the orthogonality of the columns in the subspace, each spectral vector $\mathcal{Z}(i, j, :)$
possesses the Frobenius norm that is identical to  $\mathcal{C}(i, j, :)$. In this manner, the self-similarities present in the image domain are mapped onto the domain of subspace coefficients. By applying tensor matricization, we obtain the following unfolded form:
\begin{equation}
\begin{aligned}
\mathbf{Z}_{(3)}=\mathbf{R}\mathbf{C}_{(3)}.
\label{formula10}
\end{aligned}
\end{equation}
Therefore, equations  $(\ref{formula7})$ and  $(\ref{formula8})$ can be rephrased as
\begin{equation}
\begin{aligned}
\mathbf{X}_{(3)}=\mathbf{R}\mathbf{C}_{(3)}\mathbf{B}\mathbf{S}
\label{formula11}
\end{aligned}
\end{equation}
\begin{equation}
\begin{aligned}
\mathbf{Y}_{(3)}=\mathbf{F}\mathbf{R}\mathbf{C}_{(3)}.
\label{formula12}
\end{aligned}
\end{equation}

According to (\ref{formula11}) and (\ref{formula12}), the fusion task is converted into estimating the subspace $\mathbf{R}$ and the coefficients $\mathbf{C}_{(3)}$.
Since the LR-HSI retains the majority of the spectral characteristics of the HR-HSI, the low-dimensional spectral subspace can be estimated from the LR-HSI using singular value decomposition (SVD), namely
\begin{equation}
\begin{aligned}
\mathbf{X}_{(3)}=\mathbf{U}\mathbf{\Sigma}\mathbf{V}^{T}.
\label{formula13}
\end{aligned}
\end{equation}
By preserving the $L$ largest singular values, the subspace $\mathbf{R}$ can be derived, as shown in the following equation:
\begin{equation}
\begin{aligned}
\mathbf{R}=\mathbf{U}(:, 1: L).
\label{formula14}
\end{aligned}
\end{equation}

\subsection{Coefficients Estimation}

According to the subspace representation given in (\ref{formula9}), the fusion task seeks to estimate both the spectral basis $\mathbf{R}$ and the coefficients $\mathcal{C}$ from the observations $\mathcal{X}$ and $\mathcal{Y}$. Under the condition that the spectral basis is known, $\mathbf{C}_{(3)}$ could be determined by combining (\ref{formula11}) and (\ref{formula12}), which yields
\begin{equation}
\begin{aligned}
\min\limits_{\mathbf{C}_{(3)}}
&\|\mathbf{X}_{(3)} - \mathbf{R}\mathbf{C}_{(3)}{\mathbf{B}}\mathbf{S}\|_F^2+\|\mathbf{Y}_{(3)} - \mathbf{F}\mathbf{R}\mathbf{C}_{(3)}\|_F^2
\label{formula15}
\end{aligned}
\end{equation}
where $\|\cdot\|_{F}$ denotes the Frobenius norm of a matrix.
However, (\ref{formula15}) is an ill-posed issue, and the solution of $\mathbf{C}_{(3)}$ is not unique.
To address this, related prior information must be incorporated to regularize $\mathbf{C}_{(3)}$, resulting in the following optimization problem:
\begin{equation}
\begin{aligned}
\min\limits_{\mathbf{C}_{(3)}}
&\|\mathbf{X}_{(3)} - \mathbf{R}\mathbf{C}_{(3)}{\mathbf{B}}\mathbf{S}\|_F^2+\|\mathbf{Y}_{(3)} - \mathbf{F}\mathbf{R}\mathbf{C}_{(3)}\|_F^2+\psi(\mathcal{C})
\label{formula16}
\end{aligned}
\end{equation}
where $\psi (\cdot) $ denotes the regularization function. The first two terms constitute the fidelity terms,
while $\psi(\mathcal{C})$ serves as the regularization term for the coefficients $\mathcal{C}$.

HR-HSIs are often locally low-rank. Globally, they exhibit numerous similar spatial-spectral structures. In particular, they possess nonlocal similarity in the spatial domain.

Since HR-MSI retains most spatial information of HR-HSI, the nonlocal similarity can be effectively learned from HR-MSI.
Specifically, the HR-MSI is divided into multiple patches, each possessing a spatial dimension of $\sqrt{q}\times\sqrt{q}$ and containing $s$ bands, and then the patches are clustered into $N$ groups. Each group is denoted as $\mathbf{Y}^{(n)}$ =
$\{\mathcal{Y}^{(n, p)}\}^{K_n}_{p=1}$, $n=1, 2,..., N$,
where  $K_n$ is the number of patches in the $n$-th group.
The clustering procedure is implemented through the $K$-means++ approach \cite{k-means}. Due to the orthogonality of the subspace columns, the spectral vector $\mathcal{Z}(i, j, :)$ and $\mathcal{C}(i, j, :)$ have an identical Frobenius norm. Therefore, the self-similarities present in the image domain are mapped onto the domain of subspace coefficients. The clustering structure learned from HR-MSI is utilized to cluster coefficient patches represented as $\mathbf{C}^{(n)}$ = $\{\mathcal{C}^{(n,p)}\in\mathbb{R}^{\sqrt{q}\times\sqrt{q}\times L}\}^{K_n}_{p=1}$, where $\mathcal{C}^{(n,p)}$ and $\mathcal{Y}^{(n, p)}$ have the same location. We collect the coefficients from the $n$-th group $\{\mathcal{C}^{(n,p)}\}^{K_n}_{p=1}$ into a new tensor $\mathcal{C}^n\in\mathbb{R}^{K_n\times L\times q}$. The tensor $\mathcal{C}^n$ is composed of similar coefficient patches. Thus, we impose the low-rank tensor regularization on the grouped tensor.

In addition to the low-rank prior, the local smoothness prior
has also been widely studied for HSI-SR \cite{FSTRD,LTNN,WNNM}.
Particularly, their joint models encoded as the sum of two independent regularizers were also investigated.
However, these methods fail to fully leverage both the global correlations and local smoothness inherent in HR-HSIs, which generate suboptimal fusion performance.
Therefore, we propose a unique tensor regularizer that encodes both priors simultaneously under a unified framework.
Specifically, this regularizer requires gradient computation across all tensor modes, followed by calculating the logarithmic tensor nuclear norm (LTNN) for the gradient tensors along the third dimension.
To produce a more precise low-rank approximation, we introduce the
LTNN instead of the TNN.
Our constructed model is as follows:
\begin{equation}
\begin{aligned}
\min\limits_{\mathbf{C}_{(3)}}
&\|\mathbf{X}_{(3)} - \mathbf{R}\mathbf{C}_{(3)}{\mathbf{B}}\mathbf{S}\|_F^2+\|\mathbf{Y}_{(3)} - \mathbf{F}\mathbf{R}\mathbf{C}_{(3)}\|_F^2+{\sum_{n=1}^{N}\|\mathcal{C}^n\|_\mathrm{JLRST}}
\label{formula17}
\end{aligned}
\end{equation}
where the proposed regularizer is defined as follows:
\begin{equation}
\|\mathcal{X}\|_\mathrm{JLRST}=\sum_{i=1}^{3}\alpha_i\|{\nabla_i(\mathcal{X})}\|_\mathrm{LTNN}.
\label{formula5}
\end{equation}
Here, $\alpha_1, \alpha_2$ and $\alpha_3 $ are regularization parameters.
$\|\cdot\|_\mathrm{LTNN}$ represents the LTNN, which is denoted as follows:
for $\mathcal{X}\in\mathbb{R}^{I_1\times{I_2}\times{I_3}}$,
\begin{equation}
\begin{aligned}
\|\mathcal{X}\|_\mathrm{LTNN}=\frac{1}{I_3}\sum_{i=1}^{I_3}\sum_j\mathrm{log}\left(\sigma_j(\overline{\mathcal{X}}(:, :, i))+\epsilon\right).
\label{formula4}
\end{aligned}
\end{equation}
Here, $\epsilon>0$ is a very small constant that prevents
the logarithm from being undefined when the singular value is zero. In our experiments, we set $\epsilon=10^{-8}$.
For tensor $\overline{\mathcal{X}} = \mathbf{fft}(\mathcal{X}, [\ ], 3)$, the term $\sigma_j(\overline{\mathcal{X}}(:, :, i))$ indicates the $j$-th singular value of its $i$-th frontal slice.

\subsection{Algorithm}
In this subsection, we develop an ADMM for solving the proposed model (\ref{formula17}). The ADMM is an effective approach to address the minimization problems involving non-differentiable and higher-order functionals \cite{ADMM}.
The flowchart of the proposed JLRST method is shown in Fig.~\ref{figure0}.
\begin{figure*}[h!]
\centering
{\includegraphics[width=6in]{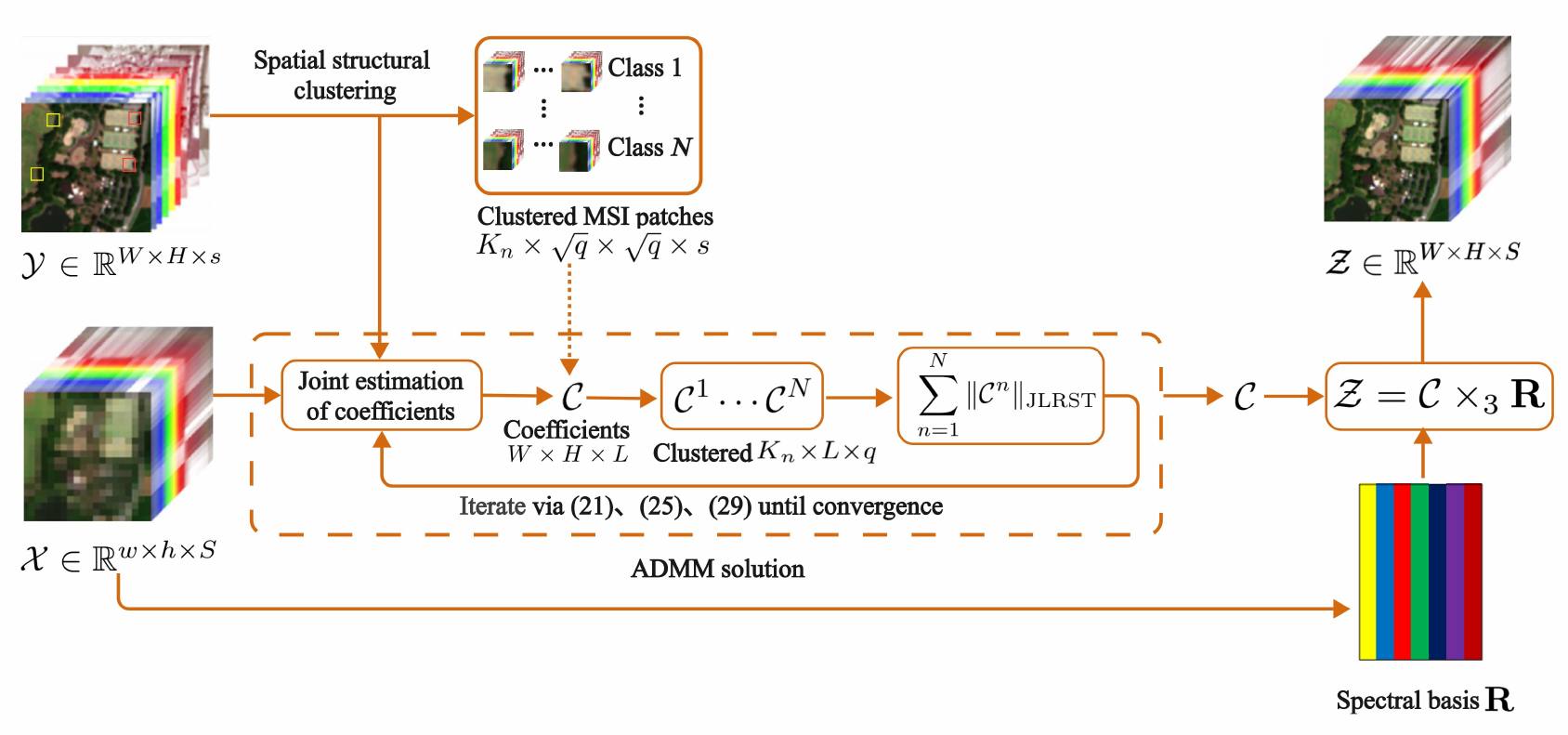}}
\caption{Illustration of the proposed JLRST method.}
\label{figure0}
\end{figure*}

More specifically, the optimization problem (\ref{formula17}) can be decomposed into several subproblems through the application of ADMM,
each of which can be solved analytically. Initially, we introduce six auxiliary variables $\mathcal{G}_1, \mathcal{G}_2, \mathcal{G}_3, \mathcal{H}_1, \mathcal{H}_2$ and $\mathcal{H}_3$ into the proposed model (\ref{formula17}). Consequently, it can be transformed into a constrained optimization problem as follows:
\begin{equation}
\begin{aligned}
\min\limits_{\mathbf{C}_{(3)}}
&\|\mathbf{X}_{(3)} - \mathbf{R}\mathbf{C}_{(3)}{\mathbf{B}}\mathbf{S}\|_F^2+\|\mathbf{Y}_{(3)} - \mathbf{F}\mathbf{R}\mathbf{C}_{(3)}\|_F^2\\
&+{\sum_{n=1}^{N}(\alpha_1\|{\mathcal{H}^{n}_1}\|_\mathrm{LTNN}+\alpha_2\|{\mathcal{H}^{n}_2}\|_\mathrm{LTNN}+\alpha_3\|{\mathcal{H}^{n}_3\|_\mathrm{LTNN}})}\\
\hbox{s.t.}\
&\mathcal{G}_1=\mathcal{C},\mathcal{G}_2=\mathcal{C},\mathcal{G}_3=\mathcal{C},  \mathcal{H}_1=\nabla_1{(\mathcal{G}_1)},
\mathcal{H}_2=\nabla_2{(\mathcal{G}_2)}, \mathcal{H}_3=\nabla_3{(\mathcal{G}_3)}.
\label{formula18}
\end{aligned}
\end{equation}
To solve problem (\ref{formula18}), six Lagrangian multipliers $\mathcal{M}_1, \mathcal{M}_2, \mathcal{M}_3$, $\mathcal{V}_1,\mathcal{V}_2$ and $\mathcal{V}_3$ are introduced. The corresponding augmented Lagrangian function is given by
\begin{equation}
\begin{aligned}
&L(\mathcal{C},\mathcal{H}_1,\mathcal{H}_2,\mathcal{H}_3,\mathcal{G}_1,\mathcal{G}_2,\mathcal{G}_3;\mathcal{V}_1,\mathcal{V}_2,\mathcal{V}_3,\mathcal{M}_1,\mathcal{M}_2,\mathcal{M}_3)\\
=&\|\mathbf{X}_{(3)} - \mathbf{R}\mathbf{C}_{(3)}{\mathbf{B}}\mathbf{S}\|_F^2+\|\mathbf{Y}_{(3)} - \mathbf{F}\mathbf{R}\mathbf{C}_{(3)}\|_F^2\\
&+\mu\sum_{t=1}^{3}\left\|\mathcal{G}_t-\mathcal{C}+\frac{\mathcal{M}_t}{2\mu}\right\|_F^2+\mu\sum_{t=1}^{3}\left\|\mathcal{H}_t-\nabla_t{(\mathcal{G}_t)}+\frac{\mathcal{V}_t}{2\mu}\right\|_F^2\\
&+\alpha_1\sum_{n=1}^{N}||\mathcal{H}_1^{n}||_{\mathrm{LTNN}}+\alpha_2\sum_{n=1}^{N}||\mathcal{H}_2^{n}||_{\mathrm{LTNN}}
+\alpha_3\sum_{n=1}^{N}||\mathcal{H}_3^{n}||_{\mathrm{LTNN}}
\label{formula19}
\end{aligned}
\end{equation}
where $\mu > 0$ serves as a penalty parameter.
The ADMM implementation requires alternating iterations across seven subproblems with simultaneous updates of six Lagrangian multipliers.
Thus, we obtain (\ref{formula20}), which is presented at the top of this page.
\begin{figure*}
\begin{equation}
\left\{
\begin{aligned}
&\mathbf{C}_{(3)}^{k+1}= \arg\min_{\mathbf{C}_{(3)}}
L(\mathcal{C},\mathcal{H}_1^{k},\mathcal{H}_2^{k},\mathcal{H}_3^{k},\mathcal{G}_1^{k},\mathcal{G}_2^{k},\mathcal{G}_3^{k};\mathcal{V}_1^{k},\mathcal{V}_2^{k},\mathcal{V}_3^{k},\mathcal{M}_1^{k},\mathcal{M}_2^{k},\mathcal{M}_3^{k})\\
&\mathcal{H}_1^{k+1}= \arg\min_{\mathcal{H}_1}
L(\mathcal{C}^{k+1},\mathcal{H}_1,\mathcal{H}_2^{k},\mathcal{H}_3^{k},\mathcal{G}_1^{k},\mathcal{G}_2^{k},\mathcal{G}_3^{k};\mathcal{V}_1^{k},\mathcal{V}_2^{k},\mathcal{V}_3^{k},\mathcal{M}_1^{k},\mathcal{M}_2^{k},\mathcal{M}_3^{k})\\
&\mathcal{H}_2^{k+1}= \arg\min_{\mathcal{H}_2}
L(\mathcal{C}^{k+1},\mathcal{H}_1^{k+1},\mathcal{H}_2,\mathcal{H}_3^{k},\mathcal{G}_1^{k},\mathcal{G}_2^{k},\mathcal{G}_3^{k};\mathcal{V}_1^{k},\mathcal{V}_2^{k},\mathcal{V}_3^{k},\mathcal{M}_1^{k},\mathcal{M}_2^{k},\mathcal{M}_3^{k})\\
&\mathcal{H}_3^{k+1}= \arg\min_{\mathcal{H}_3}
L(\mathcal{C}^{k+1},\mathcal{H}_1^{k+1},\mathcal{H}_2^{k+1},\mathcal{H}_3,\mathcal{G}_1^{k},\mathcal{G}_2^{k},\mathcal{G}_3^{k};\mathcal{V}_1^{k},\mathcal{V}_2^{k},\mathcal{V}_3^{k},\mathcal{M}_1^{k},\mathcal{M}_2^{k},\mathcal{M}_3^{k})\\
&\mathcal{G}_1^{k+1}= \arg\min_{\mathcal{G}_1}
L(\mathcal{C}^{k+1},\mathcal{H}_1^{k+1},\mathcal{H}_2^{k+1},\mathcal{H}_3^{k+1},\mathcal{G}_1,\mathcal{G}_2^{k},\mathcal{G}_3^{k};\mathcal{V}_1^{k},\mathcal{V}_2^{k},\mathcal{V}_3^{k},\mathcal{M}_1^{k},\mathcal{M}_2^{k},\mathcal{M}_3^{k})\\
&\mathcal{G}_2^{k+1}= \arg\min_{\mathcal{G}_2}
L(\mathcal{C}^{k+1},\mathcal{H}_1^{k+1},\mathcal{H}_2^{k+1},\mathcal{H}_3^{k+1},\mathcal{G}_1^{k+1},\mathcal{G}_2,\mathcal{G}_3^{k};\mathcal{V}_1^{k},\mathcal{V}_2^{k},\mathcal{V}_3^{k},\mathcal{M}_1^{k},\mathcal{M}_2^{k},\mathcal{M}_3^{k})\\
&\mathcal{G}_3^{k+1}= \arg\min_{\mathcal{G}_3}
L(\mathcal{C}^{k+1},\mathcal{H}_1^{k+1},\mathcal{H}_2^{k+1},\mathcal{H}_3^{k+1},\mathcal{G}_1^{k+1},\mathcal{G}_2^{k+1},\mathcal{G}_3;\mathcal{V}_1^{k},\mathcal{V}_2^{k},\mathcal{V}_3^{k},\mathcal{M}_1^{k},\mathcal{M}_2^{k},\mathcal{M}_3^{k})\\
&\mathcal{V}_t^{k+1}=\mathcal{V}_t^{k}+2\mu(\mathcal{H}_t^{k+1}-\nabla_t(\mathcal{G}_t^{k+1})),\ t=1,2,3\\
&\mathcal{M}_t^{k+1}=\mathcal{M}_t^{k}+2\mu(\mathcal{G}_t^{k+1}-\mathcal{C}^{k+1}),\ t=1,2,3
\end{aligned}
\label{formula20}
\right.
\end{equation}
\end{figure*}

\subsubsection{\texorpdfstring{Estimate $\mathbf{C}_{(3)}$}{C3}}
\begin{equation}
\begin{aligned}
\mathbf{C}_{(3)}^{k+1}
=\arg\min_{\mathbf{C}_{(3)}}&\|\mathbf{X}_{(3)} - \mathbf{R}\mathbf{C}_{(3)}{\mathbf{B}}\mathbf{S}\|_F^2+\|\mathbf{Y}_{(3)} - \mathbf{F}\mathbf{R}\mathbf{C}_{(3)}\|_F^2
+\mu\sum_{t=1}^{3}\left\|\mathbf{G}_{t(3)}^{k}-\mathbf{C}_{(3)}+\frac{\mathbf{M}_{t(3)}^{k}}{2\mu}\right\|_F^2
\label{formula21}
\end{aligned}
\end{equation}
where $\mathbf{G}_{t(3)}$ and $\mathbf{M}_{t(3)}$ are mode-3 unfolding matrices of $\mathcal{G}_t$ and $\mathcal{M}_t$, respectively.
Note that $\mu\left\|\mathbf{G}_{1(3)}-\mathbf{C}_{(3)}+\mathbf{M}_{1(3)}/2{\mu}\right\|_F^2$ is equivalent to $\mu\left\|\mathcal{G}_1-\mathcal{C}+{\mathcal{M}_1}/2{\mu}\right\|_F^2$.

Since the optimization problem (\ref{formula21}) is strongly convex, it has a unique global optimal solution. This solution can be determined by setting the derivative of the objective function with respect to $\mathbf{C}_{(3)}$ to zero. Therefore, we derive the following Sylvester equation:
\begin{equation}
\begin{aligned}
\mathbf{Q}_1\mathbf{C}_{(3)}+\mathbf{C}_{(3)}\mathbf{Q}_2=\mathbf{Q}_3^{k+1}
\label{formula22}
\end{aligned}
\end{equation}
where
\begin{equation}
\begin{aligned}
&\mathbf{Q}_1=[(\mathbf{F}\mathbf{R})^T\mathbf{F}\mathbf{R}+3\mu\mathbf{I}]\\
&\mathbf{Q}_2=(\mathbf{B}\mathbf{S})(\mathbf{B}\mathbf{S})^T\\
&\mathbf{Q}_3^{k+1}=(\mathbf{F}\mathbf{R})^T\mathbf{Y}_{(3)}+\mathbf{R}^T\mathbf{X}_{(3)}(\mathbf{B}\mathbf{S})^T
+\mu\sum_{t=1}^{3}\left (\mathbf{G}_{t(3)}^{k}+\frac{\mathbf{M}_{t(3)}^{k}}{2\mu}\right).
\label{formula23}
\end{aligned}
\end{equation}
Note that the spectral subspace satisfies $\mathbf{R}^T\mathbf{R}=\mathbf{I}$, where $\mathbf{I}$ denotes the identity matrix.
The Sylvester equation (\ref{formula22}) can be solved by utilizing the characteristics of $\mathbf{B}$ and $\mathbf{S}$ \cite{Sylvester}.

\subsubsection{\texorpdfstring{Estimate $\mathcal{H}_t$}{Ht}}
\begin{equation}
\begin{aligned}
\mathcal{H}_t^{k+1} =\arg\min_{\mathcal{H}_t}&\mu\left\|\mathcal{H}_t-\nabla_t{(\mathcal{G}_t^{k})}+\frac{\mathcal{V}_t^{k}}{2\mu}\right\|_F^2+\alpha_t\sum_{n=1}^{N}\left\|\mathcal{H}_t^{n}\right\|_{\mathrm{LTNN}}.
\label{formula24}
\end{aligned}
\end{equation}
Its solution can be acquired by adopting a cluster-by-cluster strategy, i.e.,
\begin{equation}
\begin{aligned}
\arg\min_{\mathcal{H}_t^{n}}&\mu\sum_{n=1}^{N}\left\|\mathcal{H}_t^{n}-\nabla_t{(\mathcal{G}_t^{n,k})}+\frac{\mathcal{V}_t^{n,k}}{2\mu}\right\|_F^2+\alpha_t\sum_{n=1}^{N}||\mathcal{H}_t^{n}||_{\mathrm{LTNN}}.
\label{formula25}
\end{aligned}
\end{equation}
According to the work in \cite{MIDB}, the solution of $\mathcal{H}_t^{n}$ is given by
\begin{equation}
\begin{aligned}
\mathcal{H}_t^{n,k+1}=\mathcal{D}_{\mathrm{LTNN}}^{\alpha_t/2\mu}(\nabla_t{(\mathcal{G}_t^{n,k})}-\frac{\mathcal{V}_t^{n,k}}{2\mu})=\mathcal{U}\mathcal{S}_{\mathrm{LTNN}}^{\alpha_t/2\mu,\epsilon}\mathcal{V}^T
\label{formula26}
\end{aligned}
\end{equation}
where $\mathcal{S}_{\mathrm{LTNN}}^{\alpha_t/2\mu,\epsilon}$=$\mathbf{ifft}(\mathcal{\overline{S}}_{\mathrm{LTNN}}^{\alpha_t/2\mu,\epsilon},[\ ],3)$, $\nabla_t{(\mathcal{G}_t^{n,k})}-\frac{\mathcal{V}_t^{n,k}}{2\mu}$=$\mathcal{U}\mathcal{S}\mathcal{V}^T$. The
 thresholding operator $\mathcal{\overline{S}}_{\mathrm{LTNN}}^{\alpha_t/2\mu,\epsilon}$ is defined as follows:
\begin{equation}
\begin{aligned}
\mathcal{\overline{S}}_{\mathrm{LTNN}}^{\alpha_t/2\mu,\epsilon}(i_1,i_2,i_3)
=\left\{
\begin{array}{ll}
0,&c_{2}\leq 0\\
\mathrm{sign}(\overline{S}(i_1,i_2,i_3))(\frac{c_1+\sqrt{c_2}}{2}),&c_{2}>0.
\end{array}
\right.
\label{formula27}
\end{aligned}
\end{equation}
Specifically, $\overline{S}$= $\mathbf{fft}(\mathcal{S,[\ ]},3)$, $c_{1}=\left|\overline{S}(i_1,i_2,i_3) \right|-\epsilon$, $c_{2}=c_{1}^2-4(\alpha_t/2\mu-\epsilon\left|\overline{S}(i_1,i_2,i_3) \right|)$. Through this approach,  the threshold operator applies reduced shrinkage to larger singular values while imposing greater shrinkage on smaller ones\cite{MNRH}.

\subsubsection{\texorpdfstring{Estimate $\mathcal{G}_t$}{Gt}}
\begin{equation}
\begin{aligned}
\mathcal{G}_t^{k+1}=\arg\min_{\mathcal{G}_t}&\mu\left\|\mathcal{G}_t-\mathcal{C}^{k+1}+\frac{\mathcal{M}_t^{k}}{2\mu}\right\|_F^2+\mu\left\|\mathcal{H}_t^{k+1}-\nabla_t{(\mathcal{G}_t)}+\frac{\mathcal{V}_t^{k}}{2\mu}\right\|_F^2.
\label{formula28}
\end{aligned}
\end{equation}
Taking the derivative in (\ref{formula28}) with respect to $\mathcal{G}_t$, we have
\begin{equation}
\begin{aligned}
(\mathcal{I}+\nabla_t^{\mathrm{T}}\nabla_t)(\mathcal{G}_t)
=\mathcal{C}^{k+1}-\frac{\mathcal{M}_t^{k}}{2\mu}+\nabla_t^{\mathrm{T}}\left (\mathcal{H}_t^{k+1}+\frac{\mathcal{V}_t^{k}}{2\mu}\right)
\label{formula29}
\end{aligned}
\end{equation}
where $\nabla_t^{\mathrm{T}}(\cdot)$ signifies the transpose operator of $\nabla_t(\cdot)$. Due to the linearity of tensor difference operations, we can employ the multi-dimensional FFT to diagonalize $\nabla_t(\cdot)$'s difference tensors $\mathcal{D}_t$. This enables us to compute the optimal solution of (\ref{formula29}) efficiently, i.e.,
\begin{equation}
\begin{split}
\mathcal{G}_t^{k+1}=\mathcal{F}^{-1}\left(
\frac{\mathcal{F}(\mathcal{C}^{k+1}-\frac{\mathcal{M}_t^{k}}{2\mu})+\mathcal{Q}}{\mathbf{1}+\mathcal{F}(\mathcal{D}_t)^{\mathrm{T}}\odot\mathcal{F}(\mathcal{D}_t)}\right)
\label{formula30}
\end{split}
\end{equation}
where $\mathcal{Q}=\mathcal{F}(\mathcal{D}_t)^{\mathrm{T}}\odot\mathcal{F}(\mathcal{H}_t^{k+1}+\mathcal{V}_t^{k}/2\mu)$,
$\mathbf{1}$ is a tensor with all elements being 1, and $\odot$ represents component-wise multiplication.

Algorithm \ref{algorithm1} outlines the proposed JLRST method.

\begin{algorithm}[!ht]\caption{JLRST for HSI-SR} \label{algorithm1}
    \renewcommand{\algorithmicrequire}{\textbf{Input:}}
	\renewcommand{\algorithmicensure}{\textbf{Output:}}
    \begin{algorithmic}[1]
        \REQUIRE $\mathcal{X}, \mathcal{Y}, \alpha_t, N, L$, and $\sqrt{q}$
        \STATE  \qquad Estimate spectral basis $\mathbf{R}$ from $\mathbf{X}_{(3)}$ via SVD;
        \STATE  \qquad Obtain the cluster structure from the HR-MSI;
        \STATE \quad\textbf{while} not converged \textbf{do}
        \STATE  \qquad Compute $\mathbf{C}_{(3)}$ by (\ref{formula22});
        \STATE  \qquad Compute $\mathcal{H}_{t}$ by (\ref{formula26});
        \STATE  \qquad Compute $\mathcal{G}_t$ by (\ref{formula30});
        \STATE  \qquad Compute Lagrangian multipliers $\mathcal{V}_t$ and $\mathcal{M}_t$ by the last two equations of (\ref{formula20});
        \STATE \quad\textbf{end while}
         \STATE  \qquad$\mathcal{Z}= \mathcal{C}\times_{3}\mathbf{R}$;
         \ENSURE $\mathcal{Z}$.
    \end{algorithmic}
\end{algorithm}

\subsection{Convergence Analysis}
First, we derive the Karush-Kuhn–Tucker (KKT) condition of \eqref{formula18}. The condition is as follows:

\begin{equation}\label{kkt}
\left\{\begin{array}{ll}
\mathcal{O}&=\mathcal{G}_t^\ast-\mathcal{C}^\ast\\
\mathcal{O}&=\mathcal{H}_t^\ast-\nabla_t{(\mathcal{G}_t^\ast)} \\
\mathbf{O}&=\mathbf{R}^T(\mathbf{R} \h C_{(3)}^\ast\h B\h S  -\h X_{(3)})\mathbf{S}^T\mathbf{B}^T  + \h R^T \h F^T (\h F \h R\h C_{(3)}^\ast-\h Y_{(3)})- \tfrac{1}{2}\sum\limits_{i=1}^{3} \h M_{i(3)}^\ast   \\
\mathcal{O} &= \mathcal{M}_t^\ast- \nabla^T_t\mathcal{V}_t^\ast\\
\mathcal{O}& \in \alpha_t\partial_{\mathcal{H}_t} \|\mathcal{H}_t^\ast\|_{\rm LTNN} + \mathcal{V}_t^\ast,\ \hbox{for}\ t = 1,2,3.
\end{array}\right.
\end{equation}

Now, we provide the convergence analysis as the following theorem:
\begin{theorem}
Let $\{\mathcal{C}^k, \mathcal{G}_t^k, \mathcal{H}_t^k \}$  be generated by Algorithm \ref{algorithm1}. Assume the successive differences of the multipliers $\mathcal{M}_t^{k+1}-\mathcal{M}_t^{k}\rightarrow \mathcal{O}, \mathcal{V}_t^{k+1}-\mathcal{V}_t^{k}\rightarrow \mathcal{O},$ when $k\rightarrow \infty$, and all the variables are bounded, then there exists a subsequence whose accumulation point satisfies the KKT condition of \eqref{formula18}.
\end{theorem}
\begin{proof}
    According to the iteration scheme in \eqref{formula20} as well as   $\lim\limits_{k\rightarrow \infty} \mathcal{M}_t^{k+1}-\mathcal{M}_t^{k} = \mathcal{O}$,
    $\lim\limits_{k\rightarrow \infty} \mathcal{V}_t^{k+1}-\mathcal{V}_t^{k} = \mathcal{O}$,
    we get

\begin{equation*}
\left\{
    \begin{aligned}
        &\lim\limits_{k\rightarrow \infty}\mathcal{G}_t^k-\mathcal{C}^k=\mathcal{O} \\
&\lim\limits_{k\rightarrow \infty}\mathcal{H}_t^k-\nabla_t{(\mathcal{G}_t^k)}=\mathcal{O}, \quad  t = 1,2,3.
    \end{aligned}
    \right.
\end{equation*}
With the boundedness of all the variables, it shows that there exists a bounded subsequence such that $\lim\limits_{j\rightarrow \infty}  \mathcal{G}_t^{k_j} = \mathcal{G}_t^\ast, \lim\limits_{j\rightarrow \infty}\mathcal{C}^{k_j} = \mathcal{C}^{\ast}$  and $\lim\limits_{j \rightarrow \infty}  \mathcal{H}_t^{k_j} = \mathcal{H}_t^{\ast}$. Thus, the first two equations in \eqref{kkt}  hold.
Additionally, the optimality condition in the $\mathcal{C}$-subproblem can be written as
\begin{equation*}
  \begin{split}
        & \mathbf{R}^T(\mathbf{R} \h C_{(3)}^\ast\h B\h S  -\h X_{(3)})\mathbf{S}^T\mathbf{B}^T + \h R^T \h F^T (\h F \h R\h C_{(3)}^\ast-\h Y_{(3)}) +  \sum_{i = 1}^3(\mu\h C_{(3)}^\ast- \mu\h G_{i(3)}^\ast - \h M_{i(3)}^\ast/2) = \mathbf{O}.
  \end{split}
\end{equation*}
Note that $\mathcal{C}^{\ast} = \mathcal{G}^{\ast}_i$, which leads to the third equation of \eqref{kkt}.
We can obtain the remaining equations in \eqref{kkt} by referring to the optimality conditions of the other subproblems.

\end{proof}

\section{Numerical
Experiments}
\label{section4}
In this section, we begin by detailing the datasets utilized in our experiments. Subsequently, we introduce five widely recognized evaluation metrics for assessing the effectiveness of our approach.
Furthermore, an in-depth analysis is performed on the impact of different parameters on the experimental results, and a table about optimal parameters is provided.
To comprehensively evaluate the proposed method's performance, we carry out systematic comparisons with six advanced approaches using both quantitative metrics and visual quality.
Additionally, ablation experiments are executed to assess the influence of each regularization term on the results.
An empirical convergence analysis of the proposed method is also conducted.
Finally, we compare the running time of all testing methods.

\subsection{Datasets}

The Pavia University dataset \cite{pavia university}, acquired using the ROSIS sensor, comprises 115 spectral bands with an original spatial resolution of 610 × 340 pixels. After eliminating bands with a low signal-to-noise ratio (SNR) and trimming the subregion, we selected the upper-left 256 $\times$ 256 pixels comprising 93 spectral bands as the reference image. The Pavia University dataset is processed using an IKONOS-like reflectance spectral response filter to simulate the HR-MSI with four bands.

The second dataset, Indian Pines \cite{indian pines}, was captured by the NASA AVIRIS sensor over the Indian Pines test site.
The original image consists of 145 $\times$ 145 pixels and 220 spectral bands. After discarding bands severely impacted by noise, the dataset was refined to dimensions of 128 $\times$ 128 $\times$ 184 as the ground truth. The HR-MSI with six bands is generated by Landsat7-like spectral response.

For the third dataset, we utilize the well-known and widely used CAVE dataset \cite{CAVE}. This dataset comprises 32 HSIs obtained from real-world indoor scenes using a universal combined pixel camera. It has become widely favored by researchers for HSI-SR experiments.
All the HSIs in the dataset contain 512 $\times$ 512 spatial pixels and 31 spectral bands.
This experiment uses ``Balloons" as the test image. Spectral downsampling is performed on the HR-HSI via the Nikon D700 camera's spectral response to produce the HR-MSI.

At last, we employ the University of Houston \cite{Houston} as the fourth dataset, which features 601 × 2384 pixels and 48 bands that span wavelengths from 0.38 to 1.05 micrometers. In our experiments, we chose a sub-image measuring 320 $\times$ 320 $\times$ 46 from the entire dataset as a reference for the HR-HSI. The spectral response of a Nikon D700 camera is used to generate the three-band HR-MSI.

To obtain the LR-HSI, we use a symmetric Gaussian blur with a standard deviation of 2 and convolve the HR-HSI using a 7 $\times$ 7 kernel, and then downsample every 4 pixels in both spatial dimensions for each HSI band. To better simulate real-world conditions, Gaussian noise is added to the LR-HSI and HR-MSI, with SNRs of 20 dB and 25 dB, respectively.

\subsection{Quantitative Metrics and Compared Methods  }

To quantitatively assess the effectiveness of HSI-SR approaches, we establish a comprehensive assessment framework incorporating five widely used metrics: the average peak signal-to-noise ratio (PSNR) \cite{WNNM}, structural similarity (SSIM) \cite{ssim}, relative dimensionless global error in synthesis (ERGAS) \cite{ergas}, spectral angle mapper (SAM) \cite{sam}, and universal image quality index (UIQI) \cite{uiqi}. Higher PSNR, SSIM, and UIQI values indicate superior fusion quality, whereas lower ERGAS and SAM values signify better fusion performance.

In our experiments, comparative methods encompass coupled sparse tensor factorization (CSTF) \cite{CSTF}, unidirectional TV (UTV) \cite{UTV}, CTRF \cite{CTRF}, factor smoothed TR decomposition (FSTRD) \cite{FSTRD}, logarithmic low-rank TR decomposition (LogLRTR) \cite{LTNN}, and continuous low-rank factorization (CLoRF) \cite{CLoRF}. To guarantee fairness during the comparison process, the parameters of the comparison approaches are adjusted to achieve optimal performance.
In experiments, while the relative error $\| \mathcal{Z}^{k}- \mathcal{Z}^{k-1}\|_F/\| \mathcal{Z}^{k}\|_F <10^{-4}$ or the number of iterations reaches $100$, the iteration will stop.

\subsection{Parameters Analysis}
The selection of parameters plays a crucial role in the experimental results, greatly affecting the algorithm's convergence rate. Our method is characterized by seven parameters: the atom count $L$, the cluster number $N$, the patch size $q$, three regularization parameters $\alpha_1$, $\alpha_2$ and $\alpha_3$, along with an algorithmic parameter $\mu$.
When adjusting these parameters, we first draw on our experience to estimate appropriate initial values. Subsequently, we gradually fine-tune the parameters through a trial-and-error approach to determine the combination that maximizes the PSNR value.
Generally speaking, only the last four parameters need to be fine-tuned.

\begin{figure}[h!]
\centering\vspace{-1em}
\subfigure[]{
\includegraphics[width=2in]{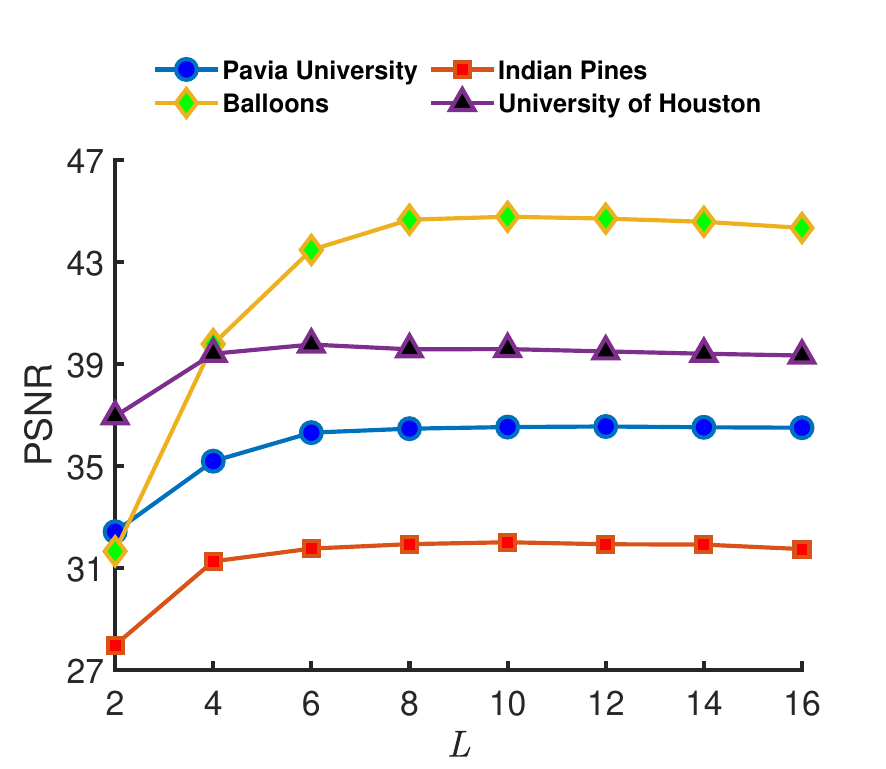}}
\subfigure[]{
\includegraphics[width=2in]{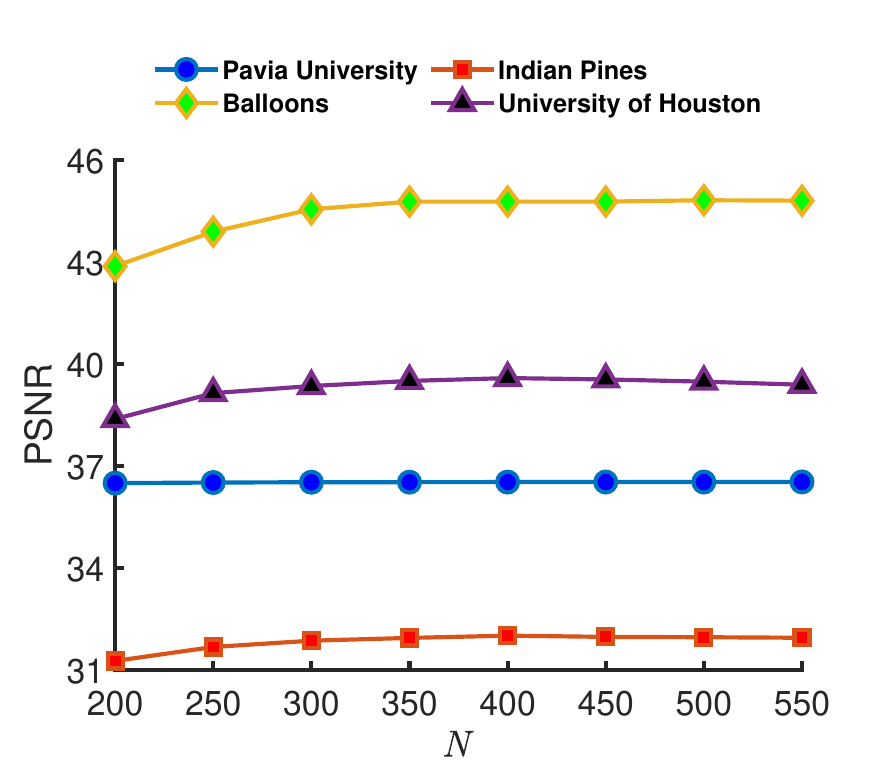}}
\subfigure[]{
\includegraphics[width=2in]{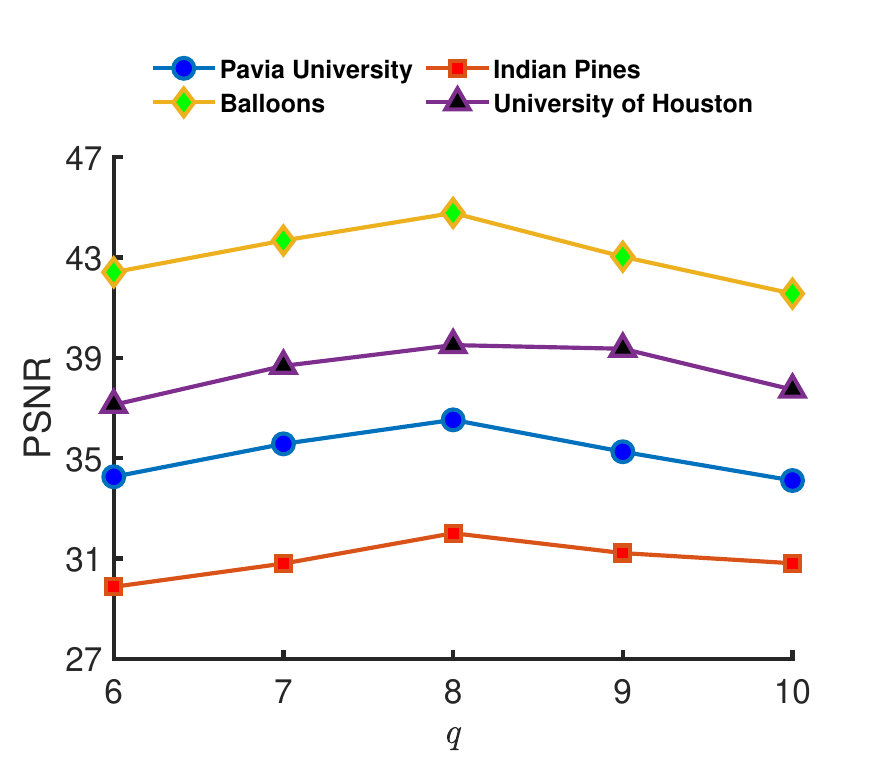}}
\caption{
Plots of PSNR against parameters $L$, $N$ and $q$, respectively.}
\label{figure1}
\end{figure}

To investigate the impact of the subspace dimension $L$, we plot the PSNR values against $L$ for four test images, as presented in Fig.~\ref{figure1}(a).
This graph shows that as the value of $L$ rises from 2 to 10, the PSNR value of each fused image also increases accordingly. However, when $L$ exceeds 10, the PSNR value tends to stabilize.
Consequently, we set $L=10$ in all experiments.
This phenomenon indicates that spectral vectors exist in low-dimensional subspaces, and the spectral dimension can be reduced through the subspace representation.

Similarly,
we depict the PSNR values in relation to $N$ for the four test images, as shown in Fig.~\ref{figure1}(b). An analysis of Fig.~\ref{figure1}(b) reveals that
the choice of the number of clusters has a significant impact on the experimental results.
This also suggests that utilizing non-local similarities to enhance fusion performance is an effective strategy.
At the same time, in order to balance algorithm efficiency, we fix $N=400$ in all experiments.

Furthermore, the patch size $q$ also influences the fusion performance. Based on prior experience, the value of $q$ is usually selected within the range of [6,10]. Fig.~\ref{figure1}(c) plots the PSNR values against $q$ for the test images. It can be observed that when $q=8$, the PSNR reaches its maximum. A smaller $q$ loses local spatial information, while a larger $q$ reduces patch homogeneity. Therefore, we set $q=8$ in all experiments.

The selection of $\alpha_1$, $\alpha_2$, $\alpha_3$ and $\mu$ is crucial. In general, the fusion result primarily depends on the regularization parameters $\alpha_1$, $\alpha_2$ and $\alpha_3$, while the penalty parameter $\mu$ solely governs the convergence speed without affecting the final result. Therefore, our analysis focuses on the effect of the regularization parameters on the experimental results. Fig.~\ref{figure_par}(a)–(c) present the PSNR curves for the JLRST method across the Pavia University and Indian Pines datasets as parameters $\alpha_1$, $\alpha_2$ and $\alpha_3$ are varied.
As seen in Fig.~\ref{figure_par}(a), the highest PSNR on the Pavia University dataset occurs at $\alpha_1 = 10^{-1}$, while the Indian Pines dataset attains its maximum at $\alpha_1 = 10^{-2}$.
Fig.~\ref{figure_par}(b) indicates that the PSNR for Pavia University remains largely stable with increasing $\alpha_2$, and both datasets yield optimal results at $\alpha_2 = 10^{-2}$.
In Fig.~\ref{figure_par}(c), a similar trend is observed for both datasets as $\alpha_3$ increases from $10^{-5}$ to $10^{0}$, with peak PSNR values achieved at $\alpha_3 = 10^{-2}$.
To further evaluate the robustness of these optimal ranges under higher signal-to-noise conditions, we conduct an additional sensitivity analysis with LR-HSI at 30 dB and HR-MSI at 35 dB, as shown in Fig.~\ref{figure_par_2}. Compared with the previous lower-noise case, the PSNR curves in Fig.~\ref{figure_par_2} exhibit a similar overall trend. Therefore, the ranges of values for $\alpha_1$, $\alpha_2$, and $\alpha_3$ are approximately the same.
From these results, we optimize $\alpha_1$, $\alpha_2$ and $\alpha_3$ within the ranges [0.001, 1] for Pavia University and [0.001, 0.1] for Indian Pines to determine the best parameter combinations.

By the trial-and-error method, the optimal sets of parameters for four images are obtained, which are shown in Table~\ref{table1}. Specifically, in our experiments, the values of $L$ and $N$ are consistently fixed at 10 and 400, respectively.
For fairness, we adjust the parameters of the comparison approaches in the same way.

\begin{figure}[h!]
\centering\vspace{-1em}
\subfigure[]{
\includegraphics[width=2in]{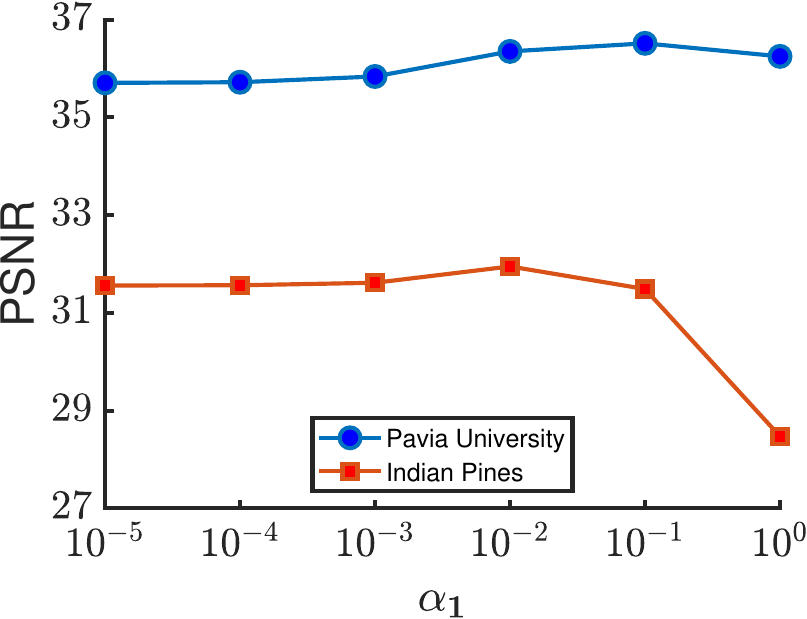}}
\subfigure[]{
\includegraphics[width=2in]{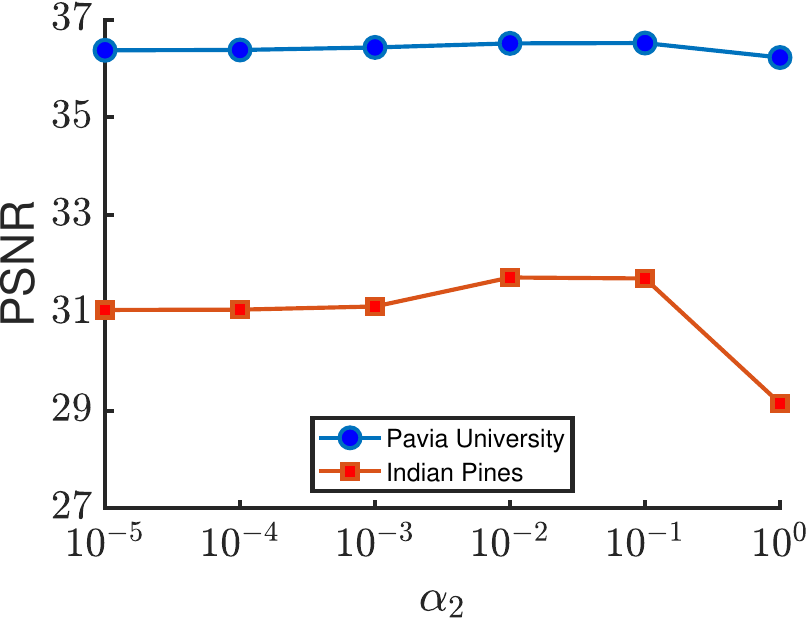}}
\subfigure[]{
\includegraphics[width=2in]{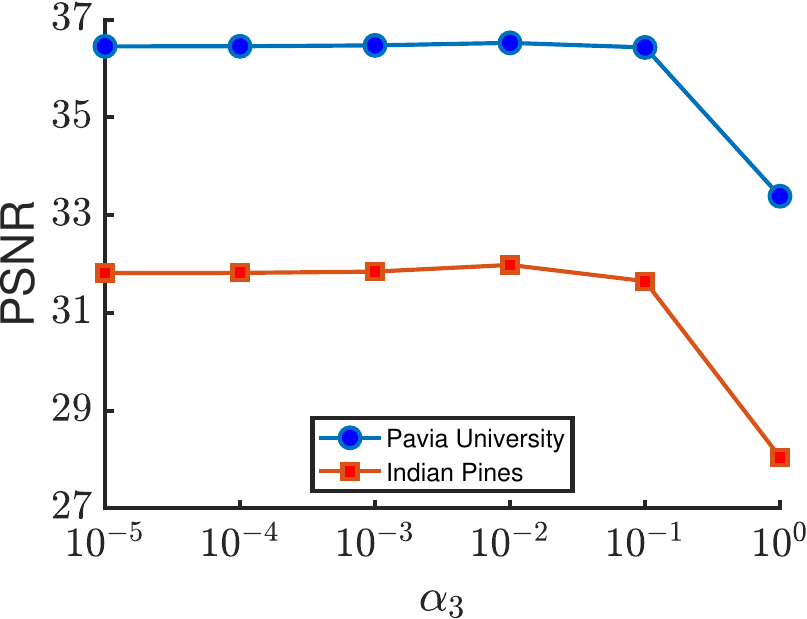}}
\caption{
Sensitivity analysis of parameters $\alpha_1$, $\alpha_2$ and $\alpha_3$ (LR-HSI: 20 dB, HR-MSI: 25 dB)}.
\label{figure_par}
\end{figure}

\begin{figure}[h!]
\centering\vspace{-1em}
\subfigure[]{
\includegraphics[width=2in]{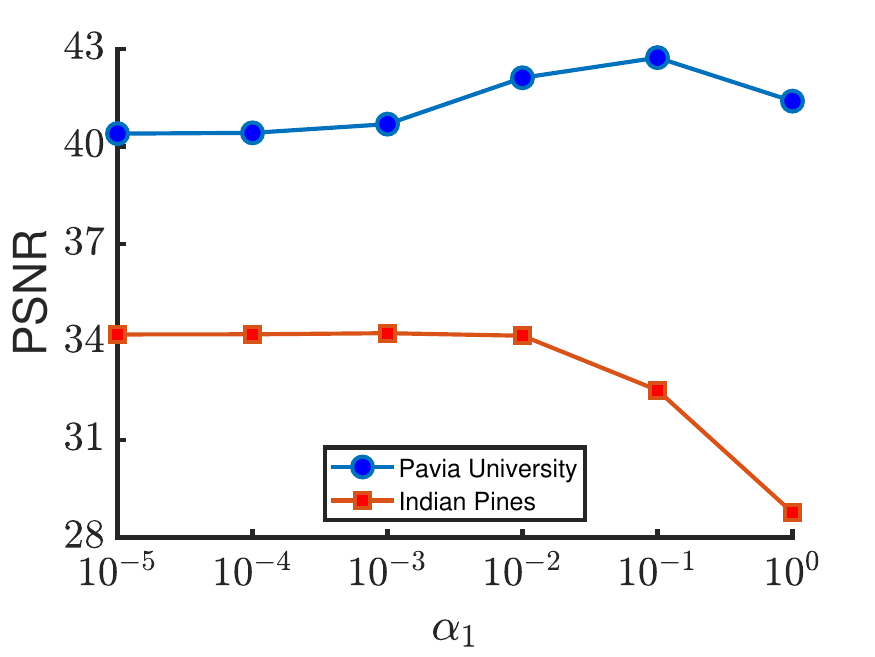}}
\subfigure[]{
\includegraphics[width=2in]{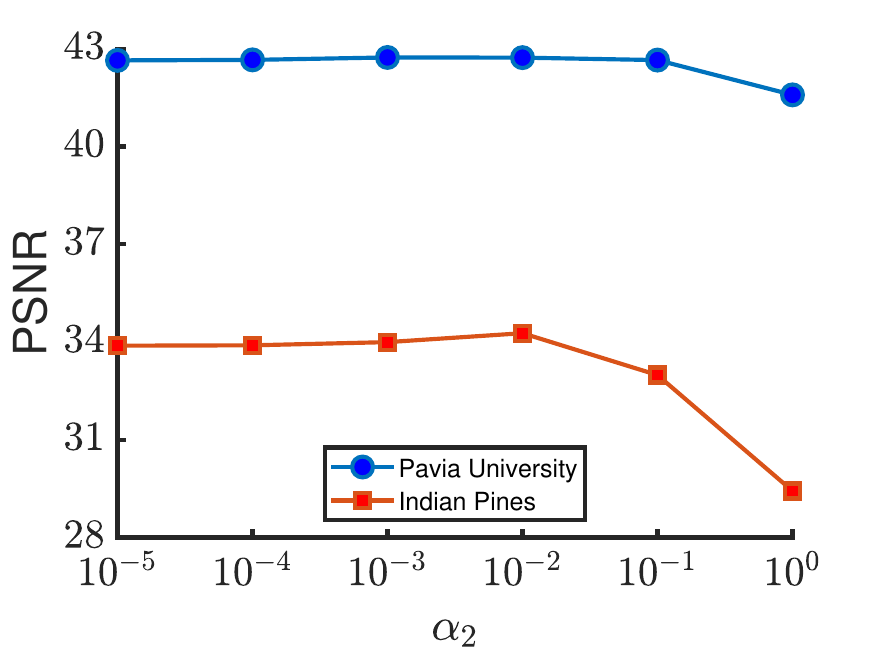}}
\subfigure[]{
\includegraphics[width=2in]{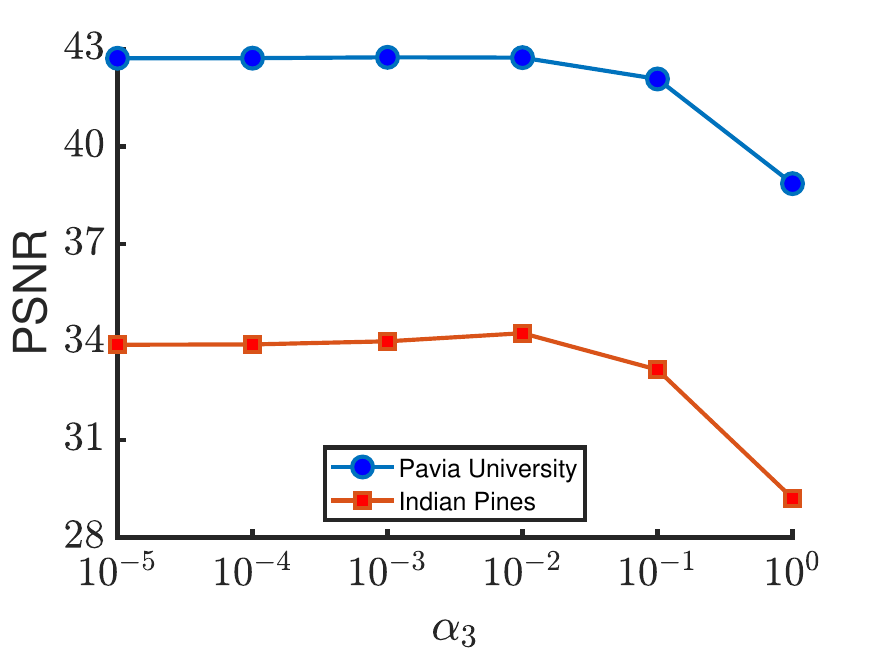}}
\caption{
Sensitivity analysis of parameters $\alpha_1$, $\alpha_2$ and $\alpha_3$ (LR-HSI: 30 dB, HR-MSI: 35 dB)}.
\label{figure_par_2}
\end{figure}

\begin{table}[htbp]\vspace{-1em}
\caption{Selected parameter sets for the proposed method}
\centering
\scalebox{0.8}{\large{
\begin{tabular}{lcccccc}
\hline
{\rule[-1mm]{0mm}{3.5mm}}Image  &$\alpha_1$ &$\alpha_2$ &$\alpha_3$ &$\mu$     \\
\hline
{\rule[-1mm]{0mm}{3.5mm}}Pavia University &0.3 &0.03 &0.009 &0.05  \\
{\rule[-1mm]{0mm}{3.5mm}}Indian Pines      &0.02 &0.03 &0.03 &0.04  \\
{\rule[-1mm]{0mm}{3.5mm}}Balloons   &0.25 &0.2 &0.1 &0.09    \\
{\rule[-1mm]{0mm}{3.5mm}}University of Houston    &0.08 &0.2 &0.05 &0.05   \\
\hline
\label{table1}
\end{tabular}}}\vspace{-1em}
\end{table}

\subsection{Experimental Results}
\begin{table*}[htbp]\vspace{-1em}
\caption{Quantitative evaluation of six different fusion approaches}
\centering
\resizebox{\textwidth}{!}{
\newcommand{\rb}[1]{\raisebox{1.0ex}[0pt]{#1}}
\begin{tabular}{lcccccccccc}
\hline
{\rule[-1mm]{0mm}{3.5mm}}\hbox{Image}&\multicolumn{5}{c}{Pavia University}        &\multicolumn{5}{c}{Indian Pines} \\
\cmidrule(r){2-6}\cmidrule(r){7-11}
{\rule[-1mm]{0mm}{3.5mm}}Method&PSNR&SSIM&ERGAS&SAM&UIQI&PSNR&SSIM&ERGAS&SAM&UIQI\\
\hline
{\rule[-1mm]{0mm}{3.5mm}}Best Values &$+\infty$ &1 &0 &0 &1  &$+\infty$ &1 &0 &0 &1  \\
{\rule[-1mm]{0mm}{3.5mm}}CSTF  &34.633 &0.895&2.785&5.594 &0.951    &30.552 &0.840 &2.296 &4.387 &0.926 \\
{\rule[-1mm]{0mm}{3.5mm}}UTV  &35.448 &0.912 &2.543 &4.957 &0.960     &30.644  &0.849 &2.276 &4.365  &0.924          \\
{\rule[-1mm]{0mm}{3.5mm}}CTRF &35.111 &0.915 &2.645 &5.152 &0.959  &30.542 &0.852 &2.285 &4.346  &0.921     \\
{\rule[-1mm]{0mm}{3.5mm}}FSTRD&35.707 &\textbf{0.940} &2.414 &\textbf{3.896} &\textbf{0.968}   &31.070 &0.866 &2.195 &4.226 &0.925       \\
{\rule[-1mm]{0mm}{3.5mm}}LogLRTR&\underline{36.366} &\underline{0.931} &\textbf{2.253} &4.251 &\textbf{0.968}  &31.457 &0.873 &2.125 &4.098&0.931     \\
{\rule[-1mm]{0mm}{3.5mm}}{CLoRF}  &{36.302} &{0.927}&{\underline{2.309}}&{\underline{4.243}} &{0.965}    &{\underline{31.626}} &{\underline{0.878}} &{\underline{2.062}} &{\underline{3.968}} &{\underline{0.937}} \\
{\rule[-1mm]{0mm}{3.5mm}}JLRST  &\textbf{36.526} &0.925 &2.311 &4.255&\underline{0.966} & \textbf{32.005}&\textbf{0.889}&\textbf{2.016}  &\textbf{3.869}  &\textbf{0.940}
\\
\hline
{\rule[-1mm]{0mm}{3.5mm}}\hbox{Image} &\multicolumn{5}{c}{Balloons}           &\multicolumn{5}{c}{ University of Houston}             \\
\cmidrule(r){2-6}\cmidrule(r){7-11}
{\rule[-1mm]{0mm}{3.5mm}}Method  &PSNR  &SSIM  &ERGAS  &SAM   &UIQI               &PSNR  &SSIM  &ERGAS  &SAM &UIQI\\
\hline
{\rule[-1mm]{0mm}{3.5mm}}Best Values &$+\infty$ &1     &0      &0      &1                  &$+\infty$ &1        &0        &0         &1      \\
{\rule[-1mm]{0mm}{3.5mm}}CSTF  &41.718&0.968&1.438&5.412&0.862   &38.189 &0.946 &1.610 &2.946 &0.978 \\
{\rule[-1mm]{0mm}{3.5mm}}UTV    &41.972&0.964&1.406&5.681&0.850    &38.411 &0.948 &1.580 &2.980&0.979 \\
{\rule[-1mm]{0mm}{3.5mm}}CTRF  &41.262&0.959&1.545&6.153&0.835        &38.353 &0.944 &1.575 &2.915 &0.978  \\
{\rule[-1mm]{0mm}{3.5mm}}FSTRD  &43.477&0.978&1.195 &4.820  &0.887    &38.848&0.952 &1.497&2.709  &0.981 \\
{\rule[-1mm]{0mm}{3.5mm}}LogLRTR   &44.176 &0.981 &1.094 &4.479  &0.896&39.112 &0.956&1.452 &2.524 &0.982  \\
{\rule[-1mm]{0mm}{3.5mm}}{CLoRF}  &{\underline{44.359}} &{\underline{0.987}}&{\underline{1.064}}&{\underline{3.825}} &{\underline{0.910}}    &{\underline{39.389}} &{\textbf{0.965}} &{\underline{1.419}} &{\textbf{2.198}} &{\textbf{0.985}} \\
{\rule[-1mm]{0mm}{3.5mm}}JLRST&\textbf{44.778} &\textbf{0.988} &\textbf{1.010}&\textbf{3.639}&\textbf{0.916}  &\textbf{39.508} &\underline{0.962} &\textbf{1.374} &\underline{2.237} &\underline{0.983}\\ \hline
\label{table2}
\end{tabular}}\vspace{-0.5em}
\end{table*}

The quantitative evaluation results of seven fusion approaches across four datasets are systematically summarized in Table~\ref{table2}. To facilitate comparison, the optimal performance metric values are highlighted in bold, while the suboptimal results are distinctly marked with underlining. This table reveals that our proposed method achieves the best results on all evaluation metrics for ``Indian Pines", and ``Balloons".
It is noteworthy that the JLRST gains PSNR improvements of 0.379 and 0.419 over the suboptimal method on the ``Indian Pines" and ``Balloons" images, respectively.

As a supplement to the table data, Fig.~\ref{figure2} illustrates the correlation curves between the spectral band and two critical quality metrics (PSNR and UIQI) across all assessed methods.
It demonstrates that JLRST acquires superior PSNR and UIQI values across nearly the entire spectral range, with particularly significant advantages observed for both the ``Indian Pines"  and ``Balloons" images.

\begin{figure*}
\centering
\subfigure[]{
    \begin{minipage}[b]{0.225\linewidth}
    \includegraphics[width=4cm]{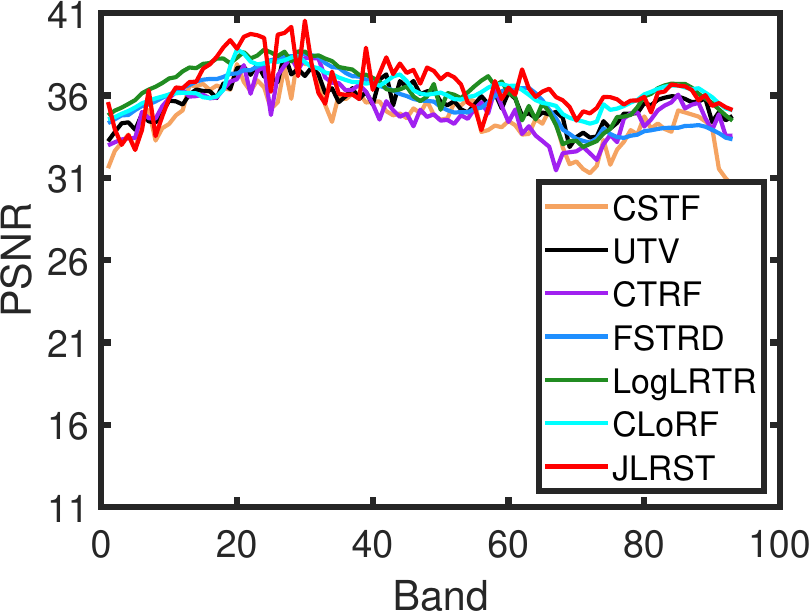}\vspace{1pt}
    \includegraphics[width=4cm]{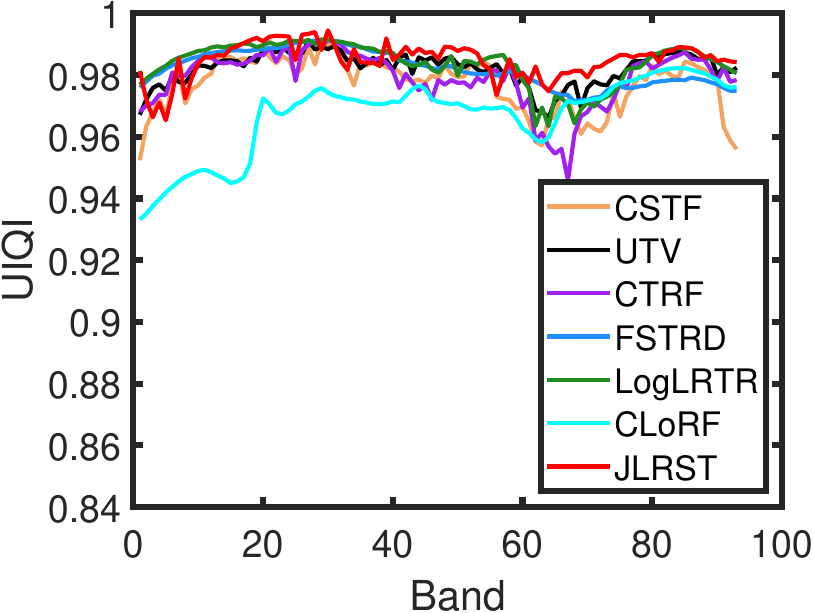}\vspace{1pt}
    \end{minipage}
}
\subfigure[]{
    \begin{minipage}[b]{0.225\linewidth}
    \includegraphics[width=4cm]{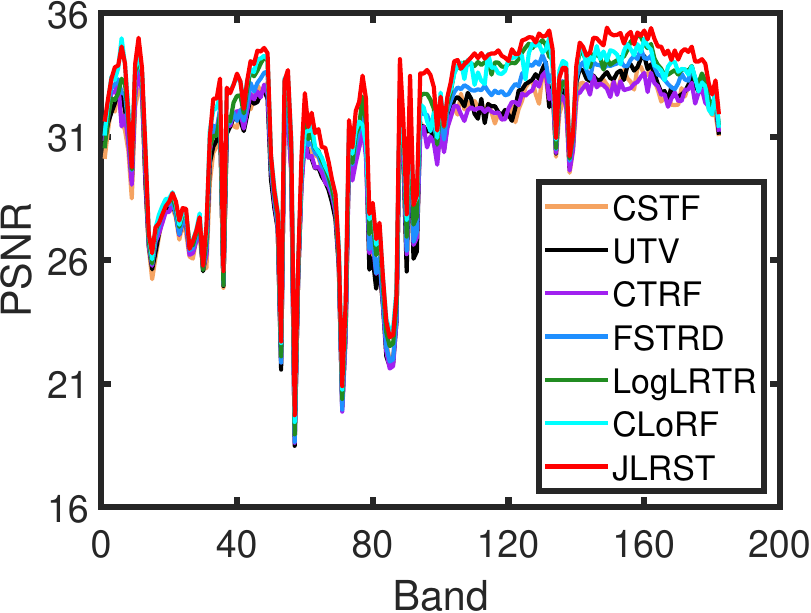}\vspace{1pt}
    \includegraphics[width=4cm]{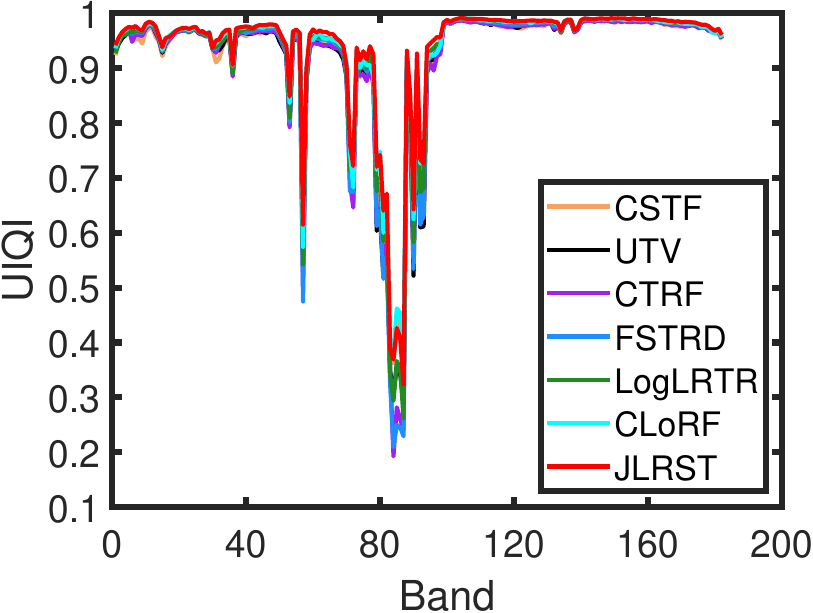}\vspace{1pt}
    \end{minipage}
}
\subfigure[]{
    \begin{minipage}[b]{0.225\linewidth}
    \includegraphics[width=4cm]{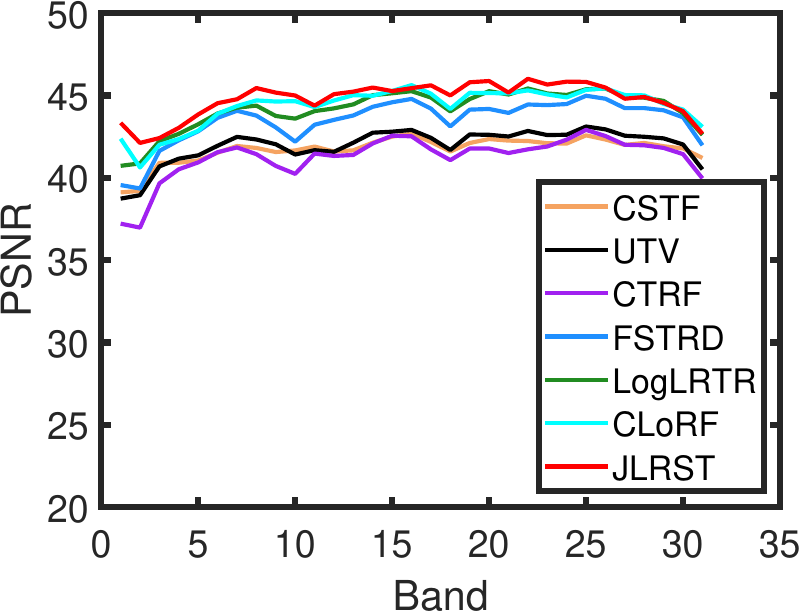}\vspace{1pt}
    \includegraphics[width=4cm]{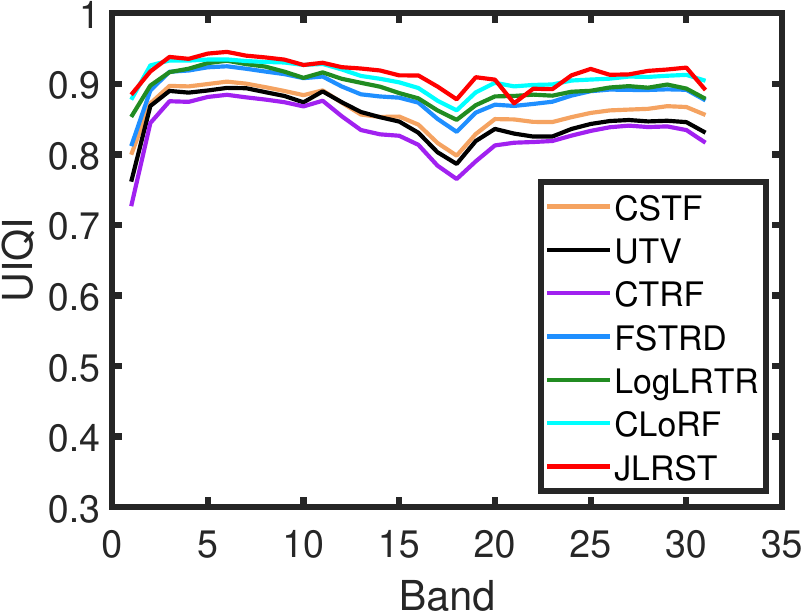}\vspace{1pt}
    \end{minipage}
}
\subfigure[]{
    \begin{minipage}[b]{0.225\linewidth}
    \includegraphics[width=4cm]{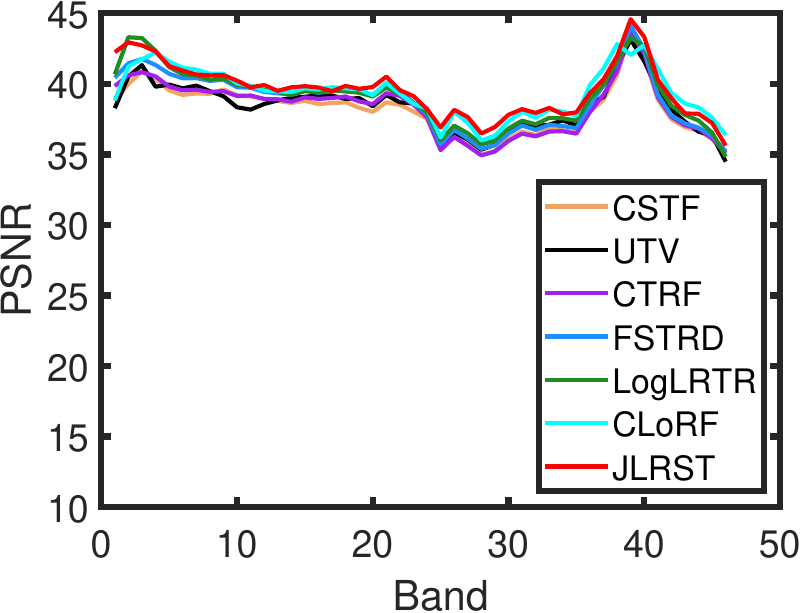 }\vspace{1pt}
    \includegraphics[width=4cm]{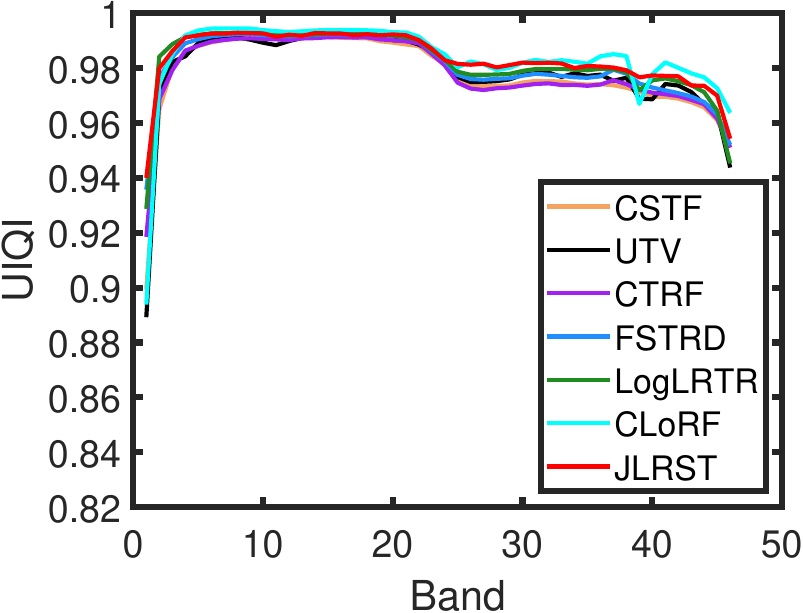 }\vspace{1pt}
    \end{minipage}
}
\caption{The plots of PSNR (top) and UIQI (bottom) values versus the spectral band for four test images. (a) Pavia University, (b) Indian Pines, (c) Balloons, (d) University of Houston.}
\label{figure2}
\end{figure*}

To visually compare the fusion effects of different approaches, we present the error images of a specific band generated by various fusion methods for each test image, along with pseudo-color images fused from three bands.
The error image directly illustrates the deviation between the fused image and the ground truth. The more blue color appears in the error image, the closer it aligns with the ground truth, indicating a superior fusion quality. Moreover, we extract and enlarge the key subregion in each pseudo-color fusion image to enhance visual comparison.

Fig.~\ref{figure3} presents the pseudo-color images of ``Pavia University" generated using bands 50, 66, and 69 through different fusion methods. In addition, the error images for band 30 are also included. From the enlarged subregions of the pseudo-color images, CTRF and LogLRTR exhibit substantial artifacts. Although there is no significant difference between FSTRD and our approach in the enlarged subregion, the error images reveal that the JLRST yields smaller error, and its fusion result is closer to the ground truth.

\begin{figure*}
\centering
\subfigure[]{
    \begin{minipage}[b]{0.1\linewidth}
    \includegraphics[width=2cm]{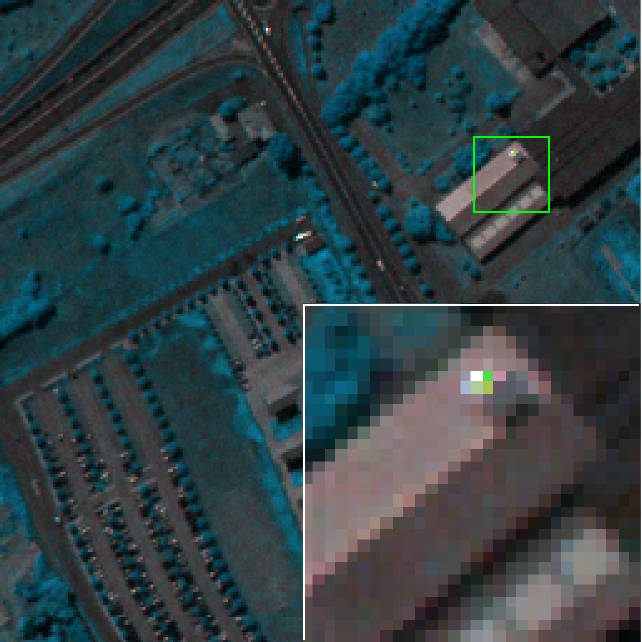}\vspace{1pt}
    \includegraphics[width=2cm]{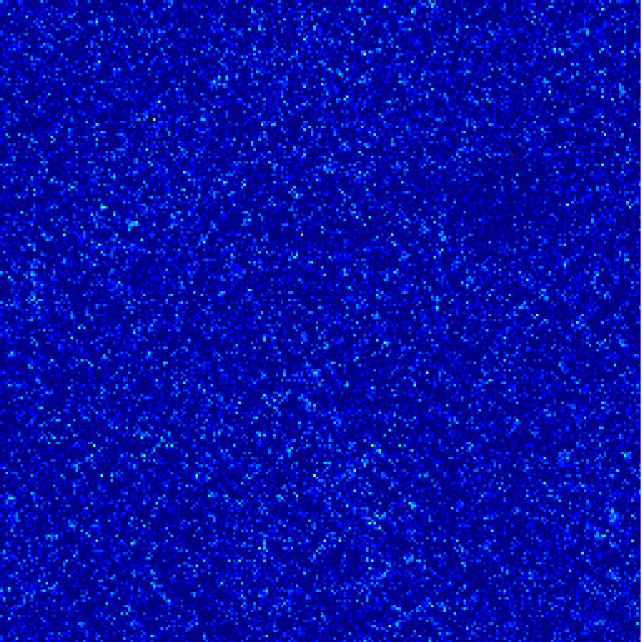}\vspace{1pt}
    \end{minipage}
}
\subfigure[]{
    \begin{minipage}[b]{0.1\linewidth}
    \includegraphics[width=2cm]{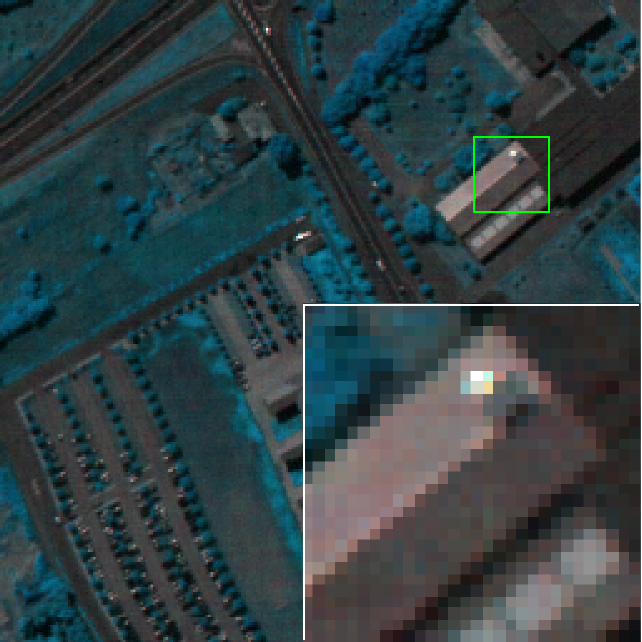}\vspace{1pt}
    \includegraphics[width=2cm]{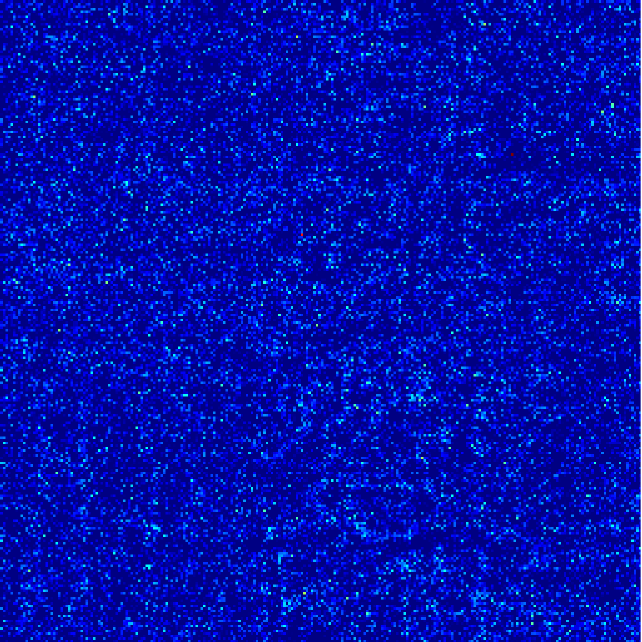}\vspace{1pt}
    \end{minipage}
}
\subfigure[]{
    \begin{minipage}[b]{0.1\linewidth}
    \includegraphics[width=2cm]{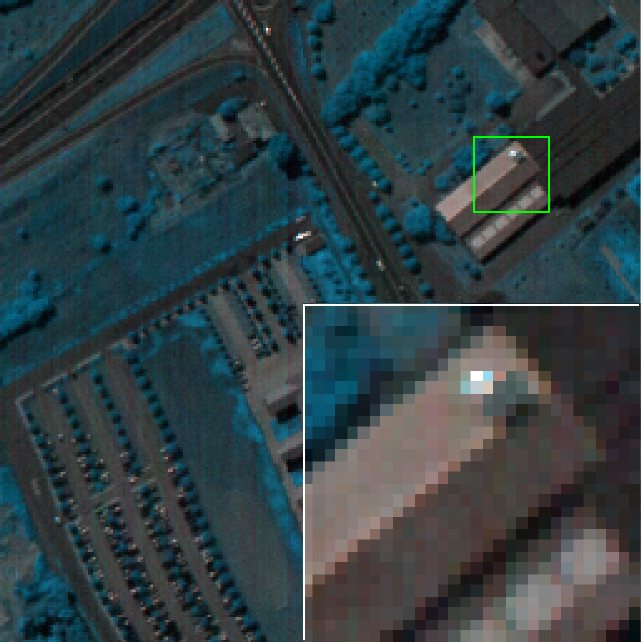}\vspace{1pt}
    \includegraphics[width=2cm]{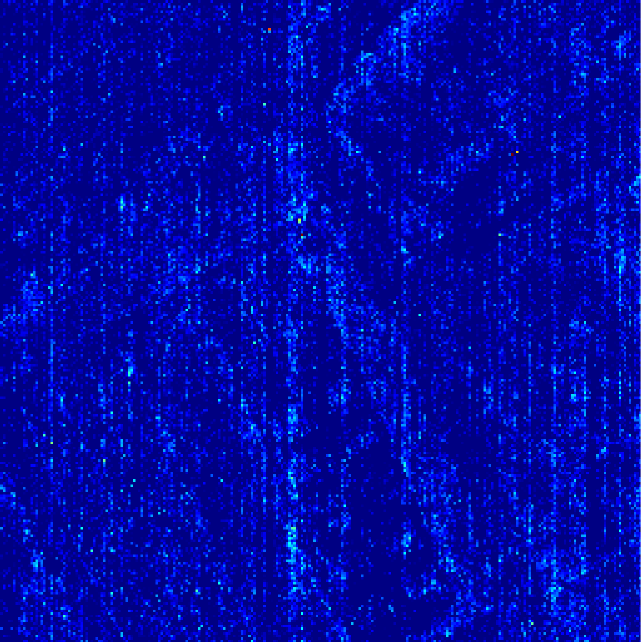}\vspace{1pt}
    \end{minipage}
}
\subfigure[]{
    \begin{minipage}[b]{0.1\linewidth}
    \includegraphics[width=2cm]{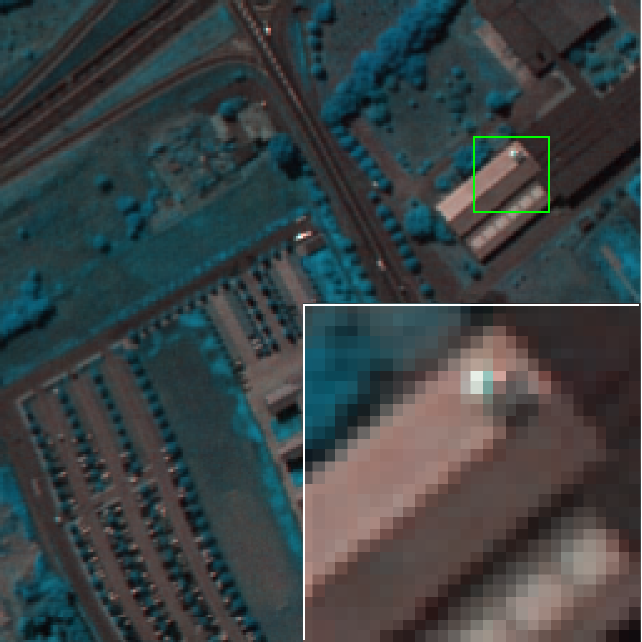}\vspace{1pt}
    \includegraphics[width=2cm]{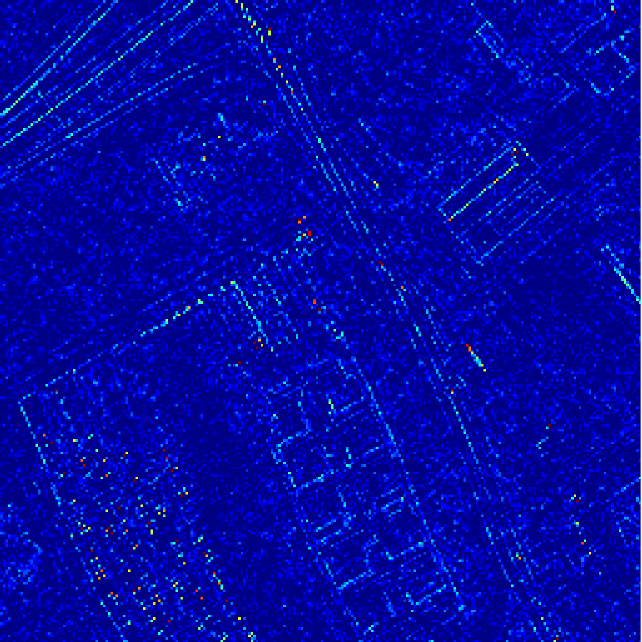}\vspace{1pt}
    \end{minipage}
}
\subfigure[]{
    \begin{minipage}[b]{0.1\linewidth}
    \includegraphics[width=2cm]{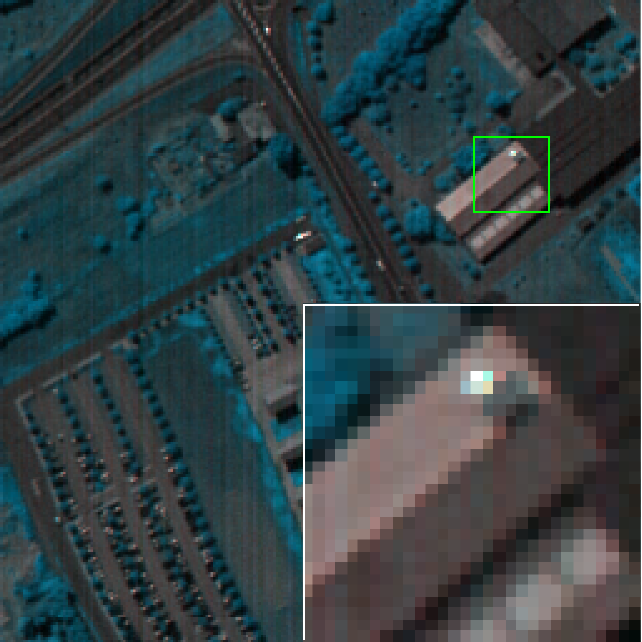}\vspace{1pt}
    \includegraphics[width=2cm]{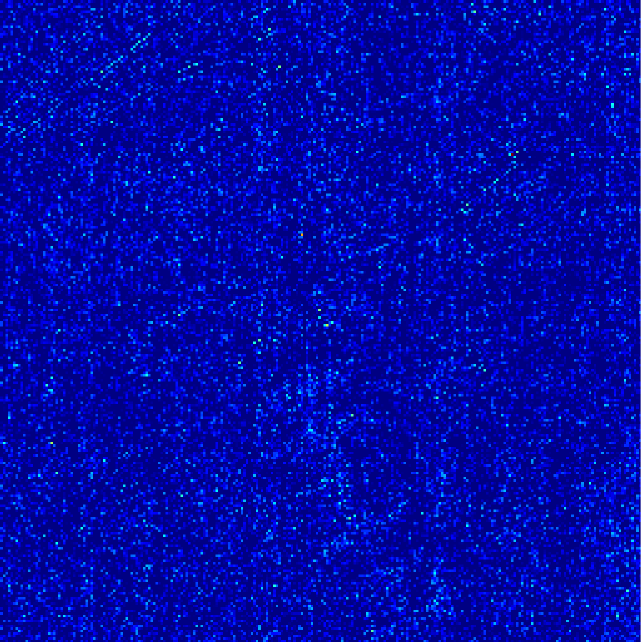}\vspace{1pt}
    \end{minipage}
}
\subfigure[]{
    \begin{minipage}[b]{0.1\linewidth}
    \includegraphics[width=2cm]{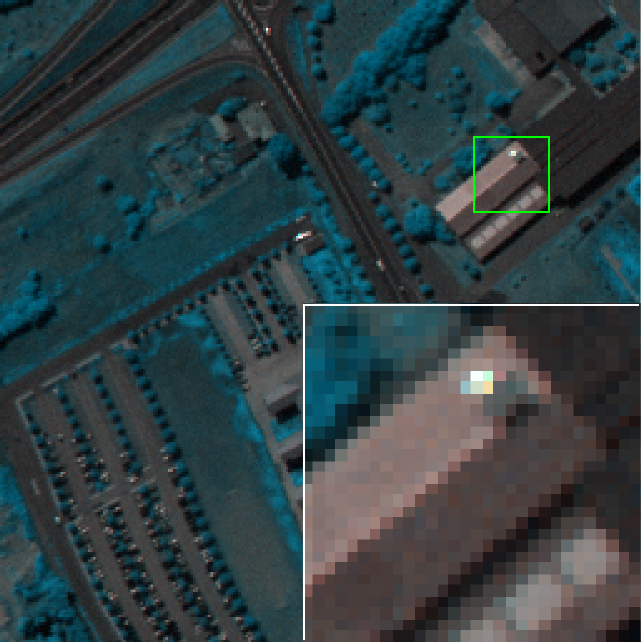}\vspace{1pt}
    \includegraphics[width=2cm]{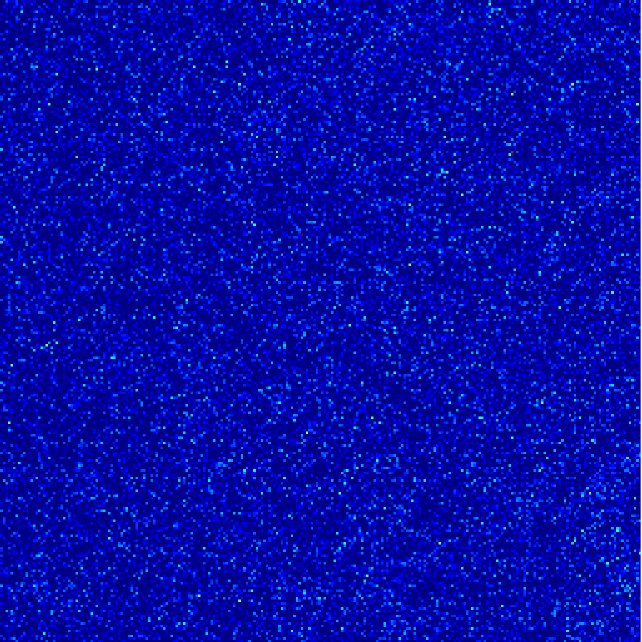}\vspace{1pt}
    \end{minipage}
}
\subfigure[]{
    \begin{minipage}[b]{0.1\linewidth}
    \includegraphics[width=2cm]{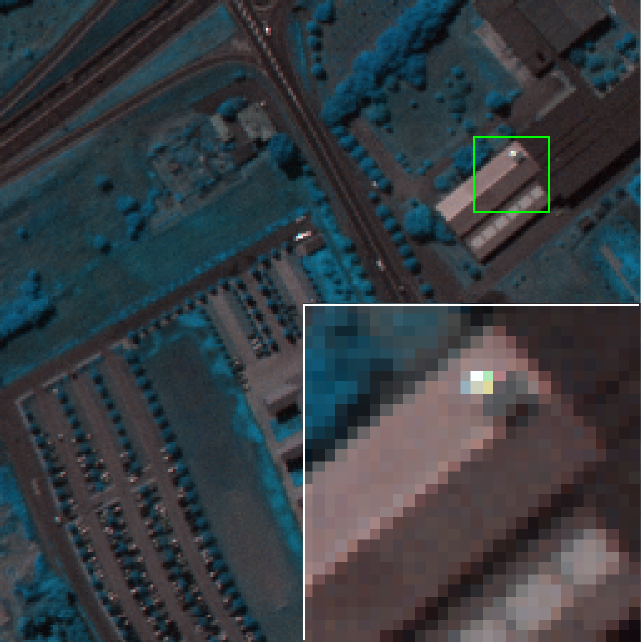}\vspace{1pt}
    \includegraphics[width=2cm]{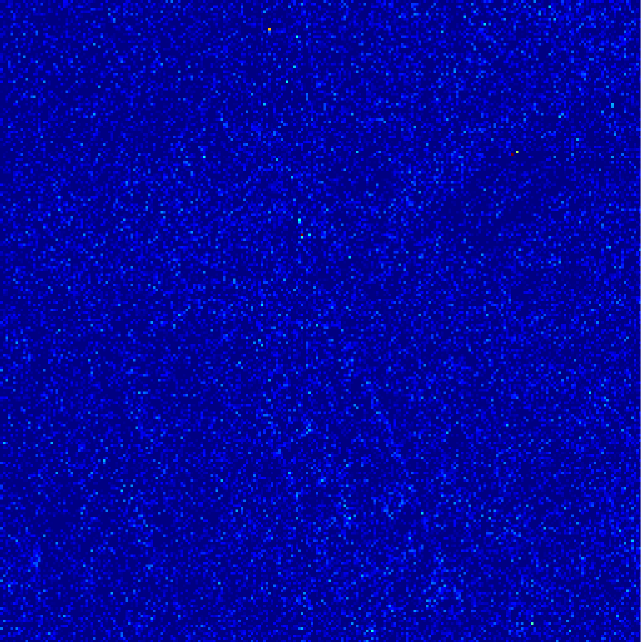}\vspace{1pt}
    \end{minipage}
}
\subfigure[]{
    \begin{minipage}[b]{0.1\linewidth}
    \includegraphics[width=2cm]{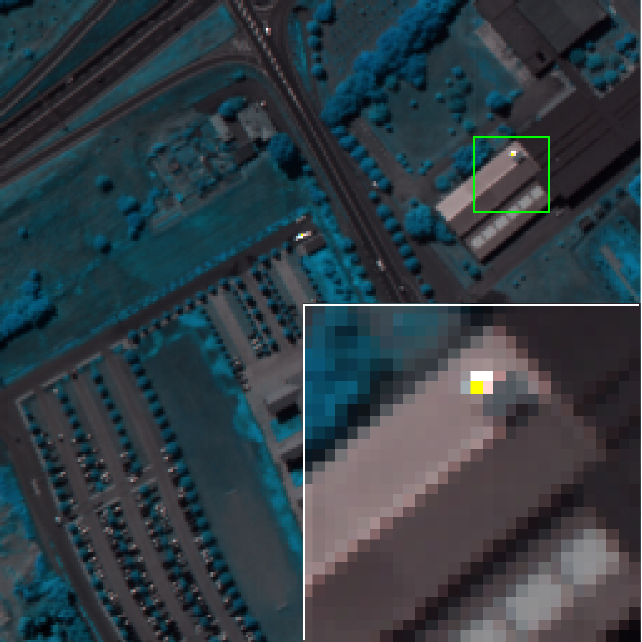}\vspace{1pt}
    \includegraphics[width=2cm]{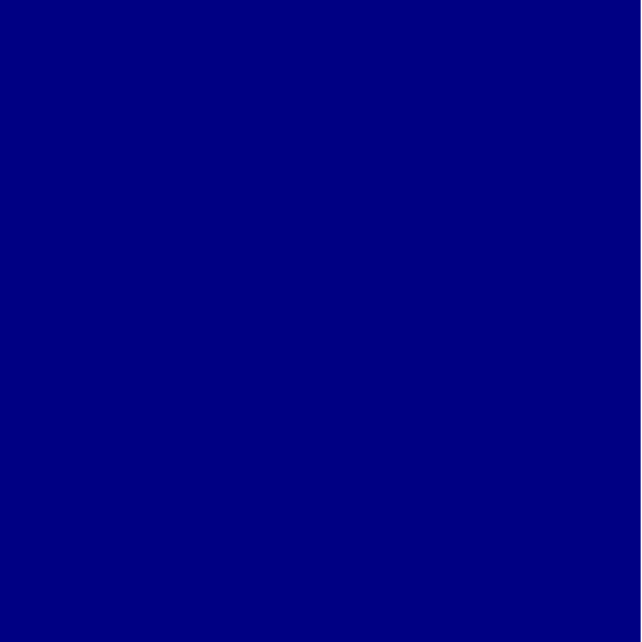}\vspace{1pt}
    \end{minipage}
}\vspace{-0.6em}
    \includegraphics[width=12cm]{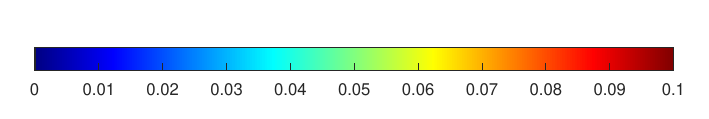}\vspace{1pt}\vspace{-1em}
\caption{The first row represents the false color images of ``Pavia University" generated by bands (R: 50, G: 66, and B: 69), and the second row displays the error images of the 30th band. (a) CSTF, (b) UTV, (c) CTRF, (d) FSTRD, (e) LogLRTR, (f) CLoRF, (g) JLRST, (h) ground truth.}
\label{figure3}
\end{figure*}

Fig.~\ref{figure4} displays pseudo-color images obtained from the 60th, 82nd, and 92nd bands of the ``Indian Pines", as well as error images corresponding to the 149th band. As shown in Fig.~\ref{figure4}, UTV and CTRF have poor performances, and the fusion results of FSTRD and LogLRTR appear staircase effects. In addition, CLoRF exists with spectral distortion, which results in noticeable color deviations in the reconstructed images compared to the ground truth. Observing the magnified subregions of the fused images, it is evident that the reconstructed image acquired by JLRST is smoother than other methods.
Additionally, the result obtained using the JLRST demonstrates superior performance in detail preservation, and its fusion result is closer to the original image.

\begin{figure*}
\centering
\subfigure[]{
    \begin{minipage}[b]{0.1\linewidth}
    \includegraphics[width=2cm]{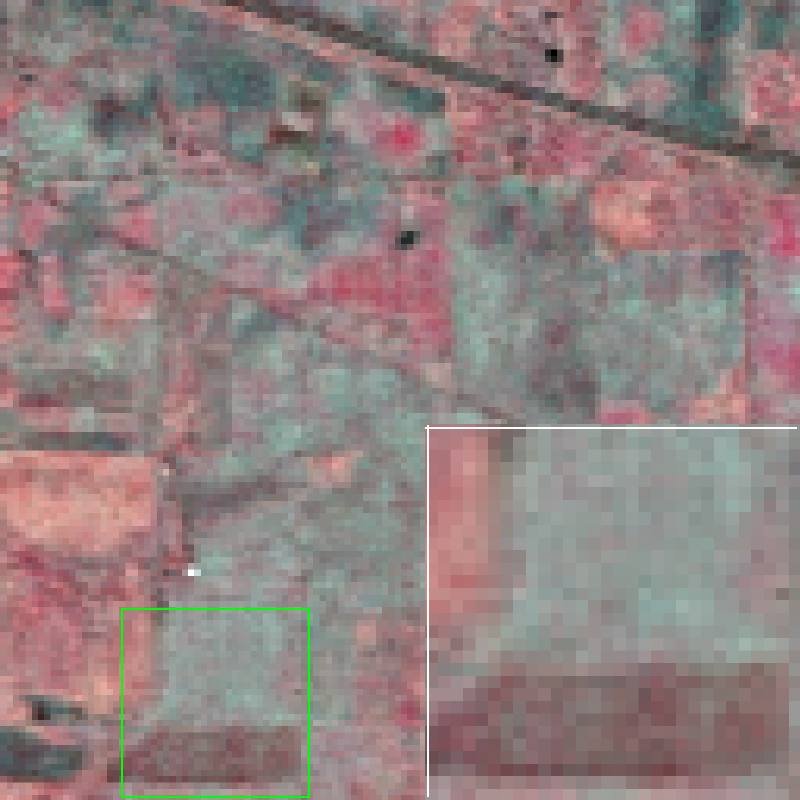}\vspace{1pt}
    \includegraphics[width=2cm]{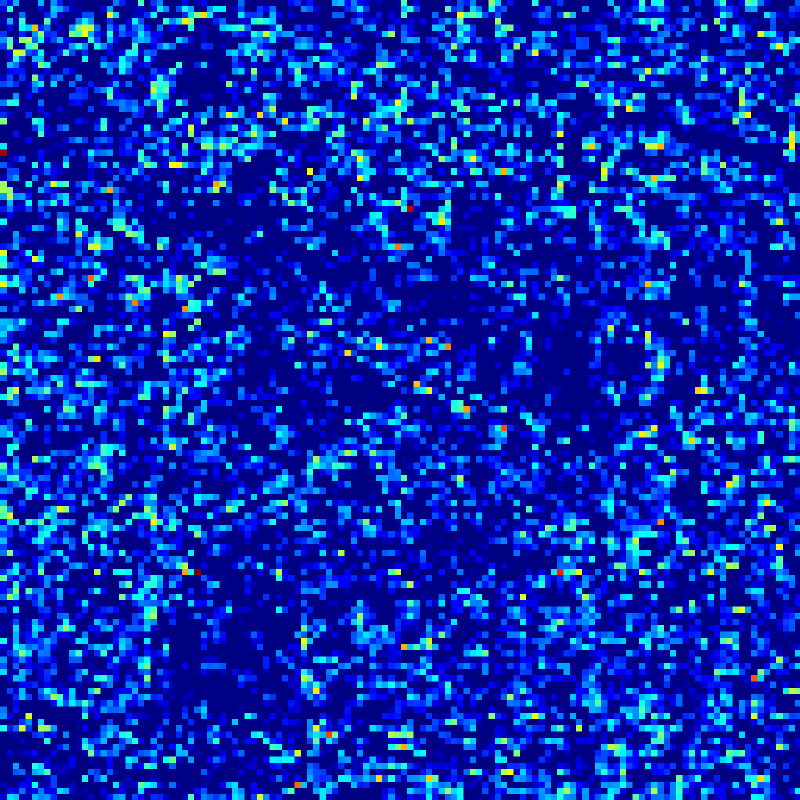}\vspace{1pt}
    \end{minipage}
}
\subfigure[]{
    \begin{minipage}[b]{0.1\linewidth}
    \includegraphics[width=2cm]{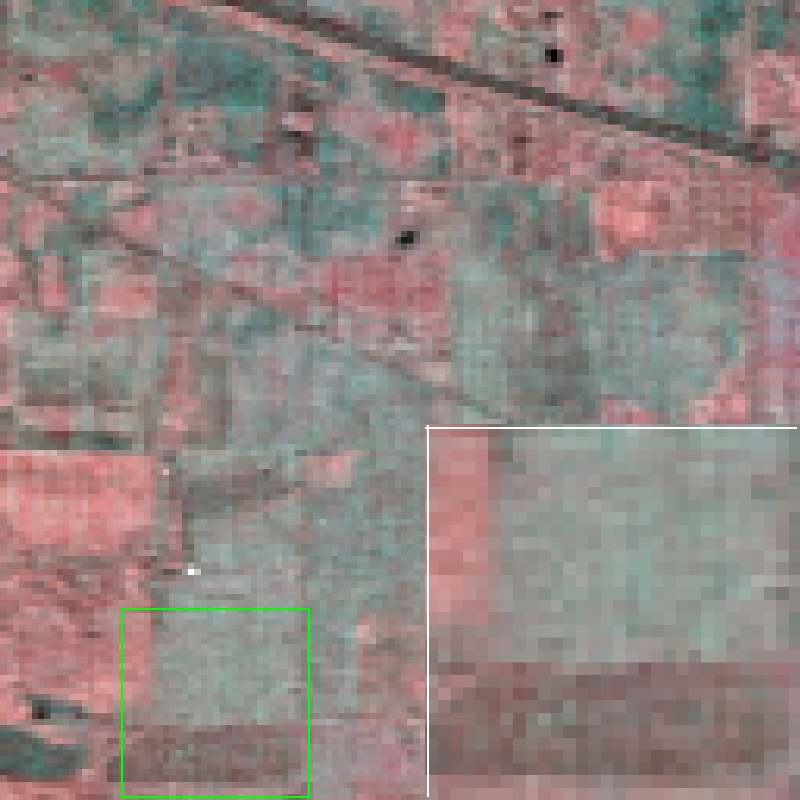}\vspace{1pt}
    \includegraphics[width=2cm]{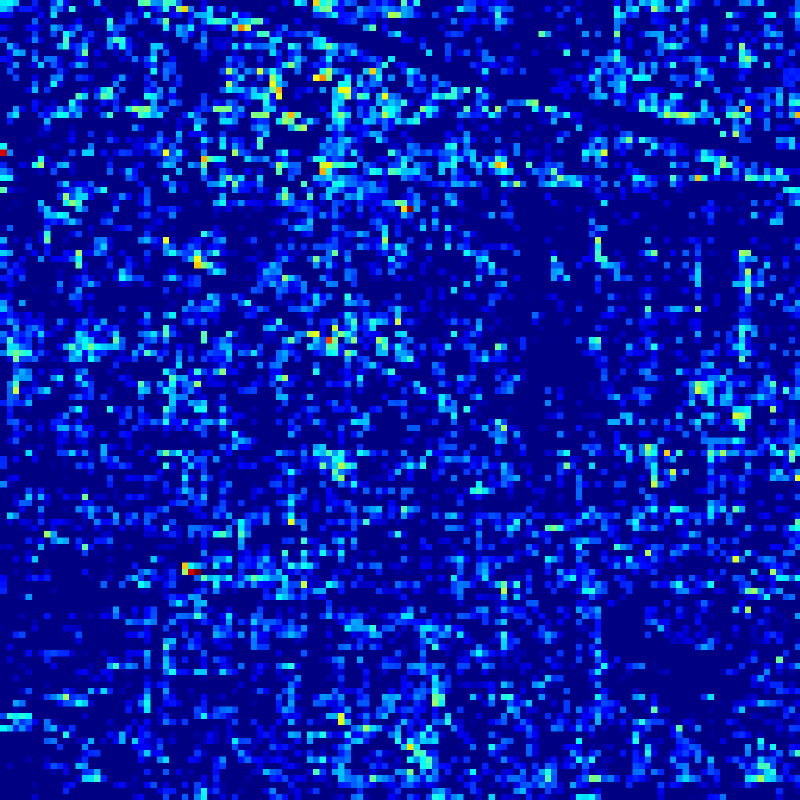}\vspace{1pt}
    \end{minipage}
}
\subfigure[]{
    \begin{minipage}[b]{0.1\linewidth}
    \includegraphics[width=2cm]{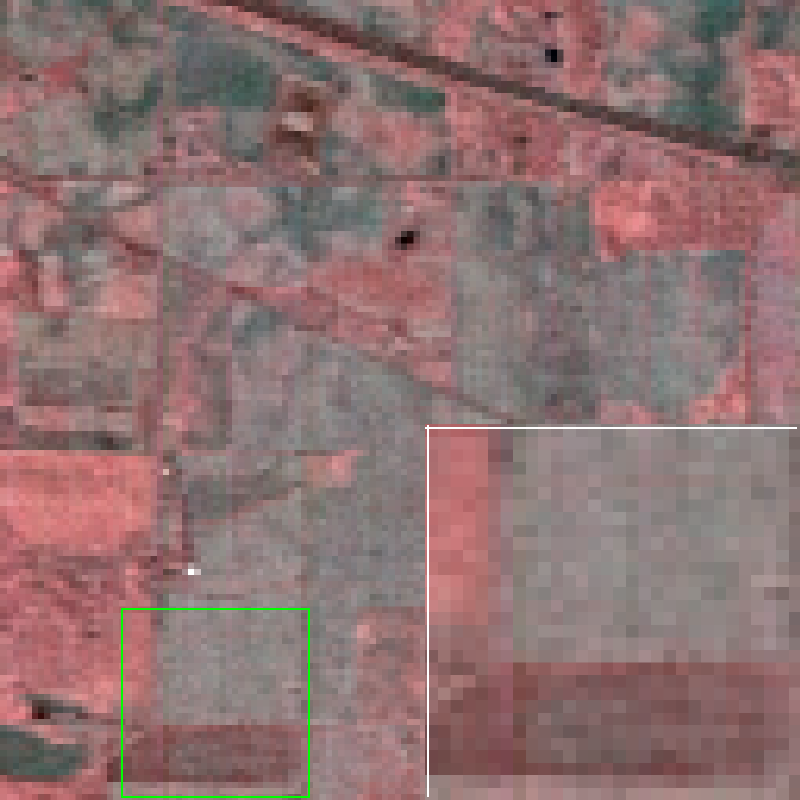}\vspace{1pt}
    \includegraphics[width=2cm]{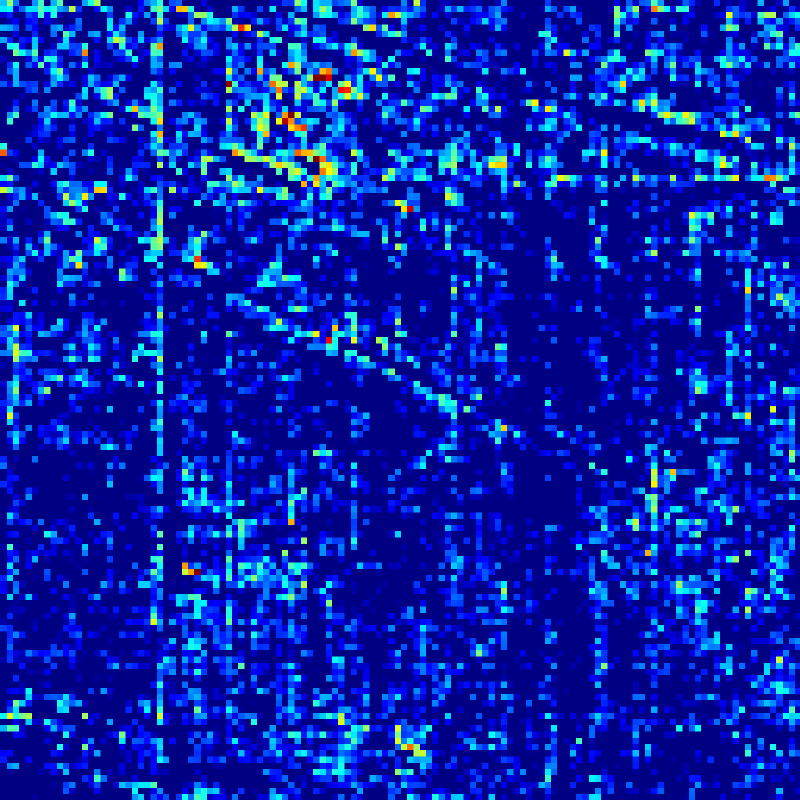}\vspace{1pt}
    \end{minipage}
}
\subfigure[]{
    \begin{minipage}[b]{0.1\linewidth}
    \includegraphics[width=2cm]{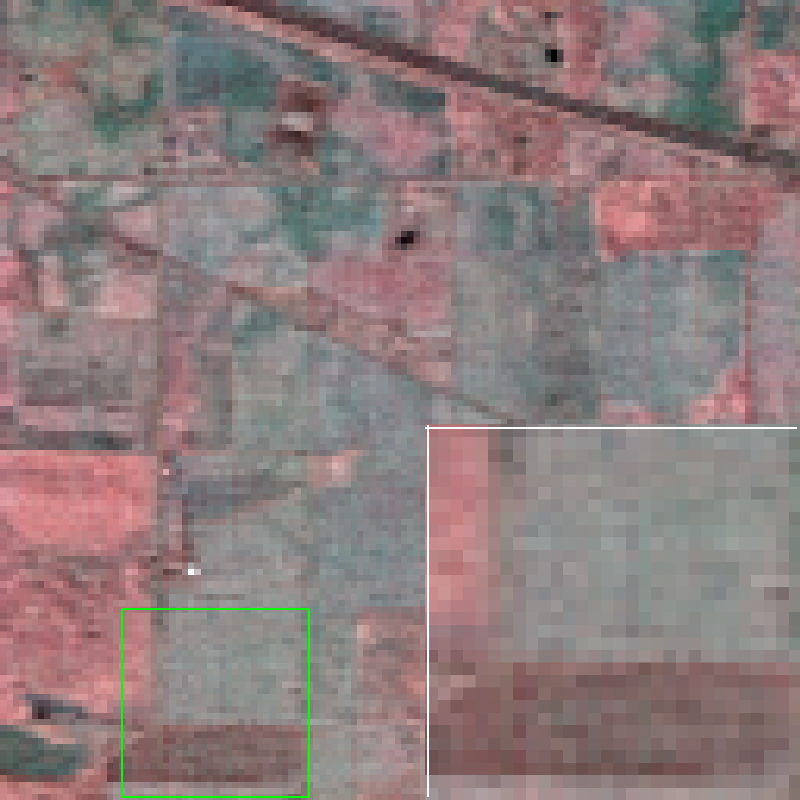}\vspace{1pt}
    \includegraphics[width=2cm]{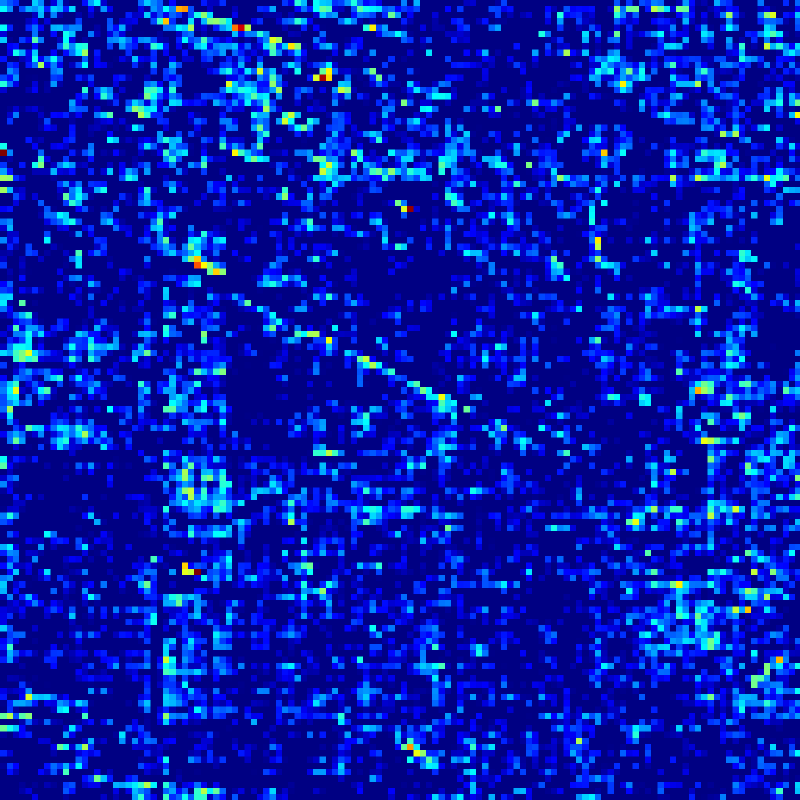}\vspace{1pt}
    \end{minipage}
}
\subfigure[]{
    \begin{minipage}[b]{0.1\linewidth}
    \includegraphics[width=2cm]{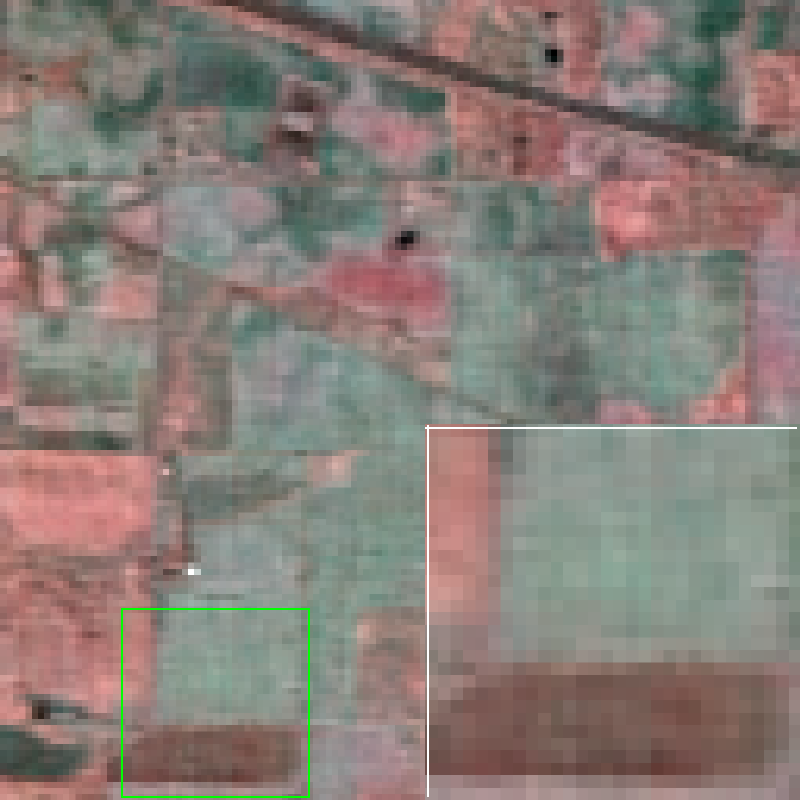}\vspace{1pt}
    \includegraphics[width=2cm]{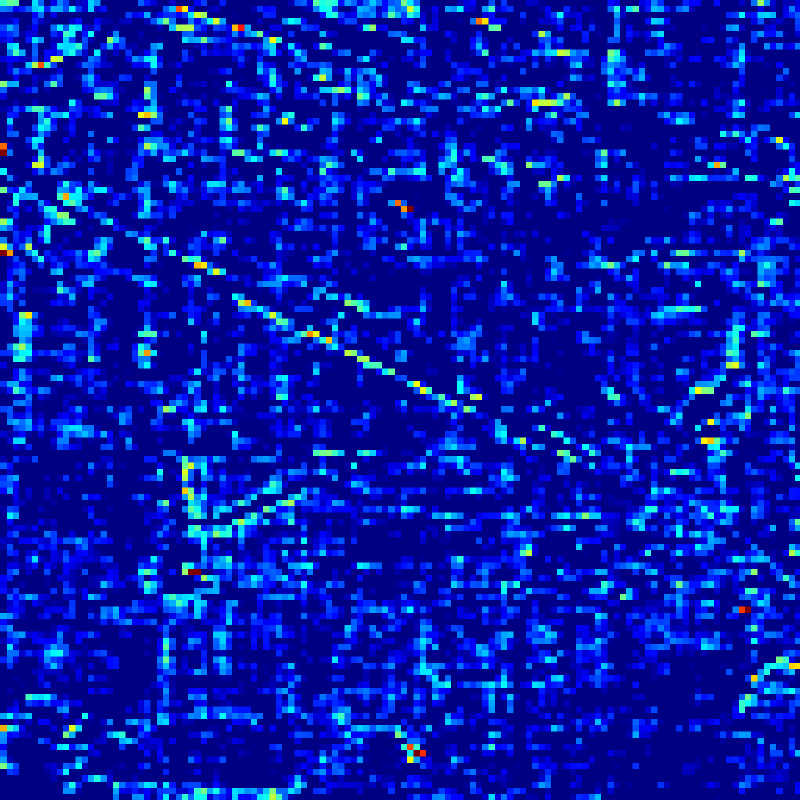}\vspace{1pt}
    \end{minipage}
}
\subfigure[]{
    \begin{minipage}[b]{0.1\linewidth}
    \includegraphics[width=2cm]{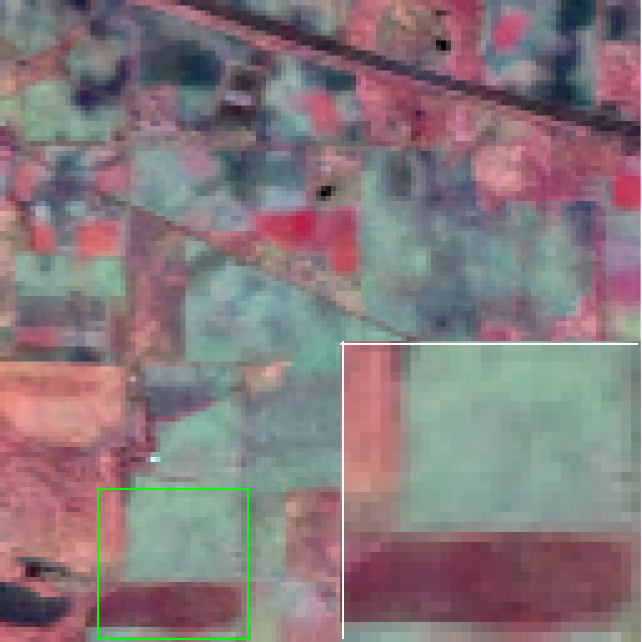}\vspace{1pt}
    \includegraphics[width=2cm]{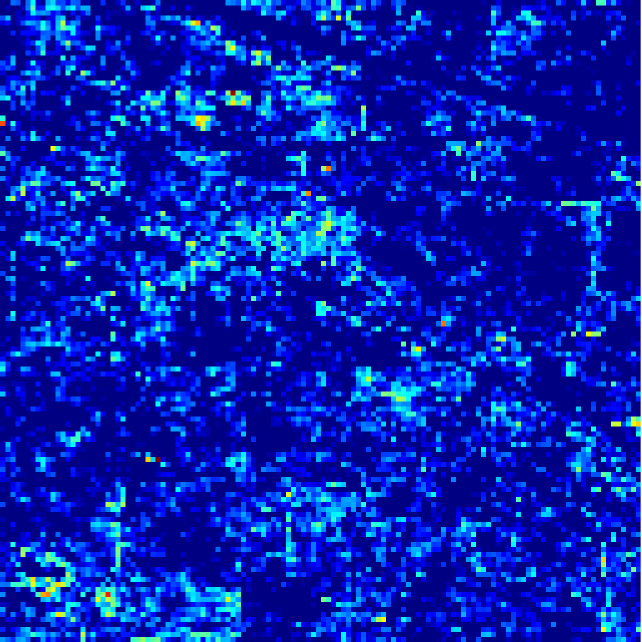}\vspace{1pt}
    \end{minipage}
}
\subfigure[]{
    \begin{minipage}[b]{0.1\linewidth}
    \includegraphics[width=2cm]{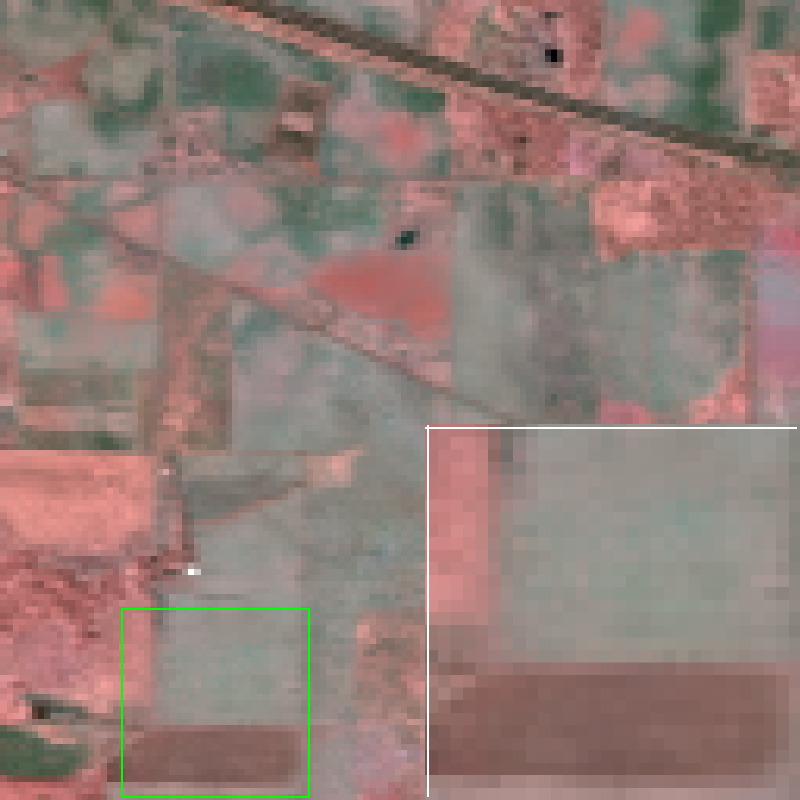}\vspace{1pt}
    \includegraphics[width=2cm]{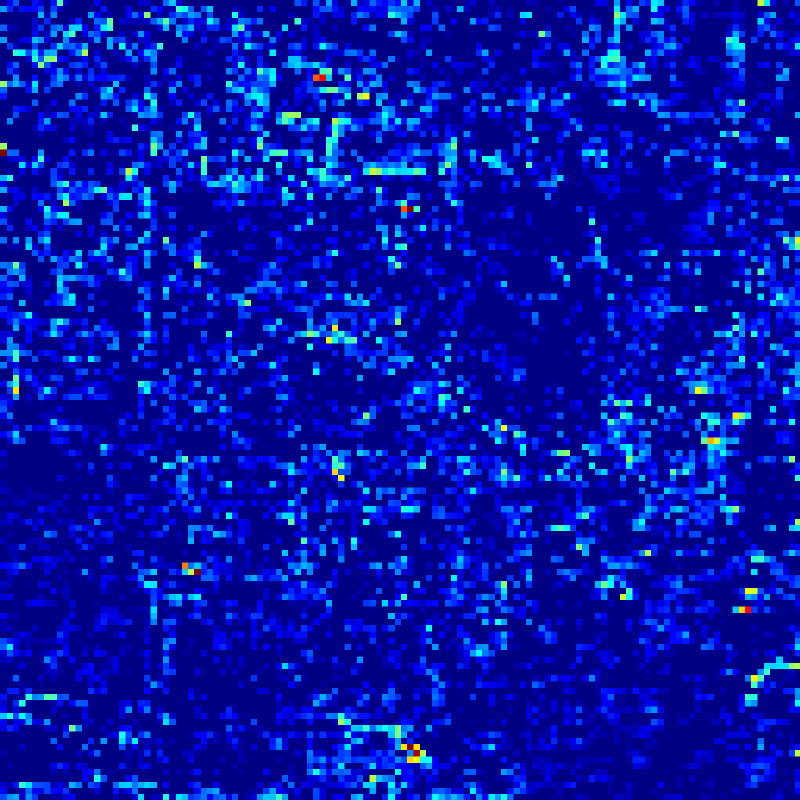}\vspace{1pt}
    \end{minipage}
}
\subfigure[]{
    \begin{minipage}[b]{0.1\linewidth}
    \includegraphics[width=2cm]{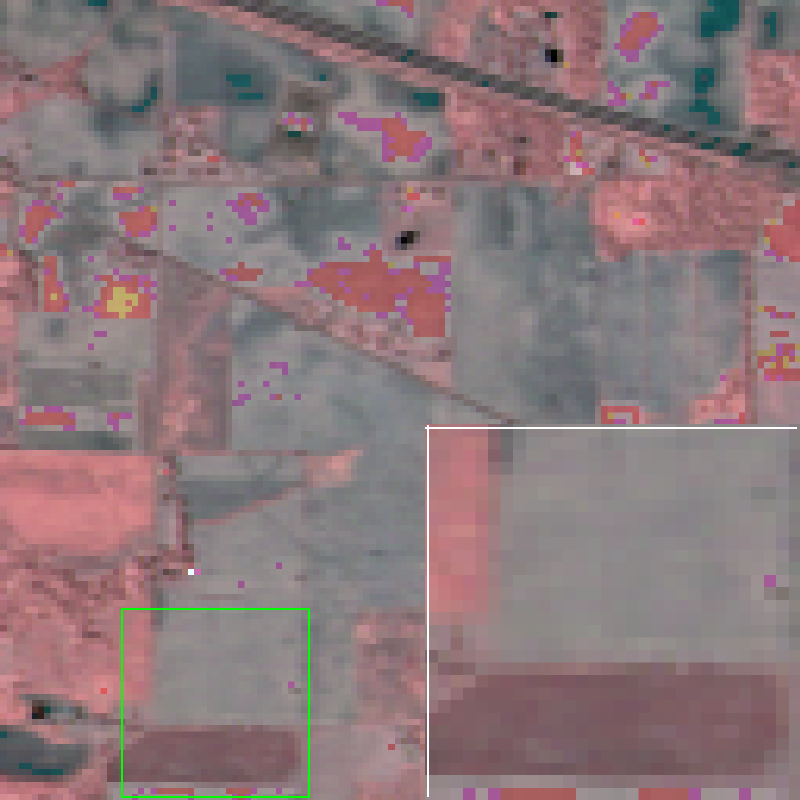}\vspace{1pt}
    \includegraphics[width=2cm]{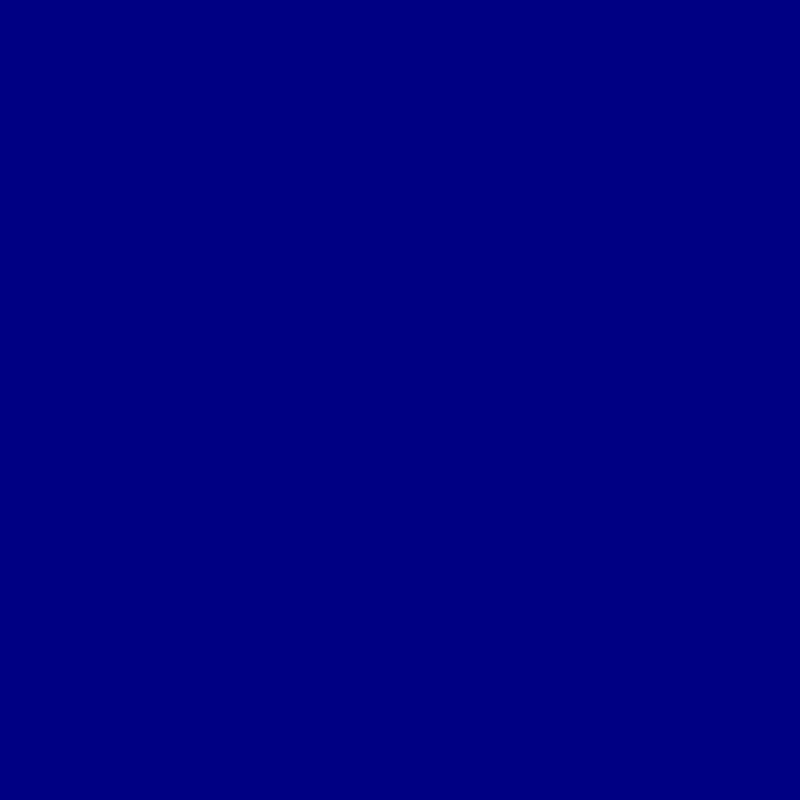}\vspace{1pt}
    \end{minipage}
}\vspace{-0.6em}
   \includegraphics[width=12cm]{Color_bar_0.1.pdf}\vspace{1pt}\vspace{-1em}
\caption{The first row shows the false color images of ``Indian Pines" generated using bands (R: 60, G: 82, and B: 92), and the second row displays the error images of the 149th band. (a) CSTF, (b) UTV, (c) CTRF, (d) FSTRD, (e) LogLRTR, (f) CLoRF, (g) JLRST, (h) ground truth.}
\label{figure4}
\end{figure*}

The pseudo-color images composed of the 16th, 1st, and 25th bands from the  ``Balloons", along with the error images corresponding to the 8th band, are presented in Fig.~\ref{figure5}. As evident from the magnified subregion in this figure, all approaches except CLoRF and JLRST exhibit artifacts in the fused images. Furthermore, JLRST produces the smallest reconstruction error among all compared methods.

\begin{figure*}
\centering
\subfigure[]{
    \begin{minipage}[b]{0.1\linewidth}
    \includegraphics[width=2cm]{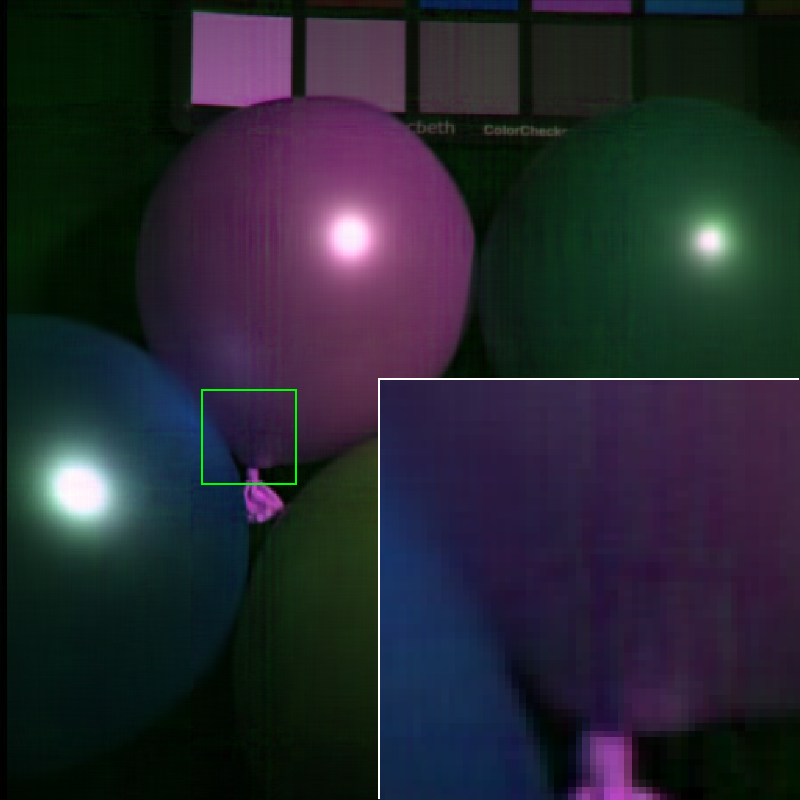}\vspace{1pt}
    \includegraphics[width=2cm]{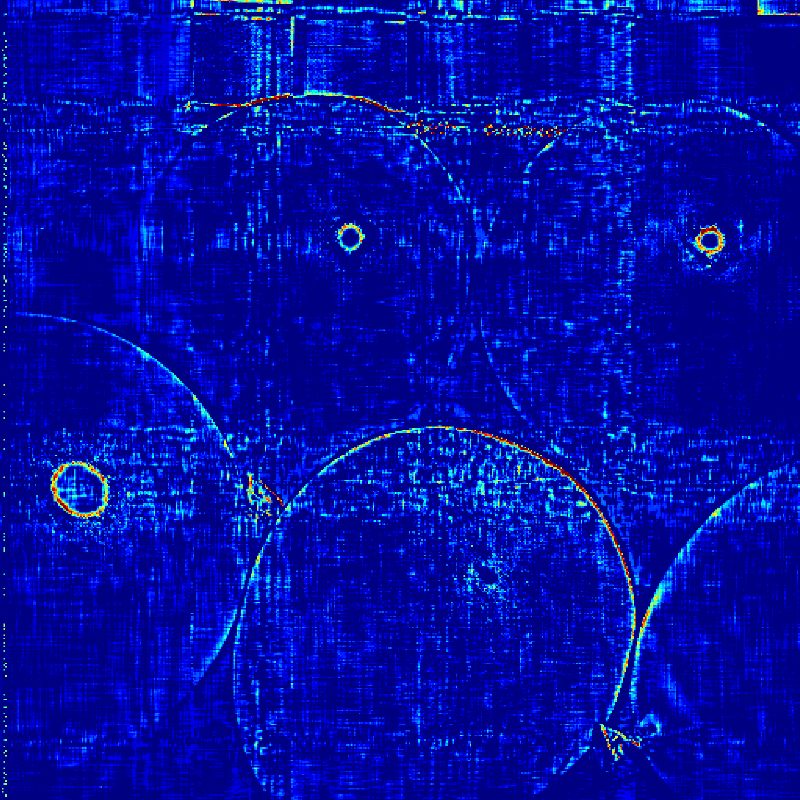}\vspace{1pt}
    \end{minipage}
}
\subfigure[]{
    \begin{minipage}[b]{0.1\linewidth}
    \includegraphics[width=2cm]{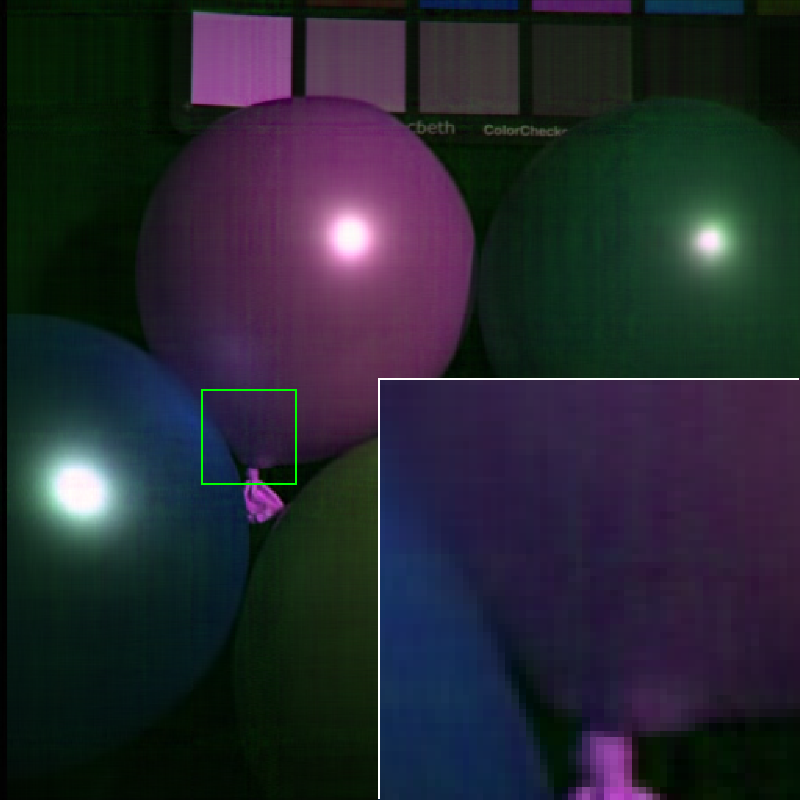}\vspace{1pt}
    \includegraphics[width=2cm]{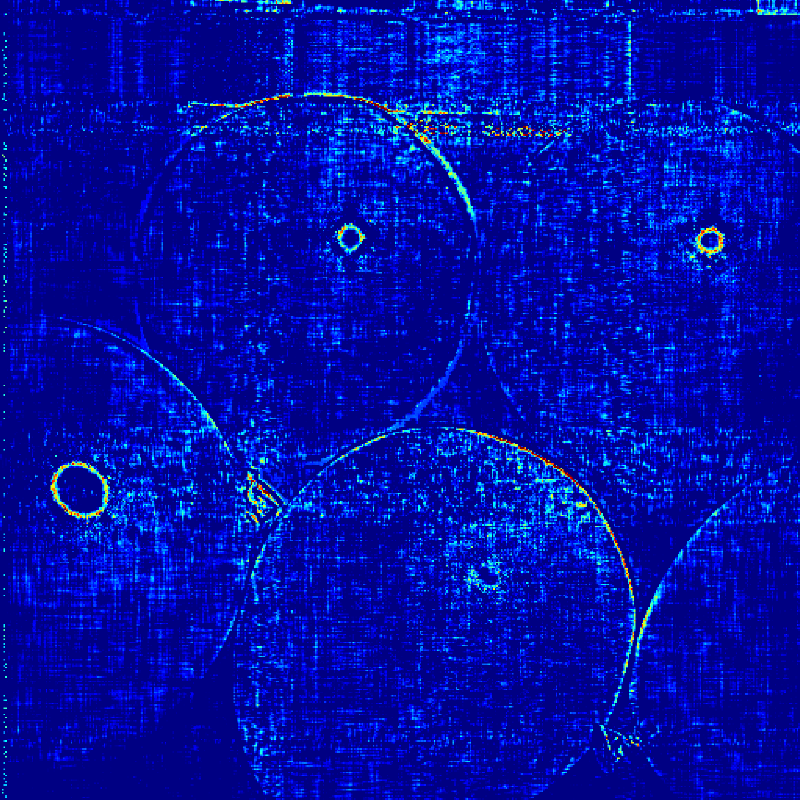}\vspace{1pt}
    \end{minipage}
}
\subfigure[]{
    \begin{minipage}[b]{0.1\linewidth}
    \includegraphics[width=2cm]{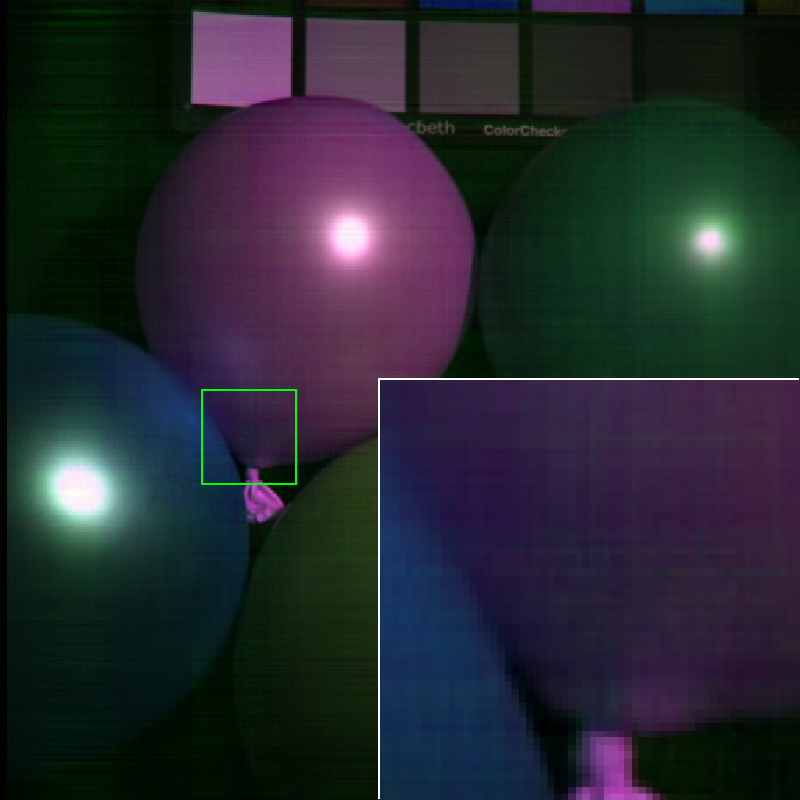}\vspace{1pt}
    \includegraphics[width=2cm]{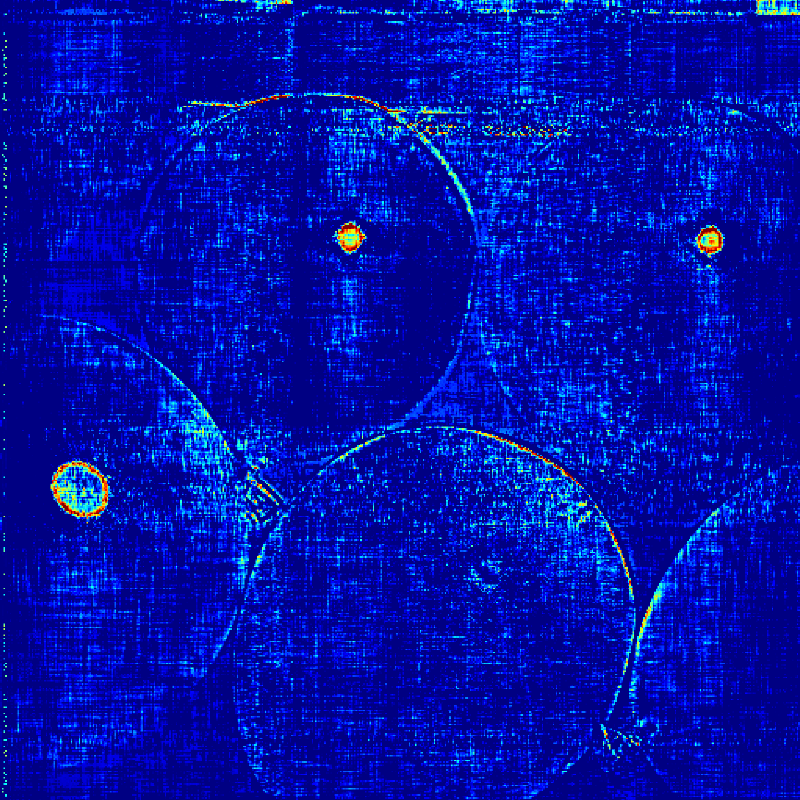}\vspace{1pt}
    \end{minipage}
}
\subfigure[]{
    \begin{minipage}[b]{0.1\linewidth}
    \includegraphics[width=2cm]{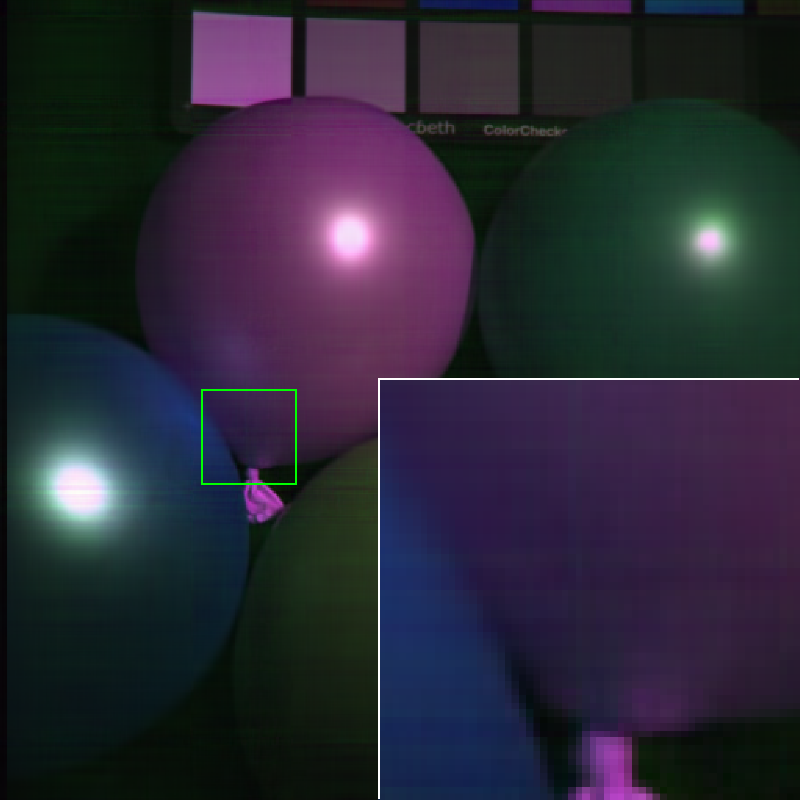}\vspace{1pt}
    \includegraphics[width=2cm]{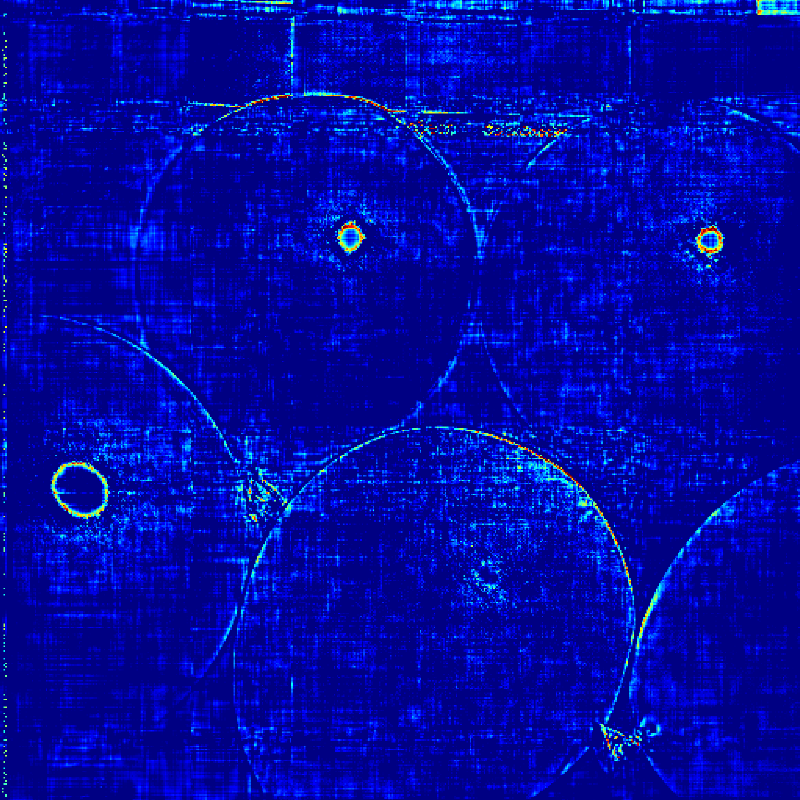}\vspace{1pt}
    \end{minipage}
}
\subfigure[]{
    \begin{minipage}[b]{0.1\linewidth}
    \includegraphics[width=2cm]{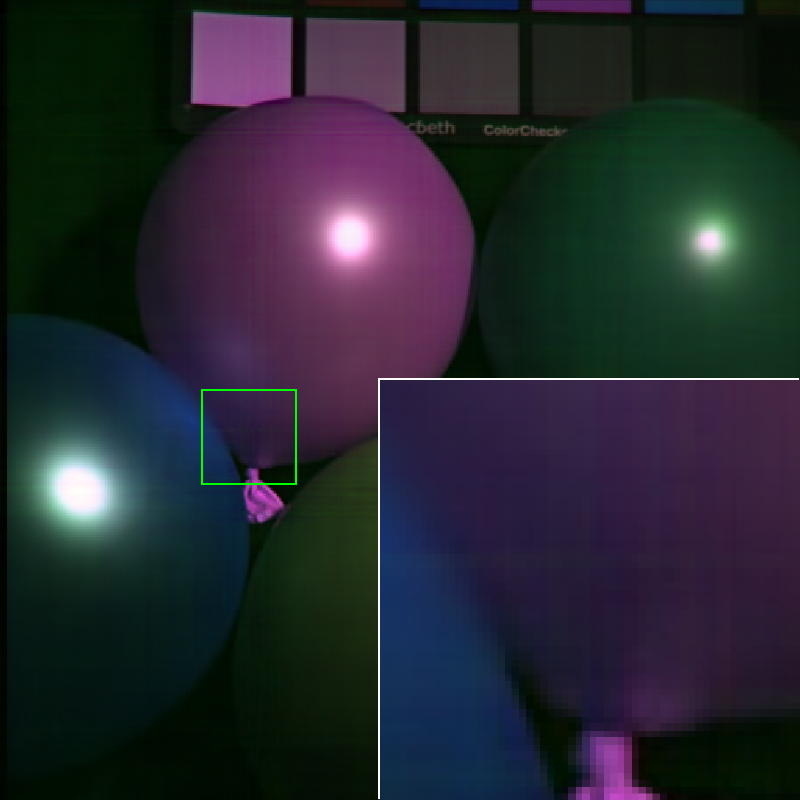}\vspace{1pt}
    \includegraphics[width=2cm]{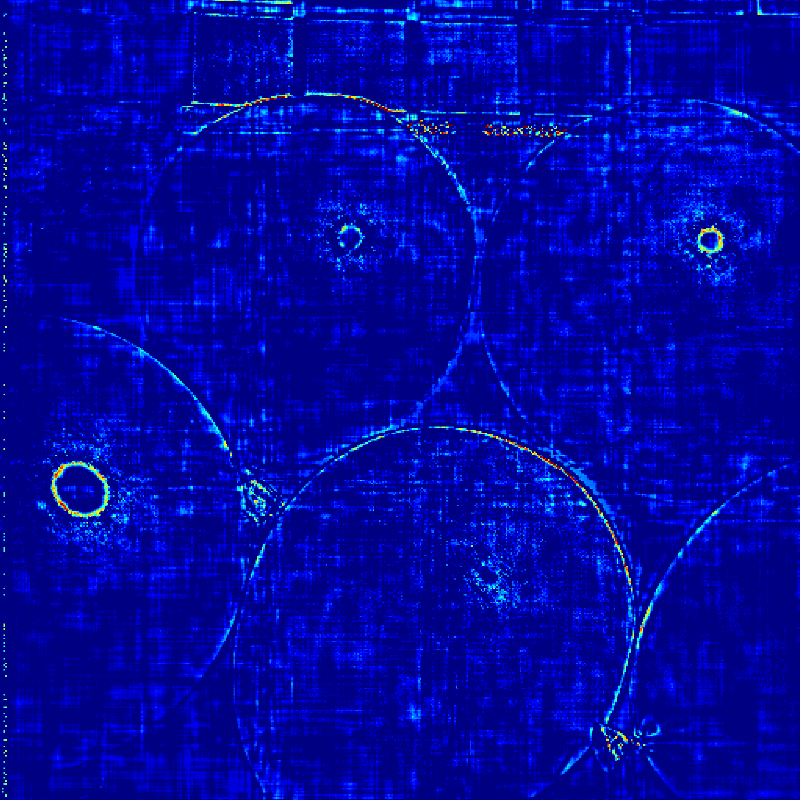}\vspace{1pt}
    \end{minipage}
}
\subfigure[]{
    \begin{minipage}[b]{0.1\linewidth}
    \includegraphics[width=2cm]{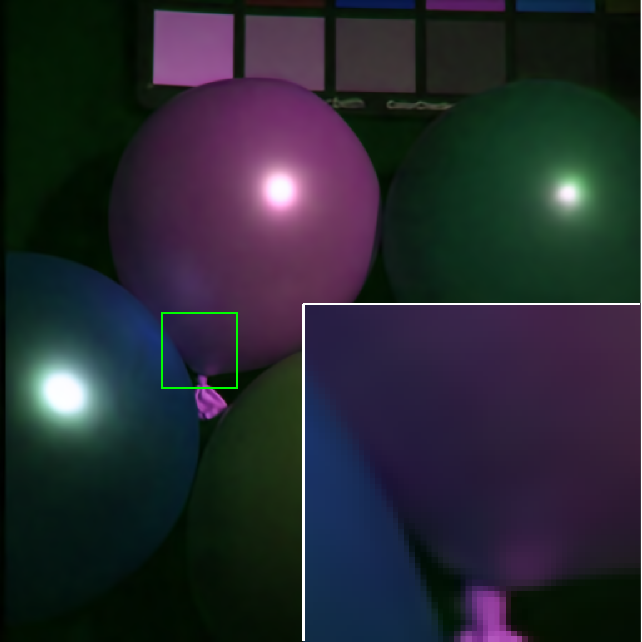}\vspace{1pt}
    \includegraphics[width=2cm]{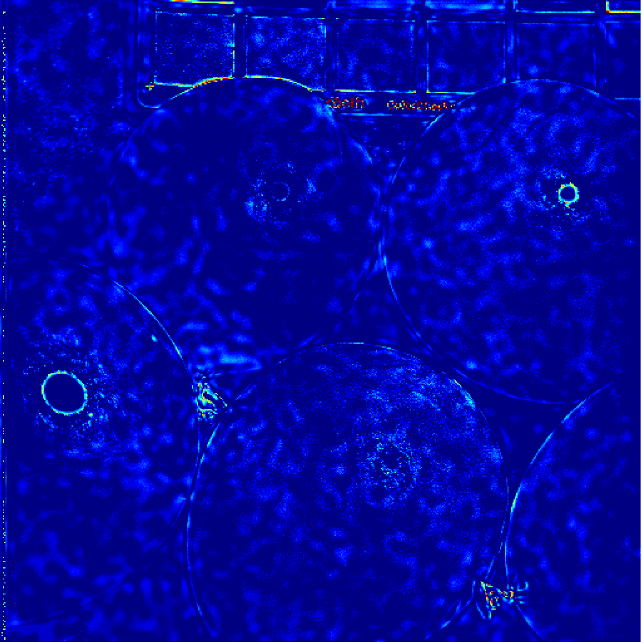}\vspace{1pt}
    \end{minipage}
}
\subfigure[]{
    \begin{minipage}[b]{0.1\linewidth}
    \includegraphics[width=2cm]{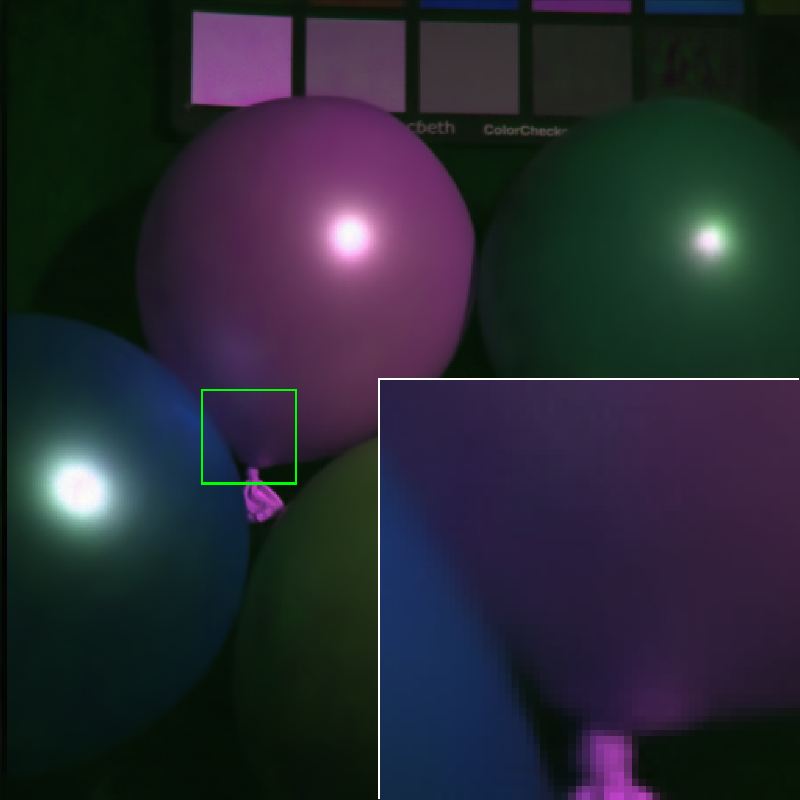}\vspace{1pt}
    \includegraphics[width=2cm]{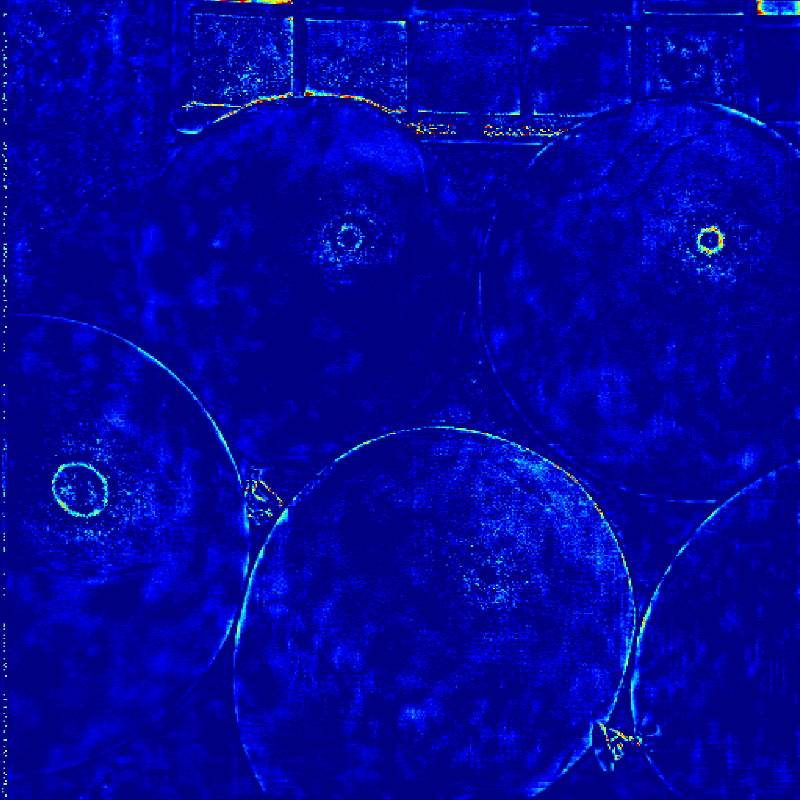}\vspace{1pt}
    \end{minipage}
}
\subfigure[]{
    \begin{minipage}[b]{0.1\linewidth}
    \includegraphics[width=2cm]{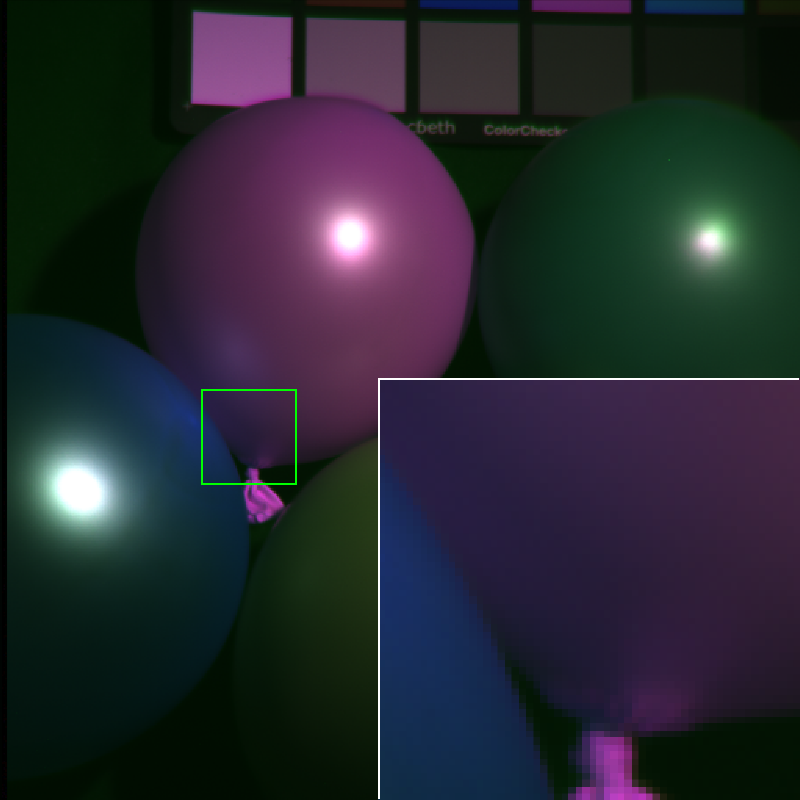}\vspace{1pt}
    \includegraphics[width=2cm]{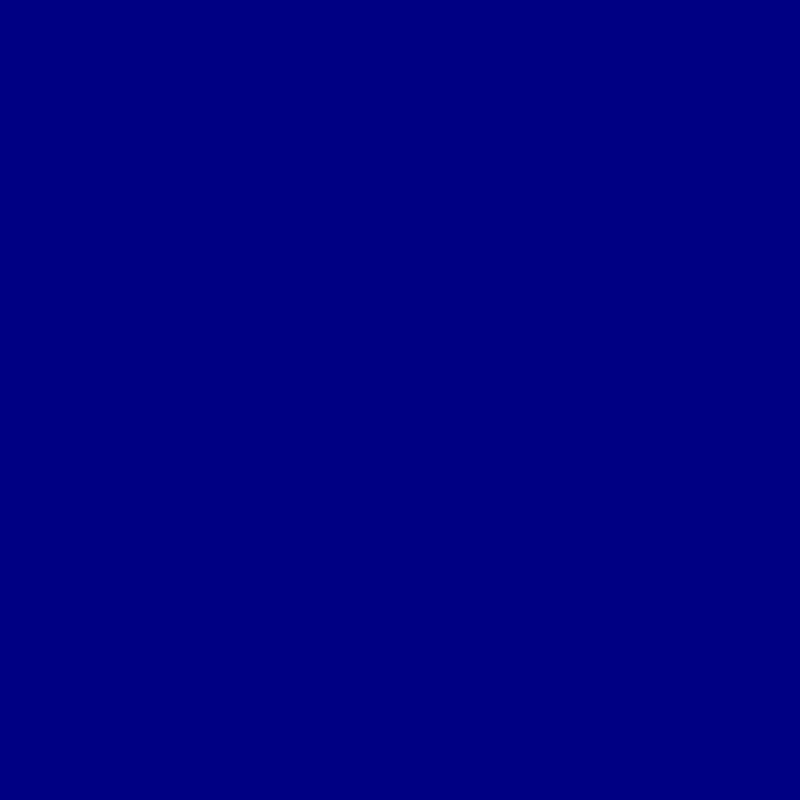}\vspace{1pt}
    \end{minipage}
}\vspace{-0.6em}
   \includegraphics[width=12cm]{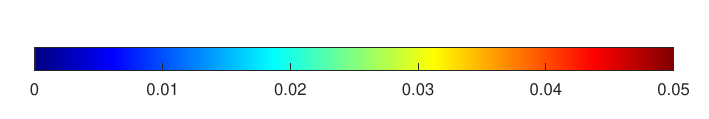}\vspace{1pt}\vspace{-1em}
\caption{The first row shows the false color images of ``Balloons" gained by bands (R: 16, G: 1, and B: 25), while the second row presents the error images of the 8th band. (a) CSTF, (b) UTV, (c) CTRF, (d) FSTRD, (e) LogLRTR, (f) CLoRF, (g) JLRST, (h) ground truth.}
\label{figure5}
\end{figure*}

The pseudo-color and error images of the ``University of Houston" are shown in Fig.~\ref{figure6}.
In the magnified regions of the pseudo-color images, noticeable staircase effects can be observed for CSTF, UTV and CTRF. However, CLoRF exhibits over-smoothing, resulting in the loss of texture details and blurred edges in the fused image. Notably, JLRST yields significantly smaller errors and demonstrates a notable advantage in preserving image details compared with other methods.

\begin{figure*}
\centering
\subfigure[]{
    \begin{minipage}[b]{0.1\linewidth}
    \includegraphics[width=2cm]{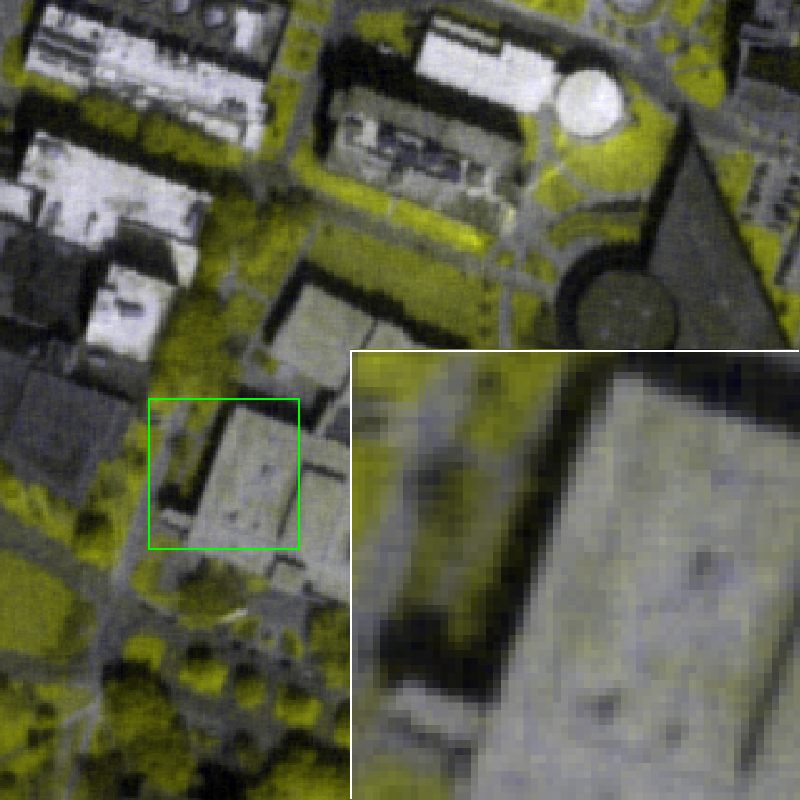}\vspace{1pt}
    \includegraphics[width=2cm]{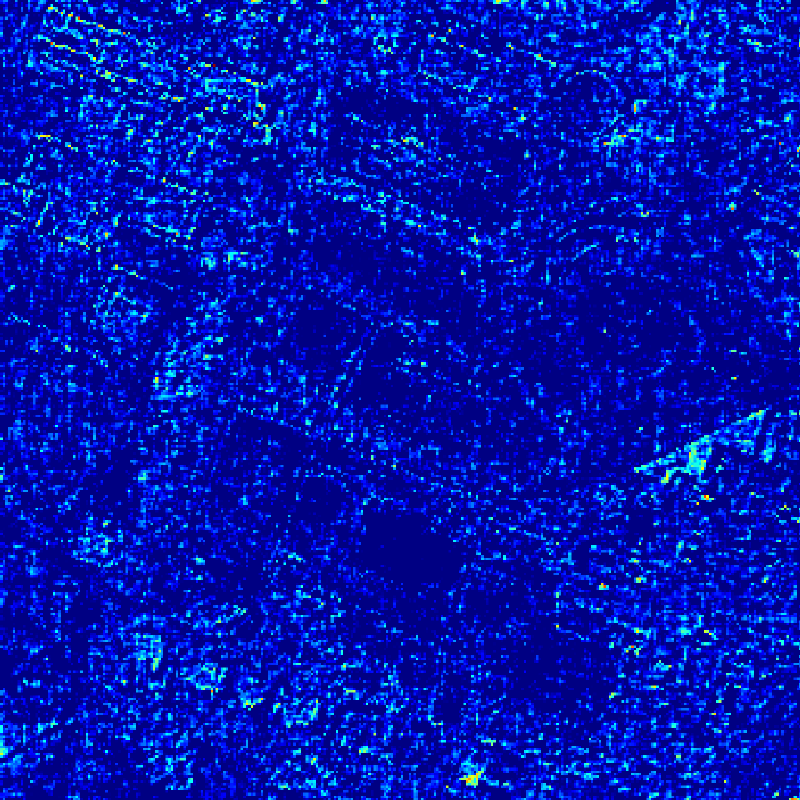}\vspace{1pt}
    \end{minipage}
}
\subfigure[]{
    \begin{minipage}[b]{0.1\linewidth}
    \includegraphics[width=2cm]{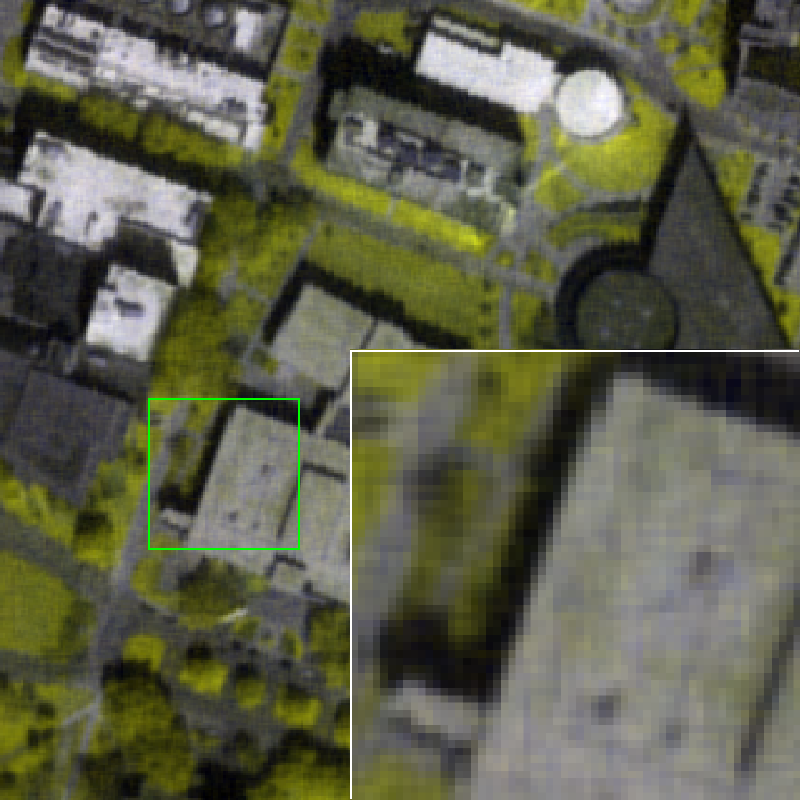}\vspace{1pt}
    \includegraphics[width=2cm]{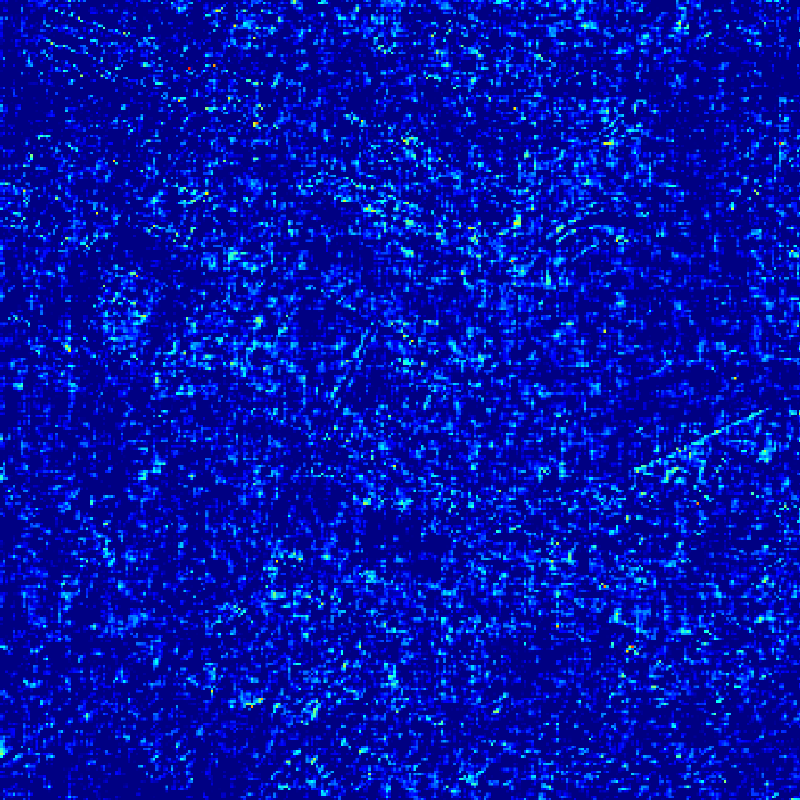}\vspace{1pt}
    \end{minipage}
 }
\subfigure[]{
    \begin{minipage}[b]{0.1\linewidth}
    \includegraphics[width=2cm]{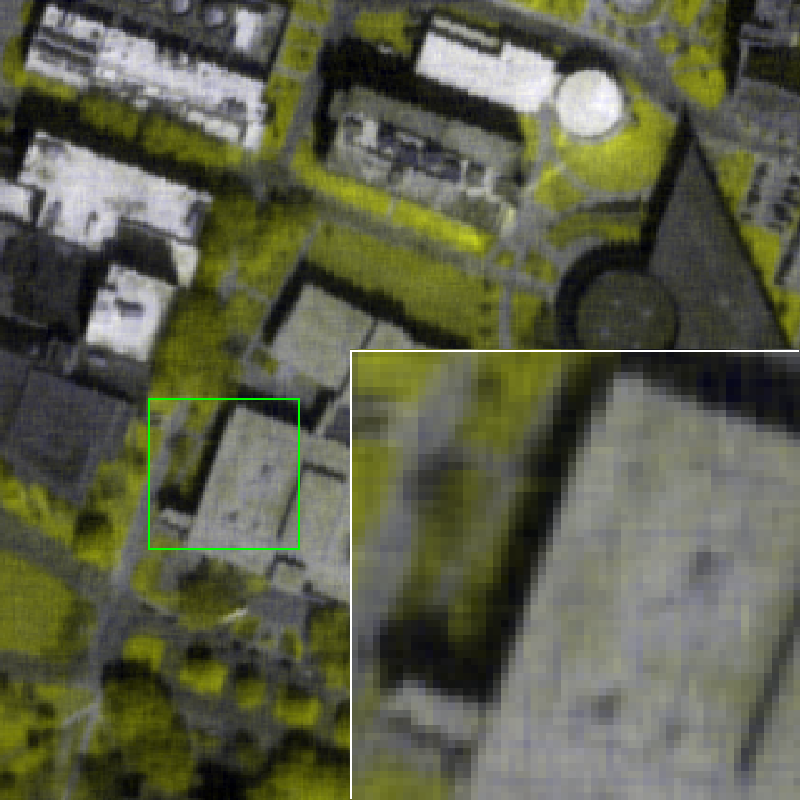}\vspace{1pt}
    \includegraphics[width=2cm]{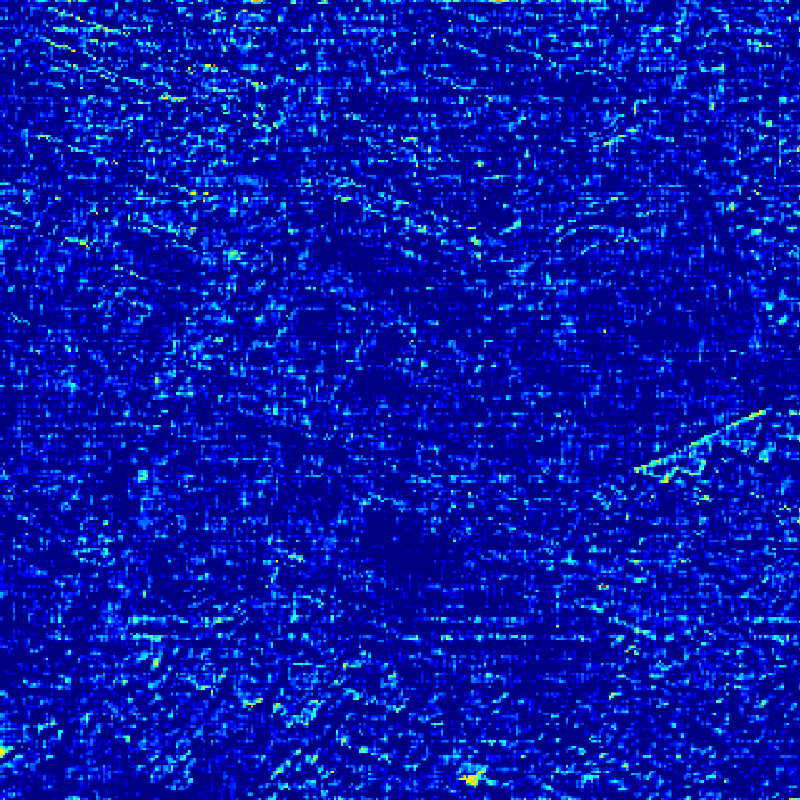}\vspace{1pt}
    \end{minipage}
}
\subfigure[]{
    \begin{minipage}[b]{0.1\linewidth}
    \includegraphics[width=2cm]{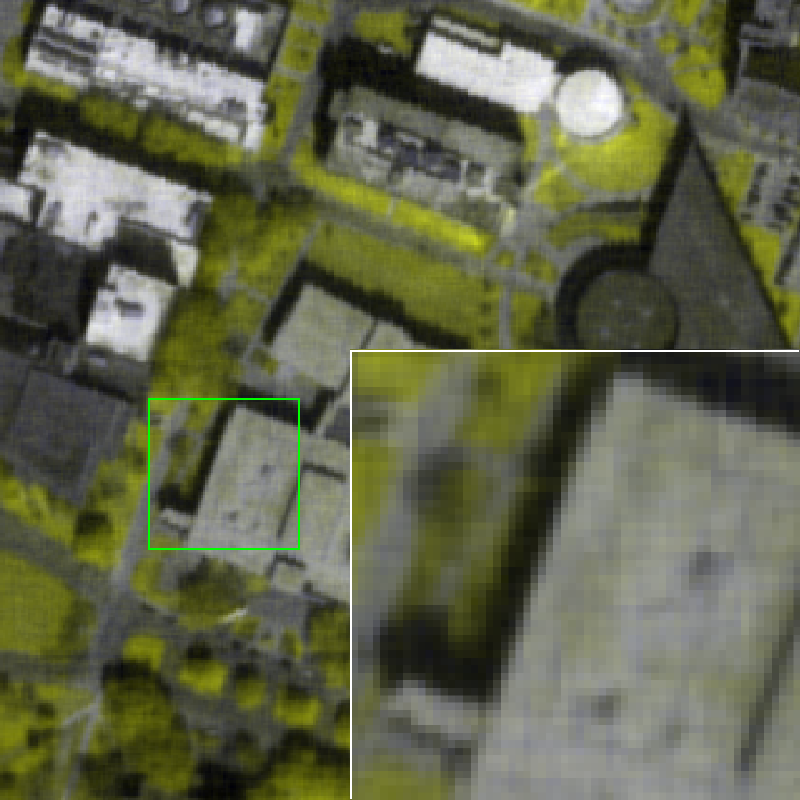}\vspace{1pt}
    \includegraphics[width=2cm]{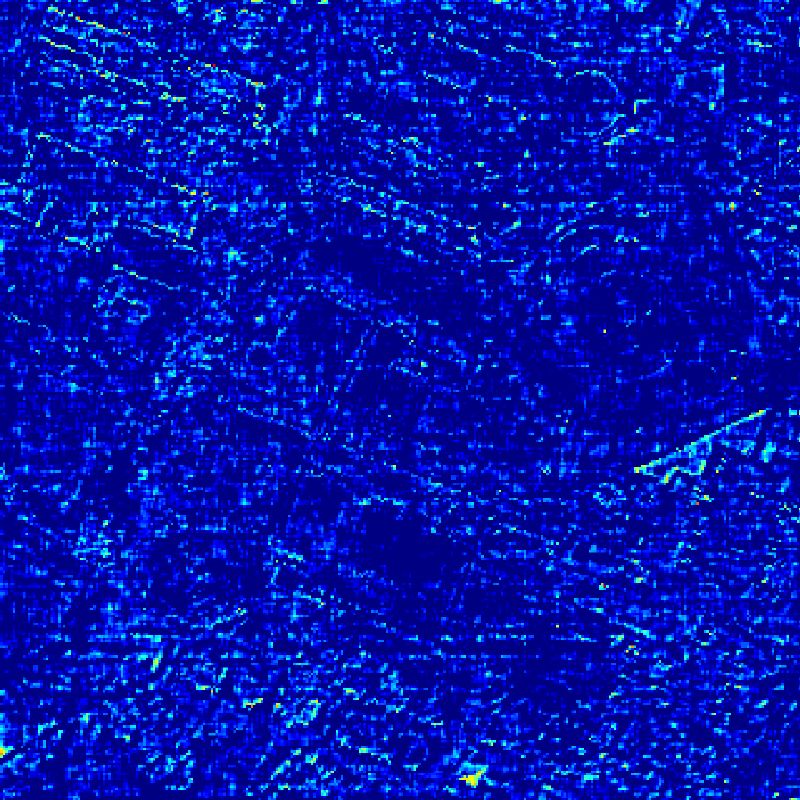}\vspace{1pt}
    \end{minipage}
}
\subfigure[]{
    \begin{minipage}[b]{0.1\linewidth}
    \includegraphics[width=2cm]{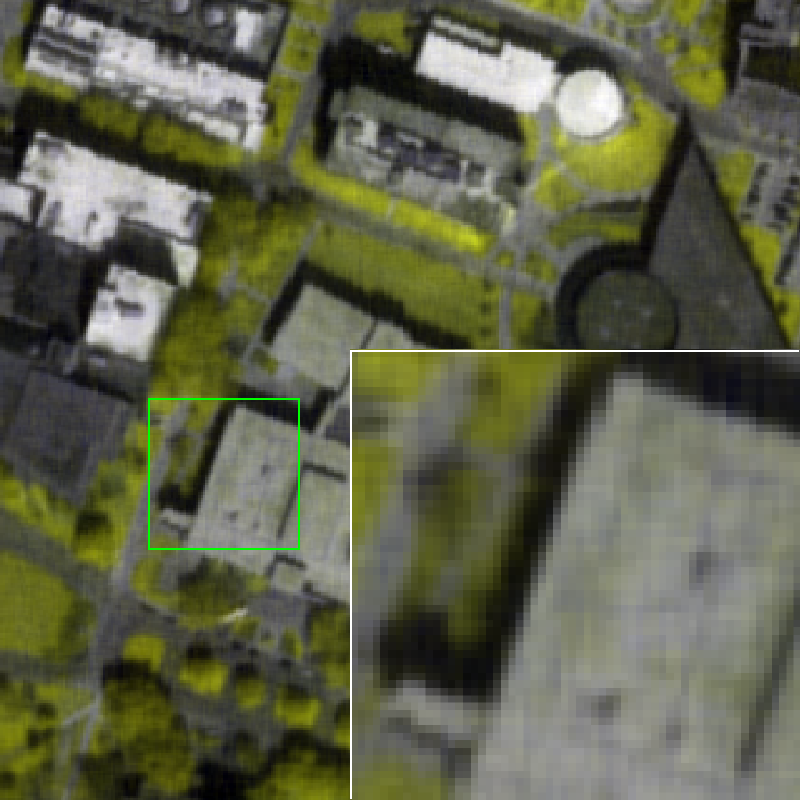}\vspace{1pt}
    \includegraphics[width=2cm]{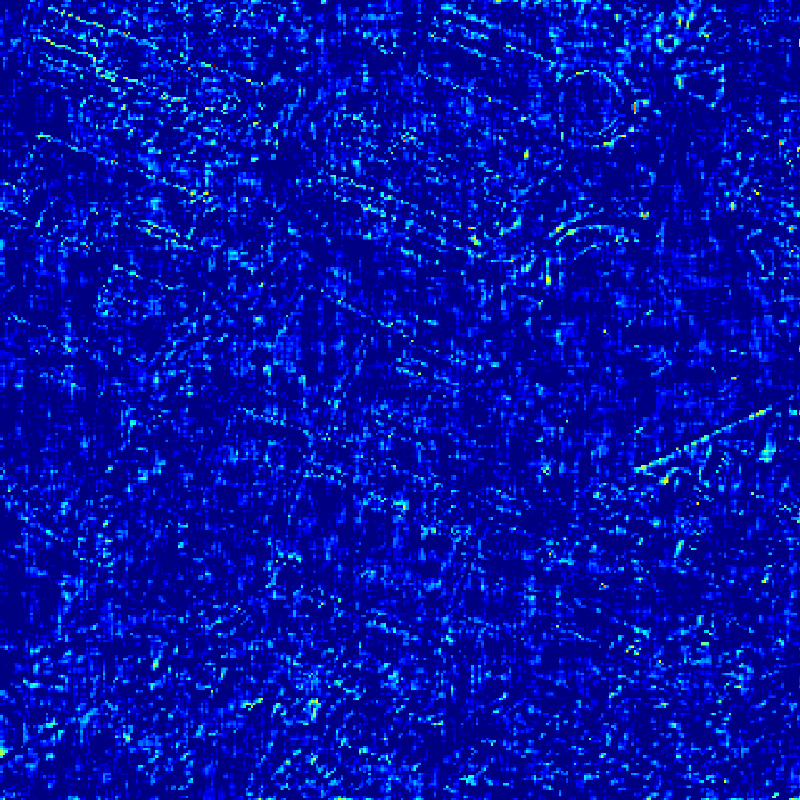}\vspace{1pt}
    \end{minipage}
}
\subfigure[]{
    \begin{minipage}[b]{0.1\linewidth}
    \includegraphics[width=2cm]{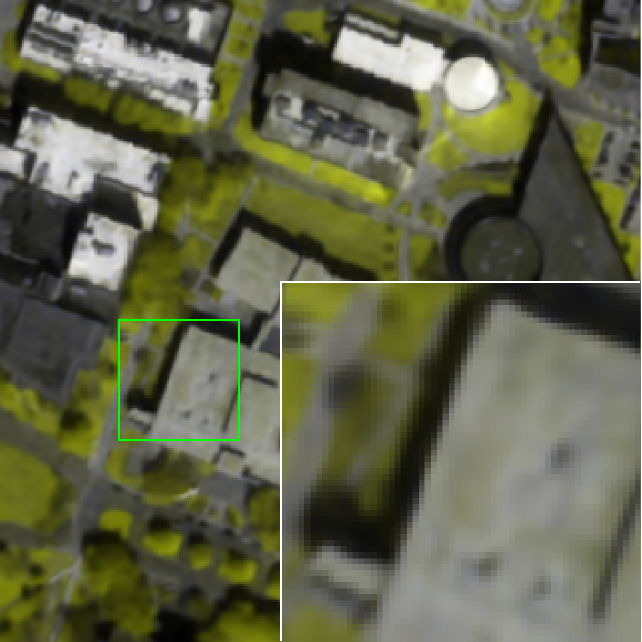}\vspace{1pt}
    \includegraphics[width=2cm]{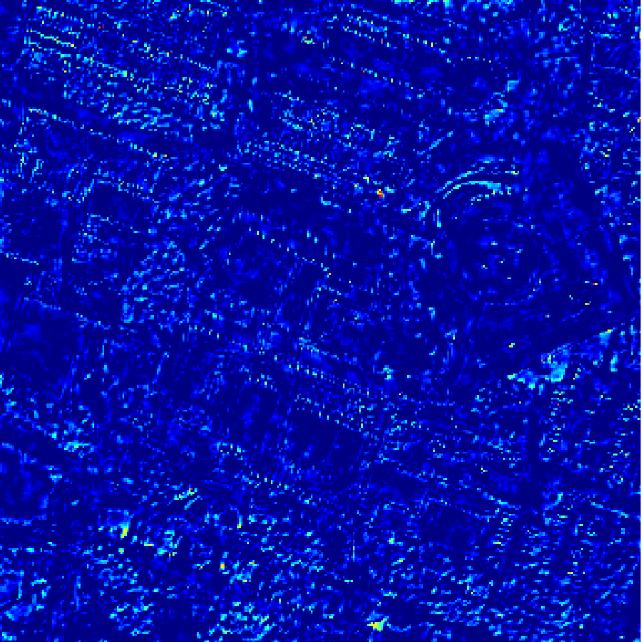}\vspace{1pt}
    \end{minipage}
}
\subfigure[]{
    \begin{minipage}[b]{0.1\linewidth}
    \includegraphics[width=2cm]{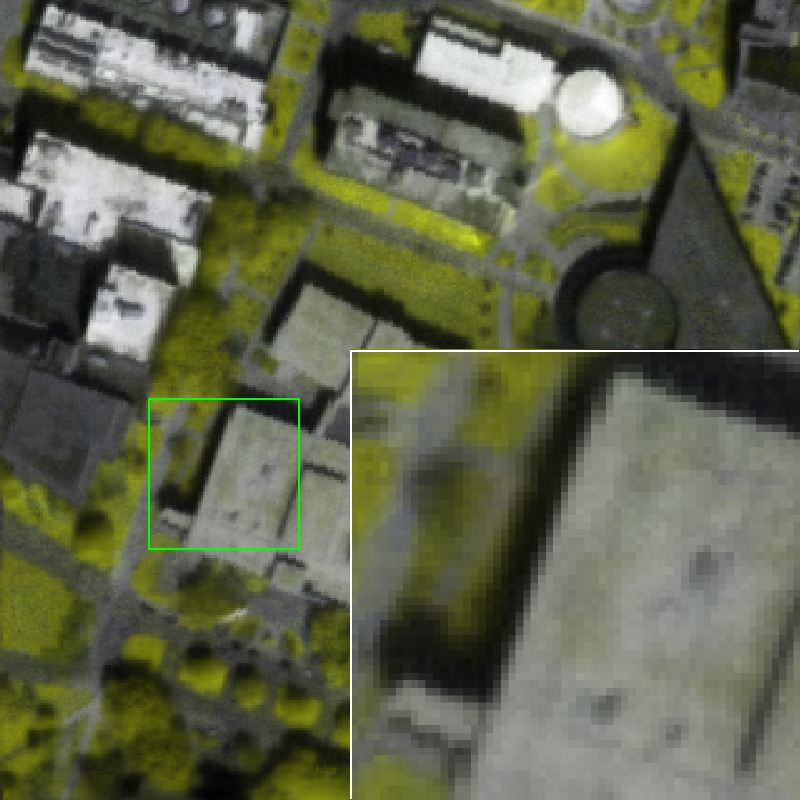}\vspace{1pt}
    \includegraphics[width=2cm]{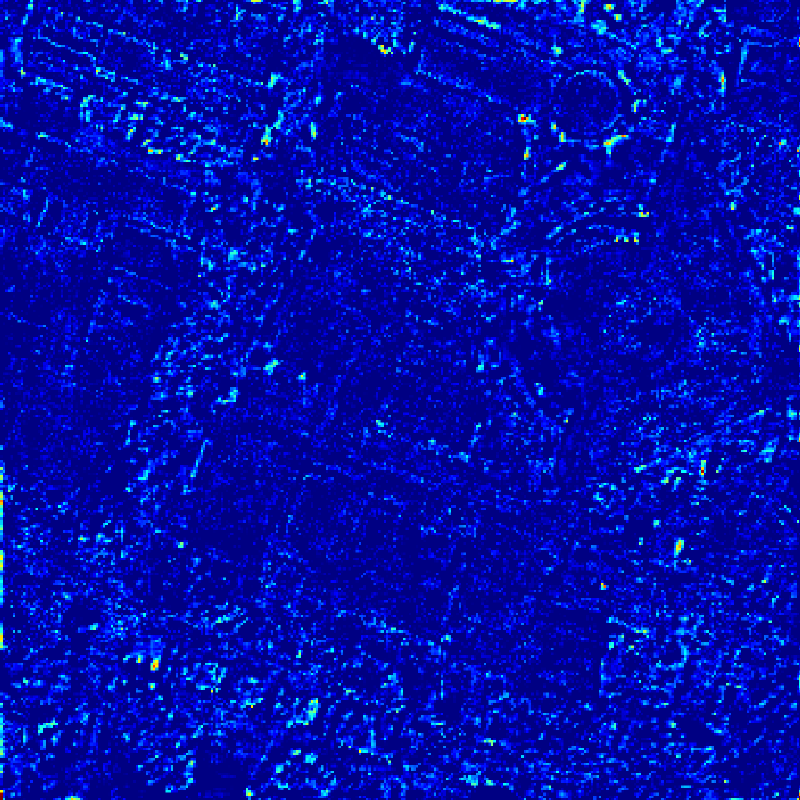}\vspace{1pt}
    \end{minipage}
}
\subfigure[]{
    \begin{minipage}[b]{0.1\linewidth}
    \includegraphics[width=2cm]{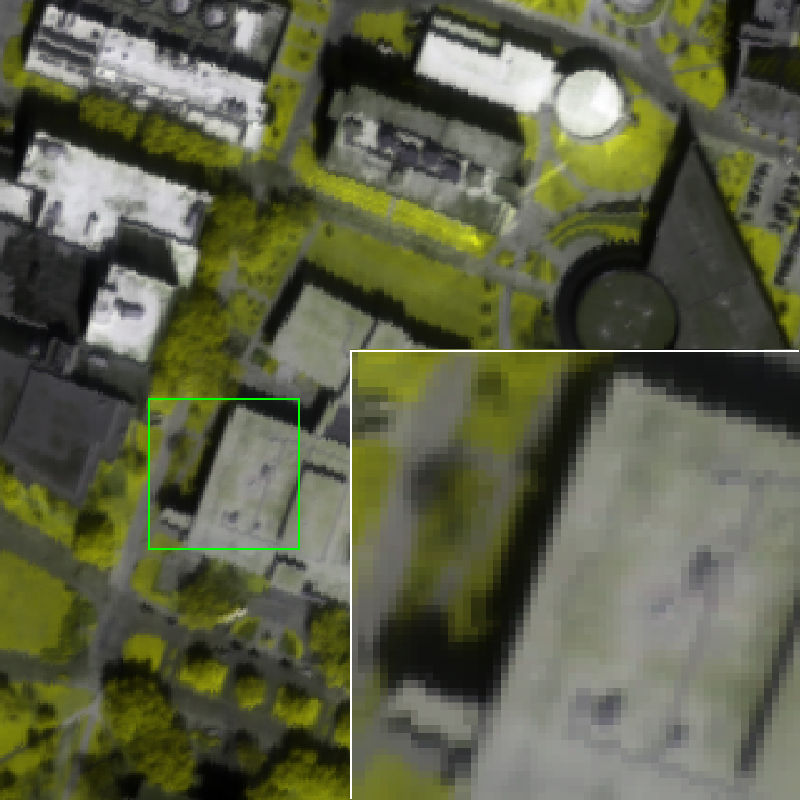}\vspace{1pt}
    \includegraphics[width=2cm]{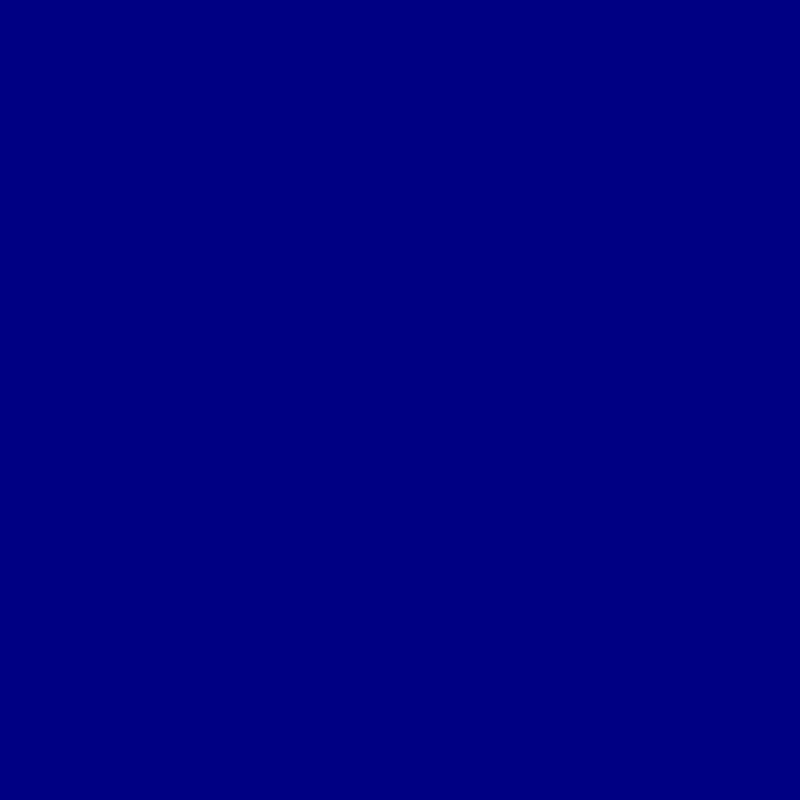}\vspace{1pt}
    \end{minipage}
}\vspace{-0.6em}
   \includegraphics[width=12cm]{Color_bar_0.1.pdf}\vspace{-1em}
\caption{The first row lists the false color images of ``University of Houston" generated by bands (R: 29, G: 32, and B: 21), and the second row displays the error images of the 27th band. (a) CSTF, (b) UTV, (c) CTRF, (d) FSTRD, (e) LogLRTR, (f) CLoRF, (g) JLRST, (h) ground truth.}
\label{figure6}
\end{figure*}

\subsection{Ablation Experiments}
To assess the individual contributions of the regularization terms in JLRST, we design comprehensive ablation experiments. Specifically, we perform three independent experiments and sequentially set $\alpha_1=0$, $\alpha_2=0$, and $\alpha_3=0$. Table~\ref{table3} presents the quantitative evaluation results for both the ``Balloons" and ``University of Houston" images under these conditions. The experimental results clearly demonstrate that the first and second regularization terms exert significant influence on model performance.

\begin{table*}[htbp]\vspace{-1em}
\caption{Comparison of quantitative indices from ablation experiments of regularization terms.}
\centering
\resizebox{\textwidth}{!}{
\newcommand{\rb}[1]{\raisebox{1.0ex}[0pt]{#1}}
\begin{tabular}{llccccccccccc}
\hline
{\rule[-1mm]{0mm}{3.5mm}}&\hbox{Image}&\multicolumn{5}{c}{Balloons}&\multicolumn{5}{c}{University of Houston}   \\
\cmidrule(r){3-7} \cmidrule(r){8-12}
{\rule[-1mm]{0mm}{3.5mm}}&Method&PSNR&SSIM&ERGAS&SAM&UIQI&PSNR&SSIM&ERGAS&SAM&UIQI\\
\hline
{\rule[-1mm]{0mm}{3.5mm}}&JLRST($\alpha_1=0$) &43.732 &0.984 &1.144 &4.000 &0.901 &38.796 &0.951 &1.525&2.715 &0.979                \\
{\rule[-1mm]{0mm}{3.5mm}}&JLRST($\alpha_2=0$)&43.928 &0.983 &1.121&4.294 &0.895&38.911 &0.853 &1.493 &2.630 &0.980        \\
{\rule[-1mm]{0mm}{3.5mm}}&JLRST($\alpha_3=0$)&44.373 &0.985&1.061 &4.311 &0.897 &39.201 &0.960&1.421 &2.562&0.982        \\
{\rule[-1mm]{0mm}{3.5mm}}&JLRST&\textbf{44.778} &\textbf{0.988} &\textbf{1.010} &\textbf{3.639} &\textbf{0.916}&\textbf{39.585} &\textbf{0.961} &\textbf{1.355}&\textbf{2.306} &\textbf{0.984}
\\
\hline
\end{tabular}}\vspace{-0.5em}
\label{table3}
\end{table*}

\subsection{Convergence Analysis}
In this subsection, we analyze the convergence of the proposed approach by performing comprehensive tests on the four images. Fig.~\ref{figure7} shows the curves of relative error and PSNR versus the number of iterations. As shown in Fig.~\ref{figure7}(a), the relative error decreases rapidly with the increase of iterations and reveals a consistent trend of converging toward zero.
Furthermore, Fig.~\ref{figure7}(b) demonstrates that the PSNR values exhibit rapid improvement with increasing iterations, ultimately reaching a stable state. These pieces of evidence further confirm the convergence of our proposed algorithm.

\begin{figure}[h!]
\centering\vspace{-1em}
\subfigure[]{
\includegraphics[width=3in]{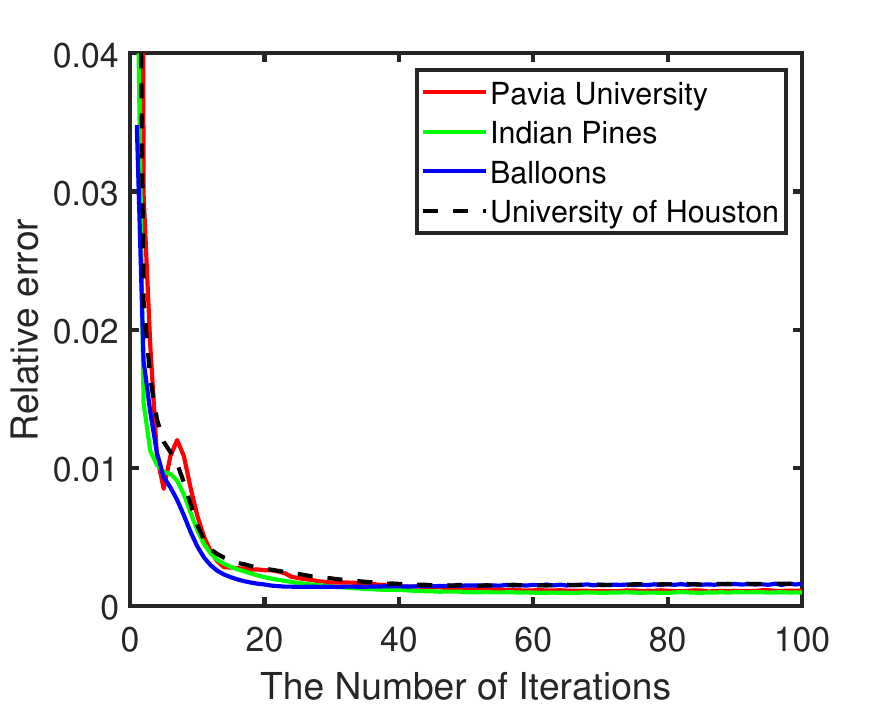}}
\subfigure[]{
\includegraphics[width=3in]{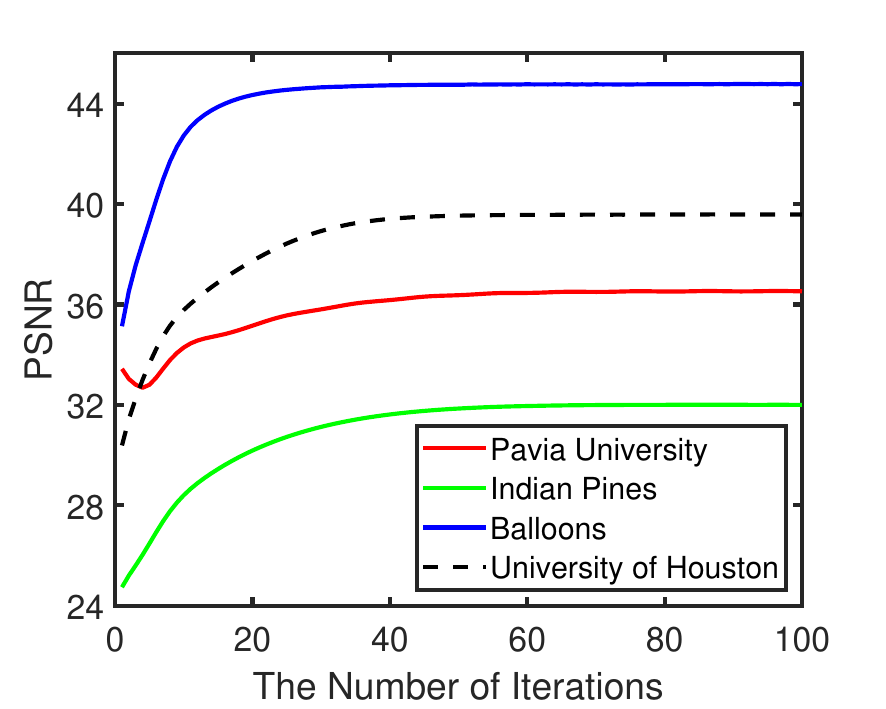}}
\caption{Relative error $\left\|\mathcal{Z}^{k}-\mathcal{Z}^{k-1}\right\|_{F}/\left\|\mathcal{Z}^{k}\right\|_{F}$ (left) and PSNR (right)
trends over iterations for four images.}
\label{figure7}
\end{figure}

\subsection{Running Time}
To evaluate the computational efficiency, Table~\ref{table4} summarizes the running time (in seconds) required by different testing approaches across the four datasets. These methods, namely CSTF, UTV, CTRF, FSTRD, LogLRTR and JLRST, were implemented on a desktop computer with a 3.20 GHz 12th Gen Intel(R) Core(TM) i9-12900K CPU and 64.0 GB RAM via MATLAB (R2023b), while CLoRF was evaluated on an RTX 5090D GPU with 64 GB RAM.
The deep learning method CLoRF relies on GPU acceleration, making its running time not directly comparable with CPU-based approaches. Therefore, it is excluded from this quantitative comparison. Among the CPU-based methods, CSTF consistently demonstrates the fastest computational speed, yet its fusion performance remains inferior to other approaches. While both FSTRD and LogLRTR yield stable fusion results, they come with significantly higher computational cost. In contrast, our proposed method not only achieves optimal fusion performance but also ensures computational efficiency.

\begin{table}[htbp]\vspace{-1em}
\caption{Running time for all testing methods.}
\centering
\scalebox{0.8}{\large{
\begin{tabular}{lcccc}
\hline
{\rule[-1mm]{0mm}{3.5mm}}Image &Pavia University  &Indian Pines &Balloons& University of Houston    \\
\hline
{\rule[-1mm]{0mm}{3.5mm}}CSTF  &70&\textbf{7}&\textbf{24}&25 \\
{\rule[-1mm]{0mm}{3.5mm}}UTV  &103&17&41&\textbf{22} \\
{\rule[-1mm]{0mm}{3.5mm}}CTRF &\textbf{56}&40&80&52\\
{\rule[-1mm]{0mm}{3.5mm}}FSTRD  &519&372 &778&515 \\
{\rule[-1mm]{0mm}{3.5mm}}LogLRTR    &491&252 &693&460  \\
{\rule[-1mm]{0mm}{3.5mm}}JLRST  &161&58&226&98 \\
\hline
\end{tabular}}}\vspace{-1em}
\label{table4}
\end{table}

\section{Conclusion}\label{section5}
This paper presented a novel subspace-based fusion approach via joint low-rank and smooth tensor regularization.
Our method effectively exploited spectral correlations by employing SVD to extract the spectral subspace from the observed LR-HSI. Then, building upon the learned clustering structure from the HR-MSI, clustering of similar patches in the coefficient tensor was established. We subsequently imposed a JLRST regularization on the collected three-order tensors to simultaneously exploit both global low-rankness and local smoothness priors inherent in the data. Moreover, we developed an efficient ADMM to solve the proposed model and analyzed its convergence theoretically. Extensive experiments conducted on four datasets show that the proposed method achieved superior fusion performance over six competing approaches.

\section*{Data Availability}
 The data used to support the findings of this study are
 available from the corresponding author upon request.

\section*{Competing Interests}
The authors declare that they have no competing interests.

\section*{Acknowledgments}
This work was partly supported by the National Natural Science Foundation of China (62461043, 12201286), the Jiangxi Provincial Natural Science Foundation (20242BAB22013, 20252BAC240172, 20252BAC250011), the Shenzhen Science and Technology Program (20231115165836001), Guangdong Basic and Applied Research Foundation (2024A1515012347).

\setcounter{equation}{0}

\end{document}